\documentclass[11pt]{article}
\usepackage[pdfpagelabels]{hyperref} \hypersetup{colorlinks=true, linktoc=all, linkcolor=black, citecolor=black }
\usepackage{algorithm, algorithmicx, algpseudocode}
\usepackage{pgfplots} \usepackage[toc,page]{appendix}
\usepackage{upgreek} \usepackage{fullpage}
\usepackage[english]{babel} \usepackage[latin1]{inputenc}
\usepackage{amsmath, amsfonts, amsthm, amssymb}
\usepackage{xcolor}
\usepackage{graphicx} \usepackage{psfrag}
\usepackage{tikz} \usetikzlibrary{arrows, calc, decorations, decorations.pathreplacing}
\usepackage{ifthen, xifthen} \usepackage{dsfont}
\usepackage{bm} \usepackage{accents}
\usepackage{xparse}
\usepackage{aligned-overset}
\usepackage{subcaption}
\usepackage{csquotes}
\usepackage{adjustbox}
\usepackage{cleveref}
\usepackage[normalem]{ulem}
\usepackage{listings}  
\usepackage{enumitem}  
\usepackage{xspace} \newcommand{\enmath}[1]{\ensuremath{#1}\xspace}
\usepackage[colorinlistoftodos]{todonotes}

\newtheorem{theorem}{Theorem}[section]

\newtheorem{corollary}[theorem]{Corollary}

\newtheorem{example}[theorem]{Example}
\newtheorem{lemma}[theorem]{Lemma}
\newtheorem{proposition}[theorem]{Proposition}
\newtheorem{remark}[theorem]{Remark}
\newtheorem{assumption}[theorem]{Assumption}

\DeclareMathOperator{\argmin}{argmin}
\DeclareMathOperator{\lip}{Lip}

\newcommand{\alg}[1]{{\rm Alg#1}}
\newcommand{\cost}[1]{{\rm cost}(\alg{#1})}
\newcommand{\norm}[2][]{\| #2 \|_{#1}}
\newcommand{\snorm}[2][]{| #2 |_{#1}}
\newcommand{\normc}[2][]{\left\| #2 \right\|_{#1}}

\newcommand{\abs}[1]{\left| #1 \right|}
\newcommand{\set}[2]{\{#1\,:\,#2\}}
\newcommand{\setc}[2]{\left\{#1\, :\,#2\right\}}
\newcommand{\dfn}{\vcentcolon=} \newcommand{\dfnn}{=\vcentcolon}

\newcommand{\cA}{{\mathcal A}} 
\newcommand{\cC}{{\mathcal C}} 
\newcommand{\cE}{{\mathcal E}} \newcommand{\cF}{{\mathcal F}}
\newcommand{\cO}{{\mathcal O}} \newcommand{\cN}{{\mathcal N}}
\newcommand{\cM}{{\mathcal M}} 
\newcommand{\cT}{{\mathcal T}}

\newcommand{\C}{{\mathbb C}} \newcommand{\N}{{\mathbb N}}
\newcommand{\R}{{\mathbb R}}
\newcommand{\bbP}{{\mathbb P}}

\newcommand{\bsc}{{\boldsymbol c}} \newcommand{\bsd}{{\boldsymbol d}}
\newcommand{\bse}{{\boldsymbol e}}
\newcommand{\bsk}{{\boldsymbol k}} \newcommand{\bsm}{{\boldsymbol m}}
\newcommand{\bsx}{{\boldsymbol x}} \newcommand{\bsy}{{\boldsymbol y}}
\newcommand{\bsG}{{\boldsymbol G}} \newcommand{\bsI}{{\boldsymbol I}}

\newcommand{\bschi}{{\bm \chi}} \newcommand{\bsrho}{{\bm \rho}}
\newcommand{\bsnu}{{\bm \nu}}   \newcommand{\bsmu}{{\bm \mu}}
\newcommand{\bseta}{{\bm \eta}} \newcommand{\bskappa}{{\bm \kappa}}
\newcommand{\bsgamma}{{\bm \gamma}}

\newcommand{\subsetq}{\subseteq }
\newcommand{\pluseq}{\mathrel{+}=}

\newcommand{\B}[1]{{\mathcal{B}(#1)}}
\newcommand{\Rp}{\R_{>0}} \newcommand{\Rpz}{\R_{\ge 0}}

\newcommand{\from}[2]{#1_{{[#2:]}}}
\newcommand{\upto}[2]{#1_{{[:#2]}}}

\newcommand{\measi}{\enmath{{\pi}}}
\newcommand{\measii}{\enmath{{\rho}}}
\newcommand{\mtar}{\enmath{{\pi}}}

\DeclareMathOperator{\dist}{dist}
\newcommand{\hellinger}{\dist_{\rm H}}
\newcommand{\wasserstein}[1]{\dist_{{\rm W}_{#1}}}
\newcommand{\totalvar}{\dist_{\rm TV}}

\newcommand{\ftar}{\enmath{f_\mtar}}
\newcommand{\ftars}{\enmath{\sqrt{\ftar}}}
\newcommand{\ftaru}{\enmath{\hat{f}_\mtar}}
\newcommand{\ftarus}{\enmath{\sqrt{\ftaru}}}
\newcommand{\ftaruss}{\enmath{\ftaru^{1/2}}}

\newcommand{\ui}{[0,1]}
\newcommand{\uid}{[0,1]^d}
\newcommand{\uij}[1]{[0,1]^{#1}}

\newcommand{\iui}{\int_0^1}
\newcommand{\iuiu}[1]{\int_0^{#1}}

\newcommand{\ckalpha}{\enmath{{C^{k,\alpha}(\uid)}}}

\newcommand{\ckmix}{\enmath{{C^{k,{\rm mix}}(\uid)}}}

\newcommand{\dd}{\;\mathrm{d}}
\newcommand{\ddd}{\mathrm{d}}

\newcommand{\Balpha}{{\boldsymbol \alpha}}
\newcommand{\mix}{{\rm mix}}

\numberwithin{equation}{section}

\begin{document}
\title{Measure transport via polynomial density surrogates}
\date{\today}
\author{
  Josephine Westermann\thanks{E-mail: josephine.westermann@uni-heidelberg.de}\qquad\qquad\qquad
  Jakob Zech\thanks{E-mail: jakob.zech@uni-heidelberg.de}
  \\[1ex]
  \textit{IWR, Universit\"at Heidelberg, 69120 Heidelberg, Germany}
}

\maketitle

\numberwithin{equation}{section}

\begin{abstract}
  We discuss an algorithm to compute transport maps that couple the uniform measure on $\uid$ with a
  specified target distribution $\pi$ on $\uid$. The primary objectives are either to sample from or to
  compute expectations w.r.t.~$\pi$. The method is based on leveraging a polynomial surrogate of the target
  density, which is obtained by a least-squares or interpolation approximation. We discuss the design and
  construction of suitable sparse approximation spaces, and provide a complete error and cost analysis for
  target densities belonging to certain smoothness classes. Further, we explore the relation between our
  proposed algorithm and related approaches that aim to find suitable transports via optimization over a class
  of parametrized transports. Finally, we discuss the efficient implementation of our algorithm and report on
  numerical experiments which confirm our theory.
\end{abstract}



\section{Introduction}

Sampling from intricate and potentially high dimensional target distributions $\pi$ is a central challenge of
many problems arising in modern applied mathematics and statistics. Notable applications include in particular
generative sampling, inverse problems and data assimilation. A methodology that has emerged in recent years in
different variants as one of the primary tools to tackle such problems, is the (approximate) transformation of
a ``latent'' or ``reference'' distribution $\rho$ into the target via a transport map $T$. In this paper we
concentrate on distributions $\rho$ and $\pi$ on the $d$-dimensional unit cube $\uid$, with $\pi$ being known
through an unnormalized version of its density. This is a typical setup encountered in (PDE driven) Bayesian
inverse problems~\cite{stuartacta}.

The primary aim is to determine a map $T:\uid\to\uid$ such that $X\sim \rho$ implies $T(X)\sim \pi$. In other
words, $T$ pushes forward $\rho$ to $\pi$, which is denoted by $T_\sharp\rho=\pi$. For reasons that will
become apparent later, throughout we concentrate on $\rho$ being the uniform measure. Once $T$ is computed,
sampling from $\pi$ becomes straightforward by mapping a uniform sample on $\uid$ under $T$. Expectations
w.r.t.~$\pi$ can be computed by
  \begin{equation*}
    \mathbb{E}_\pi[F]
    = \int_{\uid} F(\bsx) \dd \pi(\bsx) = \int_{\uid} F(T(\bsx)) \dd \bsx.
  \end{equation*}
In practice, this integral may be approximated using various techniques such as uniform Monte-Carlo samples on
$\uid$, higher-order quasi-Monte Carlo quadrature~\cite{MR2346374}, or tensorized and sparse-grid
quadrature~\cite{MR1669959,MR4113052}, among others. Such random or deterministic numerical integration
methods are widely available and well-understood for the uniform measure on $\uid$ even if $d$ is large. This
is in general not the case for integration w.r.t.~some arbitrary $\pi$.

The key difficulty is to come up with stable algorithms capable of finding such $T$, or at least finding $T$
for which $T_\sharp\rho$ approximates $\pi$ in a suitable sense. There exists a large variety of methods that
view the problem as a learning task, in which $T$ is obtained as the minimizer of an objective measuring a
distance between $T_\sharp\rho$ and $\pi$ in a suitable set $\cT$ of transports. That is, these algorithms try
to solve the optimization problem
  \begin{equation} \label{eq:opt}
    \argmin_{T\in \cT}{\rm dist}(\pi, T_\sharp\rho)
  \end{equation}
for some notion of distance ``$\dist$'', which is typically chosen as the KL-divergence. One such instance are
approaches based on a basis representation of $\cT$ via polynomials, wavelets or other function
systems~\cite{ELMOSELHY20127815, marzouk2016introduction, pmlr-v97-jaini19a, 2009.10303, zech2022sparse}. In
these papers, the authors parametrize the set $\cT$ using e.g.~multivariate polynomials, and the goal becomes
to determine the polynomial's coefficients. A main difficulty to overcome is that such a parametrization will
in general not yield a bijective map $T:\uid\to\uid$, which is a strongly desirable property since it is often
more natural to first determine $T^{-1}$ and then invert this map~\cite{marzouk2016introduction}. This is
typically remedied by introducing a smart nonlinear reparametrization which enforces bijectivity, or to
promote bijectivity through suitable regularizers in the objective. A downside of this procedure is that such
reparametrizations may result in more challenging loss landscapes~\cite{2009.10303}.  Another important class
of methods that has received widespread attention in recent years and achieves state-of-the-art performance in
certain areas, is based on neural network parametrizations of the transport. We refer in particular to
normalizing flows and invertible neural networks~\cite{pmlr-v37-rezende15,NICE,RealNVP,1808.04730,9089305} and
neural ODEs~\cite{Haber_2018,NEURIPS2018_69386f6b}. Finally, we also mention general variational
inference~\cite{pmlr-v37-rezende15, blei2017variational} and particle methods such as Stein variational
gradient descent ~\cite{NIPS2016_b3ba8f1b} for which there exist close ties to~\eqref{eq:opt}.

Differences in these methods pertain in particular to the choice of $\cT$ and the algorithms used to
solve~\eqref{eq:opt}. However, ultimately all of them hinge on solving a highly non-convex optimization
problem. Rigorous error analyses need to account for approximation errors (related to the required
expressivity of $\cT$), statistical errors (related to a~--~typically Monte Carlo based~--~approximation of
the objective), and optimization errors (related to solving the minimization problem~\eqref{eq:opt}).
Especially the last one remains poorly understood in this context, and also in practice often relies on
heuristics and trial and error. As a result, rigorous and complete analyses of these algorithms and guaranteed
error convergence are hard to obtain, and are indeed in general not available in the literature. Yet, this is
a highly active field, and recent years have seen some progress in particular related to universality and
approximation~\cite{SupApproximation, zech2022sparse, zech2021sparse, baptista2023approximation} and
statistical errors~\cite{pmlr-v151-irons22a,2207.10231} for different frameworks. The main contribution of our
paper is to present and analyze a method, which allows us to account for all introduced sources of error, thus
yielding a full error analysis.

The most related works are~\cite{MR4065222, MR4509118, cui2023self}. In~\cite{MR4065222}, the authors proposed
a novel approach, where instead of solving~\eqref{eq:opt}, they first approximate the target density using
low-rank tensor-train surrogates. Having such a density surrogate at hand allows for the construction of the
Knothe-Rosenblatt transport, which is a particular transport map that is triangular and bijective. This idea
was further refined in~\cite{MR4509118}, where a sequence of such maps is concatenated to further increase
expressivity and aid the ``training'' phase of the algorithm by using a tempering strategy. We emphasize
however, that the use of low-rank tensor-train approximations also necessitates solving a non-convex
optimization problem. In~\cite{cui2023self}, the authors replace the tensor surrogate by a polynomial
approximation, which yields a convex objective and is the same route we pursue in the present paper. The
algorithm proposed in~\cite{cui2023self} is largely equivalent to our method, and we point out that this work
has been carried out independently and at the same time. Nonetheless, our works are different, in
that~\cite{cui2023self} primarily focuses on algorithmic development and the experimental exploration of
computational strategies such as adaptivity and concatenation of transports as introduced in~\cite{MR4509118}.
In contrast, the focus of our manuscript is predominantly on the numerical analysis, and the establishment of
error bounds and convergence rates for certain smoothness classes of target densities. We next describe in
more detail the main contributions of our work.

\paragraph{Contributions}
We propose a method to construct exact transport maps from approximate polynomial surrogates of the target
density. We discuss two variants based on either weighted least-squares or polynomial interpolation to obtain
the surrogate. Our algorithm is independent of unknown normalization constants of the density, and is thus in
particular suitable for use in Bayesian inverse problems.

We provide a \emph{complete error analysis} (taking into account approximation, optimization and
generalization) of the method in terms of the target approximation w.r.t.~the Hellinger and Wasserstein
distances. This is possible, since, contrary to other works, our reinterpretation of the problem allows
bypassing nonconvex optimization in the learning phase, and instead relies either on the solution of a convex
optimization problem (for least-squares) or even deterministic collocation (for interpolation). We obtain
\emph{deterministic (rather than stochastic) error bounds} for our method, irrespective of whether
interpolation or our specific design of the weighted least-squares method are used for the density
reconstruction task. Moreover, we demonstrate that our method can still be interpreted as a variant
of~\eqref{eq:opt}. Consequently, our formulation allows us to (essentially)  \emph{solve this global
optimization problem} for a specific choice of $\cT$ and $\dist$.

Our theory establishes \emph{convergence rates} for target densities in \emph{classical smoothness spaces}
such as \ckalpha, \ckmix and the analytic functions. For \ckmix we also give a convergence proof of a
sparse-grid polynomial interpolation that shows how the curse of dimensionality can be lessened in this case,
that we could not find in the literature. In each case, we construct a-priori ansatz spaces, that provably
yield certain \emph{algebraic or exponential} (for analytic targets) convergence rates. The error is bounded
in terms of the number of required evaluations of the target density. For Bayesian inverse problems, this can
be interpreted as a measure of the computational complexity, since each evaluation of the posterior requires
evaluating the forward model, which is typically the computationally most expensive part.

In practice, sampling from the target requires evaluating the computed transport map $T$, which in turn is
constructed from the sparse polynomial surrogate of the target density. This computation is in general
non-trivial, and we discuss \emph{efficient implementation} strategies, based on orthonormal representations
of the surrogate. Additionally, we provide a \emph{cost analysis} and bound the number of floating point
operations required for one evaluation of $T$, which amounts to the complexity of computing one (approximate)
sample of the target.

Another key distinction to earlier theoretical works on this topic~\cite{zech2021sparse, 2207.10231,
baptista2023approximation}, is that our analysis does \emph{not require the target density to be uniformly
positive}, and thus allow us to treat \emph{multimodal distributions} (with properly separated modes) without
any negative impact on the error bounds. We explain this in more detail in \Cref{sec:multimodal} ahead.

Finally, we conduct some simple \emph{numerical tests}, that confirm our theoretical considerations.

\paragraph{Outline}

This work is organized as follows. After introducing some basic notation and conventions, in \Cref{sec:main}
we first review necessary background, then describe our algorithm and present a high-level version of our main
findings for the least-squares version of the method. \Cref{sec:main_proofs} provides proofs, further in-depth
discussions, and extensions of these results. In \Cref{sec:implementation_and_sampling_cost}, we describe in
detail how the evaluation of a transport map constructed from polynomial surrogates can be implemented
efficiently and provide bounds on the cost of required floating point operations in dependence of the size of
the ansatz space. \Cref{sec:interpolation} presents the interpolation based version of the algorithm, and to
avoid repetitiveness we focus mainly on the differences to the least-squares based variant and provide most
technical proofs in the appendix. Finally, we apply our method to two exemplary applications in
\Cref{sec:num_exp}.

\paragraph{Notation}

We denote $\N \dfn \{1, 2, \dots\}$, $\N_0 \dfn \{0\} \cup \N$, $\Rpz \dfn [0,\infty)$ and $\Rp \dfn
(0,\infty)$.

For $d\in \N$ and vectors $\bsx \in \R^d$ we use the notation $\upto{\bsx}{j}$ to refer to its first $j \in
\{1, \dots, d\}$ components, i.e.~$\upto{\bsx}{j} \dfn (x_1, \dots, x_j)$, and similarly $\from{\bsx}{j} \dfn
(x_j, \dots, x_d)$ for its last $d-j+1$ components. We also employ this notation for multi-indices $\bsnu
\subset \N^d_0$ and $d$-dimensional functions $F : \R^d \to R^d$.

We denote with $\B{\Omega}$ the Borel $\sigma$-algebra over a domain $\Omega \subset \R^d$.

For $\Omega_1, \Omega_2 \subset \R^d$ and a measurable Function $F:\Omega_1 \to \Omega_2$ we denote with
$F_\sharp \rho_1$ the pushforward measure of a measure $\rho_1$ on $\Omega_1$. It is a measure on $\Omega_2$
and defined via
$F_\sharp \rho_1(S_2) = \rho_1 (F^{-1}(S_2))$ for all $S_2 \in \B{\Omega_2}$. Analogously, the pullback
measure $F^\sharp \rho_2$ on $\Omega_1$ of a measure $\rho_2$ on $\Omega_2$ is defined via $F^\sharp
\rho_2(S_1) = \rho_2 (F(S_1))$ for all $S_1 \in \B{\Omega_1}$.

For $p\in [1,\infty]$ we denote by $L^p(\Omega_1,\rho_1;\Omega_2)$ the usual $L^p$-space
of functions $f:\Omega_1\to\Omega_2$. If $\Omega_2=\R$ we omit the argument $\Omega_2$, and in case
$\rho_1$ is the Lebesgue measure we omit the argument $\rho_1$.

We denote with $\lambda$ the Lebesgue measure, and write $\mu$ for the uniform probability measure
$\lambda|_{[0,1]^d}$ on $\uid$. For a probability measure $\pi$ on $\uid$ that is absolutely continuous
w.r.t.~$\lambda$, we write $\ftar : \uid \to \Rpz$ to denote its Lebesgue density. We use the generic notation
$\ftaru$ to denote a (possibly) unnormalized density of $\pi$, i.e.~given the probability measure $\pi$,
$\ftaru\in L^1(\uid;\Rpz)$ is a function that with $c_\pi\dfn\int_{\uid}\ftar(\bsx)\dd\bsx>0$ satisfies
$\ftar c_\pi\dfn\ftaru$.

\section{Main ideas and results}\label{sec:main}

In this section, we provide the motivation and derivation of the proposed algorithm, present our main
findings, and establish connections with existing methods. We begin by revisiting the Knothe-Rosenblatt map
construction in \Cref{sec:KR}. Next, we outline a high-level overview of our algorithm in \Cref{sec:alg}. The
weighted least-squares algorithm that we use to construct density surrogates is briefly summarized in
\Cref{sec:constrdens}. Lastly, we elucidate the relationship of our resulting algorithm with variational
inference algorithms, as expressed in Equation~\eqref{eq:opt}, and provide a convergence result in
\Cref{sec:mainresult}.

\subsection{Knothe-Rosenblatt (KR) map}\label{sec:KR}

Let $\rho$ and $\pi$ be two distributions on $[0,1]^d$. It is well-known, that under some mild conditions
there exists a unique map $T:\uid \to \uid $, referred to as the Knothe-Rosenblatt map, such that
$T_\sharp\rho=\pi$ and additionally
\begin{enumerate}
  \item $T$ is \emph{triangular}, i.e.~the $j$-th component $T_j:\uij{j}\to [0,1]$ of $T$ is a function of
    $\upto{\bsx}{j}$ and
  \item $T$ is \emph{monotone}, i.e.~$x_j\mapsto T_j(\upto{\bsx}{j-1}, x_j)$ is strictly monotonically
    increasing for every fixed $\upto{\bsx}{j-1}\in \uij{j-1}$.
\end{enumerate}
See for example~\cite[Section~2.3]{santambrogio} or~\cite{bogachev} for discussions on existence and
uniqueness of such $T$.

In case that (i)~$\rho=\mu$ (i.e.~the reference $\rho$ is the uniform measure with Lebesgue density $1$ on
$\uid$) and (ii)~$T$ is a $C^1$-diffeomorphism, the inverse KR map $S \dfn T^{-1}$ has a particularly simple
construction, which is based on two observations:
\begin{itemize}
  \item For any $C^1$-diffeomorphism $T:\uid \to \uid $ with inverse $S$, the pushforward density satisfies
    \begin{equation*}
      \frac{\ddd T_\sharp\mu}{\ddd\lambda}(\bsx) = \frac{\dd\mu}{\dd\lambda}(T^{-1}(\bsx)) \det dT^{-1}(\bsx)
      = \det dS(\bsx) \qquad \text{a.e.\ in } \uid,
    \end{equation*}
    and hence it holds that
    \begin{equation}\label{eq:equivalence}
      T_\sharp\mu = \pi \qquad \Leftrightarrow \qquad
      \frac{\ddd \pi}{\ddd\lambda}(\bsx) = \det dS(\bsx) \quad \text{a.e.\ in } \uid.
    \end{equation}
  \item If $T$ is additionally monotone and triangular, then so is its inverse $S$. In that case, the right
      equality in~\eqref{eq:equivalence} is equivalent to
    \begin{equation} \label{eq:toachieve} \prod_{j=1}^d \partial_j
      S_j(\upto{\bsx}{j}) = \frac{\ddd \pi}{\ddd\lambda}(\bsx) \qquad \text{a.e.\ in } \uid.
    \end{equation}
\end{itemize}
Therefore, in order to determine $T$ it suffices to find a monotone triangular $C^1$-diffeomorphism
$S:\uid\to\uid$ satisfying~\eqref{eq:toachieve}, and then set $T \dfn S^{-1}$.

\begin{remark}
  We stress that~\eqref{eq:equivalence} hinges on $\mu$ being the uniform measure. For a general reference
  measure $\rho$,~\eqref{eq:equivalence} reads
  \begin{equation*}
    T_\sharp\rho = \pi
    \quad \Leftrightarrow \quad
    \frac{\ddd \pi}{\ddd\lambda}(\bsx) = \frac{\ddd\rho}{\ddd\lambda}(S(\bsx)) \det dS(\bsx)
    \quad \Leftrightarrow \quad
    \frac{\ddd\pi}{\ddd\rho}(\bsx)
    = \frac{\ddd\lambda}{\ddd\rho}(\bsx)\frac{\ddd\rho}{\ddd\lambda}(S(\bsx)) \det dS(\bsx).
  \end{equation*}
\end{remark}

The final step to obtain a concrete algorithm, is to observe that $S$ in~\eqref{eq:toachieve} allows for an
explicit computation of its components. Specifically, as the next lemma shows, there exists a unique partition
of the density $\ftar$ into a $d$-fold product of functions depending only on the first $j$ components, namely
the conditional densities of $\pi$. By~\eqref{eq:toachieve}, these functions must coincide with
$\partial_j S_j(\upto{\bsx}{j})$.

\begin{lemma}\label{lma:decomposition}
  Let $d\in\N$. Let $f:\uid \to\Rp$ satisfy $\int_{\uid} f(\bsx) \dd\bsx = 1$. Then there exist unique
  functions $f_j : \uij{j} \to \Rp$, $j\in\{1,\dots,d\}$, such that
  \begin{enumerate}
    \item\label{item:intfj1}
      $\iui f_j(\upto{\bsx}{j-1}, x_j) \dd x_j=1$ for all
      $\upto{\bsx}{j-1} \in \uij{j-1}$ and all $j\in\{1,\dots,d\}$,
    \item\label{item:prodfj}
      $f(\bsx) = \prod_{j=1}^df_j(\upto{\bsx}{j})$ for all $\bsx\in \uid $.
  \end{enumerate}
  They are given by $f_{j}(\upto{\bsx}{j}) = \frac{\int_{\uij{d-j}}f(\bsx) \dd
  \from{\bsx}{j+1}}{\int_{\uij{d-j+1}}f(\bsx) \dd \from{\bsx}{j}}$.
\end{lemma}

Proven in \Cref{app:decomposition}, this lemma provides a method to calculate the factors $\partial_j S_j$
in~\eqref{eq:toachieve} as the marginal densities of $\ftar$ and, consequently, the components of $S$ as
  \begin{equation} \label{eq:decomposition}
    S_{j}(\upto{\bsx}{j})
    = \frac{\int_{\uij{d-j}} \int_0^{x_j} f(\upto{\bsx}{j-1}, t, \from{\bsx}{j+1}) \dd t \dd \from{\bsx}{j+1}}
    {\int_{\uij{d-j+1}}f(\upto{\bsx}{j-1}, \from{\bsx}{j}) \dd \from{\bsx}{j}}.
  \end{equation}
As will become apparent in \Cref{sec:alg}, our density surrogates may not be positive everywhere, thus not
satisfying the prerequisites of \Cref{lma:decomposition} and leaving the expression~\eqref{eq:decomposition}
possibly undefined.
In the following, we construct with \Cref{alg:S} an algorithm that handles such cases by simply mapping any
$x_j$ for which $S_j(\upto{x}{j})$ would be undefined onto itself. We show in \Cref{prop:evaluate_S} that the
resulting transport is indeed a measurable, bijective mapping pushing the uniform measure $\mu$ to the target
measure $\pi$ (despite no longer being a diffeomorphism) if the zeros of the density are a set of measure zero
with a particular structure, see \Cref{ass:A}.

\begin{algorithm}[H]
  \caption{Inverse transport computation from unnormalized density\\
    \emph{Input:} $f : \uid \to \Rpz$\\
    \emph{Output:} $S:\uid\to\uid$}\label{alg:S}
  \begin{algorithmic}[1]
    \State $s_d(\bsx) \dfn f(\bsx)$
    \For{$j=d,\dots,1$}
      \State $s_{j-1}(\upto{\bsx}{j-1}) \dfn \int_{0}^1 s_j(\upto{\bsx}{j})\dd x_j$
      \If{$s_{j-1}(\upto{\bsx}{j-1})>0$}
        \State $S_{j}(\upto{\bsx}{j}) \dfn \frac{\iuiu{x_j} s_j(\upto{\bsx}{j-1},t) \dd t}{s_{j-1}(\upto{\bsx}{j-1})}$
      \Else
        \State $S_{j}(\upto{\bsx}{j}) \dfn x_j$
      \EndIf
    \EndFor
    \State
    \Return $S=S_j(\upto{\bsx}{j})_{j=1}^d$
  \end{algorithmic}
\end{algorithm}

\begin{assumption}\label{ass:A}
  The function $f \in C^0(\uid,\Rpz)$ satisfies that $\set{\bsx\in \uid}{f(\bsx)=0}$ is a null-set and for
  each axis-parallel $1$-dimensional affine subspace $A\subseteq\R^d$ it holds that either (i)~$f$ is constant
  zero on $A\cap\uid$ or (ii)~$\set{\bsx\in A\cap\uid}{f(\bsx)=0}$ is a null-set w.r.t.~the $1$-dimensional
  Lebesgue measure.
\end{assumption}

This assumption ensures that all components $S_j$, $j \in \{1,\dots,d\}$ are bijective in $x_j$, and thus the
resulting transport $S$ is a bijection. Note that in particular strictly positive functions on $\uid$ satisfy
this condition. However, we also allow for functions that may for example be zero on certain $j$-dimensional
submanifolds, with $j<d$. Notably, any non-negative polynomial meets this assumption. Interpreting such $f$ as
an unnormalized version of a probability density on $\uid$, \Cref{alg:S} constructs an \emph{exact} transport
for the normalized density:

\begin{proposition}\label{prop:evaluate_S}
  Let $f$ satisfy \Cref{ass:A} and let $S=S(f)$ be generated by \Cref{alg:S}. Then $S$ and $T \dfn S^{-1}$ are
  monotone, triangular bijections from $\uid\to\uid$. Moreover, $T$ is measurable and the probability measure
  $T_\sharp\mu$ has Lebesgue density $\frac{f(\bsx)}{\int_{\uid} f(\bsy)\dd\bsy}$ on $\uid$.
\end{proposition}

The proof of \Cref{prop:evaluate_S} is given in \Cref{sec:ST}.

\begin{remark}\label{rmk:cf}
  \Cref{prop:evaluate_S} states in particular that \Cref{alg:S} is independent of a (possibly unknown)
  normalization constant.
\end{remark}

\subsection{Approximating $T$ from surrogate densities}\label{sec:alg}

If $\ftar$ is a density satisfying \Cref{ass:A}, then by \Cref{prop:evaluate_S}, \Cref{alg:S} provides an
explicit construction of the inverse KR map obtained by computing the antiderivatives of the marginal
densities of $\pi$. On first sight, the practical value of this observation seems limited, as the computation
of the marginals amounts to an intractable high-dimensional integration problem w.r.t.~$\pi$. To overcome
this problem, the idea is to use a two-step process:
  \begin{enumerate}
    \item Approximate the target density within a class of functions for which the integrals in \Cref{alg:S}
      are easy to compute.
    \item Use \Cref{alg:S} to determine the \emph{exact} KR map for the \emph{approximated} density.
  \end{enumerate}
This leads to a transport map pushing forward the nominal to the approximated target. It remains to choose an
approximation class for the density and a suitable notion of distance w.r.t.~which we approximate the target
density.

In this paper we propose to work with \emph{polynomial density surrogates} and the \emph{Hellinger distance}.
Recall that for two measures $\pi\ll\lambda$ and $\rho\ll\lambda$ on $\uid$, the Hellinger distance is
defined as the $L^2$-difference between the square root of their densities, i.e.
  \begin{equation*} \label{eq:hellinger}
    \hellinger(\pi,\rho) \dfn \norm[L^2(\uid)]{\ftars-\sqrt{f_\rho}}.
  \end{equation*}

This choice of distance is crucial in obtaining a convex optimization problem: The $L^2$-norm remains
well-defined if we replace $\sqrt{f_\rho}$ with any $g \in L^2(\uid, \mu)$~--~even if $g$ is unnormalized or
negative. We may thus choose any subspace $\cA \subset L^2(\uid, \mu)$, from which to construct (the
square-root of) a density surrogate. Minimizing~\eqref{eq:hellinger} over $\cA$ yields a convex optimization
problem. For any (possibly negative and unnormalized) $g \in \argmin_{h \in \cA} \norm[L^2(\uid)]{f_\pi^{1/2}
- h}$, $g^2$ is then positive and the algorithm to construct a transport, \Cref{alg:S}, is independent of a
normalization constant. Finally, note that we can equivalently solve for $\hat g \in \argmin_{h \in \cA}
\norm[L^2(\uid)]{\ftaruss - h}$~--~since $\hat g = \sqrt{c_\pi} g$ and the constant again does not affect the
resulting transport map. These considerations lead to the following algorithm:

\begin{algorithm}[H]
  \caption{Approximate transport from unnormalized density\\
    \emph{Input:} $\ftaru : \uid \to \Rpz$, subspace $\cA\subset L^2(\uid)$\\
    \emph{Output:} $T : \uid \to \uid$ such that $T_\sharp \mu \approx \pi$ } \label{alg:general}
  \begin{algorithmic}[1]
    \State determine $g \in \argmin_{h \in \cA} \norm[L^2(\uid)]{\ftarus - h}$
    \State compute $S \dfn S(g^2)$ via \Cref{alg:S}
    \State
    \Return $T \dfn S^{-1}$
  \end{algorithmic}
\end{algorithm}

\subsection{Constructing surrogate densities}\label{sec:constrdens}

Usually, the minimizer $g \in \argmin_{h \in \cA} \norm[L^2(\uid)]{\ftarus - h}$ cannot be
computed. Using point-wise density evaluations, however, we can solve the minimization problem approximately
via weighted least-squares (WLS). Stability and accuracy of this type of least-squares problem were analyzed
in detail in~\cite{MR3105946,MR3716755} and the related works~\cite{MR3342229,MR3640644,MR3949706,
cohen2021optimal}. Also see~\cite{MR3340149,MR3764433,MR3626543} for similar results.

To recall the method, let $f:\uid\to\R$ belong to $L^2(\uid, \rho)$ for some probability measure $\rho$ on
$\uid$. Let further $\cA = {\rm span} \{b_1, \dots, b_m\} \subset L^2(\uid, \rho)$ with $m \in \N$ and
$(b_i)_{i=1}^m$ orthonormal in $L^2(\uid, \rho)$. The goal is to approximate the minimizer $g \in \argmin_{h
\in \cA} \norm[L^2(\uid, \rho)]{f - h}$. To that end, the WLS method relies on $n \in \N$ random points
$\bschi = (\bschi_k)_{k=1}^n \subseteq \uid$, iid distributed according to a sampling measure $\eta$ on $\uid$
with positive Lebesgue density $f_\eta : \uid \to \Rp$. With

\newcommand{\wlsop}[3]{P^{#1}_{#2}[#3]}
\newcommand{\wlsseminorm}[1]{\snorm[\bschi]{{#1}}}
\newcommand{\wlsscalarproduct}[2]{{\langle#1,#2\rangle_{\bschi}}}
  \begin{equation*}
    \wlsseminorm{h} \dfn \sqrt{\wlsscalarproduct{h}{h}}, \quad
    \wlsscalarproduct{h}{g} \dfn \frac{1}{n}\sum_{k=1}^n \frac{1}{f_\eta(\bschi_k)} h(\bschi_k) g(\bschi_k)
  \end{equation*}
the WLS solution is given by
  \begin{equation*}\label{eq:tgls}
    \wlsop{\bschi}{\cA}{f} \dfn \sum_{i=1}^m c_i b_i \in \argmin_{h \in \cA} \wlsseminorm{f - h}.
  \end{equation*}
The coefficient vector $\bsc \dfn (c_i)_{i=1}^m \in \R^{m}$ can be computed by solving the linear system of
equations
  \begin{equation*} \label{eq:wls_system}
    \bsG \bsc = \bsd
  \end{equation*}
where
  \begin{equation*}
    \bsG \in \R^{m \times m}, \ G_{ij} \dfn  \wlsscalarproduct{b_i}{b_j} \quad \text{ and } \quad
    \bsd \in \R^{m}, \ d_i \dfn \wlsscalarproduct{f}{b_i}.
  \end{equation*}
The WLS algorithm as used throughout this work is summarized below in \Cref{alg:wls}. It includes an
additional resampling step (\Cref{alg:wls:resampling_a} to \Cref{alg:wls:resampling_b}) that ensures the
Gramian matrix $\bsG$ to be well-conditioned. This in turn allows us to establish the error bound on the WLS
surrogate as in \Cref{thm:wls}. The specific choices of $n$ and $f_\eta$ ensure that the number of resampling
steps is small in expectation (cf.~\eqref{thm:wls:proof:n_iterations}).

\begin{algorithm}[H]
  \caption{Weighted least-squares approximation\\
    \emph{Input:} $f \in L^2(\uid, \rho)$, orthonormal system $(b_i)_{i=1}^m \subset L^2(\uid, \rho)$ \\
    \emph{Output:} $\wlsop{\bschi}{\cA}{f} \in \cA = {\rm span} \{b_1, \dots, b_m\}$}\label{alg:wls}
  \begin{algorithmic}[1]
    \State $n = 10 m \lceil \log(4 m) \rceil$ \label{alg:wls:n}
    \State $f_\eta(\bsx) \dfn \frac{1}{m} \sum_{i=1}^m b_i(\bsx)^2$ \label{alg:wls:f_eta}
    \State $\bsG = \mathbf{0} \in \R^{m \times m}$ \label{alg:wls:resampling_a}
    \While{$\norm{\bsG - \bsI} > \frac{1}{2}$} \label{alg:wls:check_e}
      \State $\bschi = (\bschi_k {\sim}\eta)_{k=1}^n$ \label{alg:wls:sampling}
      \State $\bsG = \left(
                \frac{1}{n}\sum_{k=1}^n \frac{1}{f_\eta(\bschi_k)} b_i(\bschi_k) b_j(\bschi_k)
              \right)_{i,j\in \{1,\dots,m\}}$ \label{alg:wls:gramian}
    \EndWhile \label{alg:wls:resampling_b}
    \State $\bsd = \left(
              \frac{1}{n} \sum_{k=1}^n \frac{1}{f_\eta(\bschi_k)} b_i(\bschi_k) f (\bschi_k)
            \right)_{i \in \{1,\dots,m\}}$ \label{alg:wls:b}
    \State $\bsc = \bsG^{-1} \bsd \in \R^{m}$ \label{alg:wls:solve}
    \State
    \Return $\wlsop{\bschi}{\cA}{f} = \sum_{i=1}^m c_i b_i$ \label{alg:wls:return}
  \end{algorithmic}
\end{algorithm}

\subsection{Main results}\label{sec:mainresult}

Let now $\cA\subseteq C^0(\uid)$ be a finite dimensional linear space such that for each $0\neq g\in\cA$, the
function $g^2$ satisfies \Cref{ass:A}. This is for example the case, if $\cA$ is a space of multivariate
polynomials. With the set of measures
\newcommand{\cAM}{\cM_{\ge}}
\newcommand{\cTM}{\cT_{\ge}}
\begin{equation}\label{eq:cM}
  \cAM \dfn \setc{\tilde\pi\text{ is a probability measure}}{\Big(\frac{\dd\tilde\pi}{\dd\lambda}\Big)^{1/2}=g\in\cA,~g\ge 0\text{ in }\uid}
\end{equation}
we then have the following result. The proof is given in \Cref{sec:proofmain1}.

\begin{theorem}\label{thm:main1}
  Let $\pi$ be a probability measure on $\uid$ with (possibly) unnormalized Lebesgue density $\ftaru\in
  L^2(\uid)$. Let $T$ be obtained by \Cref{alg:general}. Then
  \begin{equation}\label{eq:main1}
    \hellinger(\pi,T_\sharp\mu) \le 2 \inf_{\tilde\pi\in\cAM} \hellinger(\pi,\tilde\pi).
  \end{equation}
\end{theorem}

\begin{remark}\label{rmk:vi}
  With the set of transportss
    \begin{equation}\label{eq:Tvi}
      \cTM \dfn \set{S(g^2)^{-1}}{0\neq g\in\cA,~g\ge 0\text{ in }\uid}
    \end{equation}
  where $S(g^2)$ is as in \Cref{alg:S}, we can equivalently formulate \Cref{thm:main1} as
    $$
      \hellinger(\pi,T_\sharp\mu) \le 2 \inf_{\tilde T \in \cTM} \hellinger(\pi, \tilde T_\sharp\mu).
    $$
\end{remark}

\Cref{thm:main1} and \Cref{rmk:vi} state that our algorithm essentially allows us to solve a variational
inference problem of the type~\eqref{eq:opt}, i.e.~we are able (up to a factor $2$) to minimize the distance
between the target distribution $\pi$ and all distributions within the class $\cAM$ in~\eqref{eq:cM}. In order
to establish convergence results, it now suffices to bound the right-hand side in~\eqref{eq:main1}, which can
be done by choosing $\cAM$ large enough.

Using polynomial ansatz spaces allows proving convergence rates for densities belonging to classical
smoothness spaces. In the following we consider the orthonormal Legendre polynomials $(L_n)_{n \in \N}$
normalized in $L^2(\ui)$, cf.~\eqref{eq:def_legendre_ui}. A multivariate orthonormal basis on $\uid$ is
then given by $(L_\bsnu \dfn \prod_{i=1}^d L_{\nu_i})_{\bsnu \in \N^d}$, and finite-dimensional polynomial
ansatz space can be defined via a finite multi-index set $\Lambda\subseteq\N_0^d$ as
  \begin{equation}\label{eq:PLambda}
    \cA = \mathbb{P}_{\Lambda} \dfn {\rm span}\set{L_\bsnu(\bsx)}{\bsnu\in\Lambda}.
  \end{equation}
With this ansatz space we obtain the following \emph{computable} variant of \Cref{alg:general}:

\begin{algorithm}[H]
  \caption{Approximate transport computation from unnormalized density\\
    \emph{Input:} $\ftaru : \uid \to \Rpz$, $\Lambda\subsetq\N_0^d$\\
    \emph{Output:} $T : \uid \to \uid$ such that $T_\sharp \mu \approx \pi$ } \label{alg:main}
  \begin{algorithmic}[1]
    \State $g = \alg{\ref{alg:wls}}(\ftarus, (L_\bsnu)_{\bsnu \in \Lambda})$\label{alg:main:surrogate}
    \State $S = \alg{\ref{alg:S}}(g^2)$
    \State
    \Return $T \dfn S^{-1}$
  \end{algorithmic}
\end{algorithm}

The following theorem states an algebraic (for $C^k$ densities) or exponential (for analytic densities)
convergence rate of \Cref{alg:main} in terms of the cardinality of the ansatz space, and, up to logarithmic
factors, the number of required target density evaluations, cf.~\Cref{rmk:npoints}. Its proof, together with a
more detailed statement and other variants, will be given in \Cref{sec:main_proofs}.

\begin{theorem}\label{thm:main2}
  For $d\in \N$, let $\pi$ be a probability measure on $\uid$ with density $\ftar\in L^1(\uid)$. Then, for
  every $m\in\N$ exists $\Lambda_m \subsetq \N_0^d$ of cardinality $m$ such that for
    \begin{equation*}
      T^m \dfn \alg{\ref{alg:main}}(\ftaru,\Lambda_m)
    \end{equation*}
  holds
    \begin{equation*}
      \hellinger(\pi, T^m_\sharp \mu) \le
        \begin{cases}
          C_1 \exp(- \beta m^{1/d}) \ & \text{ if } \ftarus \text{ is analytic }\\
          C_2 m^{-\frac{k+\alpha}{d}} \  & \text{ if } \ftarus \in \ckalpha \text{ for some $k \in N$, $\alpha \in [0,1]$}
        \end{cases}
    \end{equation*}
  with constants $C_1 = C_1(d,\ftar) > 0, C_2 = C_2(d,k,\ftar) >0$ and $\beta = \beta(d,\ftar) > 0$.
\end{theorem}

\begin{remark}\label{rmk:npoints}
  The computation of $T^m$ requires $O(m\log(m))$ evaluations of the unnormalized density $\ftaru$ (see
  \Cref{alg:wls}, \Cref{alg:wls:n}). Thus, up to logarithmic terms, \Cref{thm:main2} provides a convergence
  rate in terms of the number of density evaluations. Moreover, the sets $\Lambda_m$ used in this work can be
  explicitly constructed with complexity $O(m)$, see \Cref{alg:construct_lambda}.
\end{remark}

Let us highlight the following essential characteristics of \Cref{alg:general} (and \Cref{alg:main}) and
\Cref{thm:main2}: First, both our algorithm and convergence statements remain completely independent of the
unknown normalization constant. This feature is of paramount significance in Bayesian inference applications,
where one typically encounters an unnormalized density. Second, unlike other analyses such as those
in~\cite{zech2022sparse, zech2021sparse, baptista2023approximation}, \Cref{thm:main2} does not necessitate
uniform positivity of the target density. Specifically, the target distribution $\pi$ may exhibit multiple
disjoint modes, and we further comment on this in \Cref{sec:multimodal}. Lastly, because our algorithm
circumvents the need to solve a nonconvex optimization problem we can provide a comprehensive convergence
analysis giving a bound on the number of required evaluations of $\ftaru$, and the resulting error.

In \Cref{sec:implementation_and_sampling_cost} efficient algorithms to compute $S$ and $T$ will be discussed.
The resulting cost of one evaluation of the inverse map $S$ scales as $\cO(\abs{\Lambda}^{1+\frac{1}{d}})$ in
the dimension $\abs{\Lambda}$ of the ansatz space. Since $S$ cannot be inverted analytically, we construct an
approximation $\tilde T \approx T = S^{-1}$ via a multivariate numerical bisection algorithm. Here, we
carefully choose the accuracy in order to achieve the same convergence rates for $\tilde T$ as we have derived
for the (inaccessible) $T$. The resulting cost of evaluating $\tilde T$ is equivalent to the cost of
evaluating $S$ up to a multiplicative factor logarithmic in the inverse of the convergence rate.

\subsection{Multimodality}\label{sec:multimodal}

We shortly explain why, unlike in other methods and analyses, we do not require uniform positivity of the
target density $\ftar$ in \Cref{thm:main2}. Consider the one-dimensional example in \Cref{fig:multimodal}. In
this simple setting, the inverse transport $S:[0,1]\to [0,1]$ corresponds to the cumulative distribution
function $S(x)=\int_0^x \ftar(y)\dd y$, and the transport $T:[0,1]\to [0,1]$ pushing forward $\mu$ to $\pi$ is
the generalized inverse of $S$.

In case $\ftar(x)$ vanishes on some subinterval of $[0,1]$ of positive length, then $T$ is discontinuous. This
must (also in dimension $d>1$) necessarily be the case as $T$ will transform the unimodal uniform distribution
into a multimodal one. As such, polynomials are not well suited to approximate $T$. The key observation is
that polynomials may still effectively approximate $\ftar$ and (by consequence) its antiderivative $S$. This
can be considered as an instance of a general issue encountered in normalizing flows, of one direction of the
process being stable, and the reverse direction being unstable. By construction, our algorithm does not
approximate $T$ by a polynomial, but by the \emph{inverse} of a polynomial, which is how we are able to bypass
this issue in the sense of obtaining the convergence rate in \Cref{thm:main2}.

One additional important point is that computing $T:[0,1]\to [0,1]$ requires inverting $S:[0,1]\to [0,1]$.
Since $S$ is obtained as the integral of a squared polynomial, it is guaranteed that $S$ is (either constant
$0$) or strictly positively increasing, so that it may always be inverted. This is the reason why we work
under \Cref{ass:A}. We emphasize once more, that this assumption is \emph{not required} on the
target density $\ftar$, but only on its surrogate. For the polynomial based surrogates considered in this
work, the assumption is always satisfied.

\begin{figure}
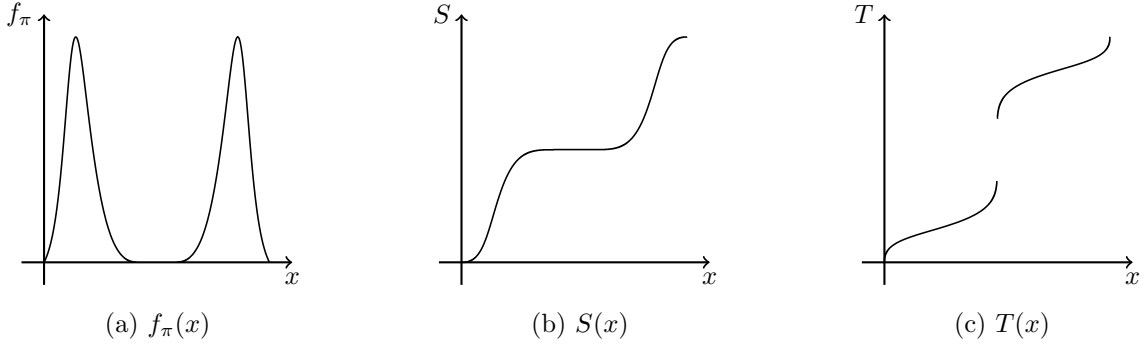

  \centering
  \begin{subfigure}{0.32\textwidth}
    \centering
    \input{f_density_plot.tex}
    \caption{$\ftar(x)$}
  \end{subfigure}
  \hfill
  \begin{subfigure}{0.32\textwidth}
    \centering
    \input{f_antiderivative_plot.tex}
    \caption{$S(x)$}
  \end{subfigure}
  \hfill
  \begin{subfigure}{0.32\textwidth}
    \centering
    \input{f_inverse_antiderivative_plot.tex}
    \caption{$T(x)$}
  \end{subfigure}
  \caption{For a multimodal distribution $\pi$ on $[0,1]$, its density $\ftar$ will vanish or be very small in
    certain parts of the domain. By consequence, the inverse transport $S=T^{-1}$ (corresponding to the
    antiderivative of $\ftar$) will be (almost) constant in those areas, and the transport $T=S^{-1}$ can
    exhibit discontinuities. Even if $\ftar$ and $S$ are $C^\infty$ functions that are well approximated by
    polynomials, the same is in general not true for $T$.}\label{fig:multimodal}
\end{figure}

\section{Proof and discussion of the main results}\label{sec:main_proofs}

This section provides detailed proofs, additional insights, and further variants of the results presented in
\Cref{sec:main}.

\subsection{Well-definedness of $S$ and $T$ (proof of \Cref{prop:evaluate_S})}\label{sec:ST}

We give the proof of \Cref{prop:evaluate_S} and split the argument into several lemmata. Throughout this
section let the assumptions of \Cref{prop:evaluate_S} be satisfied, i.e.~let $f$ satisfy \Cref{ass:A}, and let
$s_j$, $S_j$ and $S=(S_j)_{j=1}^d$ be as in \Cref{prop:evaluate_S}, that is $s_d(\bsx)=f(\bsx)$ and
  \begin{subequations}
    \begin{equation} \label{eq:Sproof}
      s_j(\upto{\bsx}{j})=\int_{\uij{d-j}}f (\bsx)\dd \from{\bsx}{j+1}\qquad\forall j\in\{0,\dots,d-1\}
    \end{equation}
    as well as
    \begin{equation}\label{eq:Siproof}
      S_j(\upto{\bsx}{j})=
      \begin{cases}
        \frac{\int_{0}^{x_j}s_j(\upto{\bsx}{j-1},t)\dd t}{s_{j-1}(\upto{\bsx}{j-1})} &\text{if }s_{j-1}(\upto{\bsx}{j-1})>0\\
        x_j&\text{else}.
      \end{cases}
    \end{equation}
  \end{subequations}

Additionally, denote in the following $N_0 \dfn \emptyset$ and
  \begin{equation*}
    N_j \dfn s_{j}^{-1}(\{0\}) = \set{\bsx \in \uid}{s_{j}(\upto{\bsx}{j}) = 0} \subseteq \uid \qquad
    \forall j\in\{1,\dots,d\}.
  \end{equation*}
Observe that for $j<d$ holds that $s_j(\upto{\bsx}{j})=\int_0^1s_{j+1}(\upto{\bsx}{j+1})\dd x_{j+1}=0$ implies
$s_{j+1}(\upto{\bsx}{j+1})=0$ for all $x_{j+1}\in [0,1]$ so that
  \begin{equation}\label{eq:Ninested}
    \emptyset = N_0 \subseteq N_1 \subseteq \cdots \subseteq N_d.
  \end{equation}
In other words, while $N_d$ denotes the set on which the density $f$ is zero, $N_j$ denotes the subset of
$N_d$, for which it holds true that if $\bsx = (\upto{\bsx}{j}, \from{\bsx}{j+1}) \in N_j$, then
$(\upto{\bsx}{j}, \from{\bsx^*}{j+1}) \in N_d$ for all $\from{\bsx^*}{j+1} \in \uij{j-1}$.

\begin{lemma} \label{lma:polynomial_transport_well_def}
  The map $S=(S_j)_{j=1}^d:\uid\to\uid$ with the $S_j$ as in \Cref{eq:Siproof} is a triangular, monotone
  bijection, and it holds
    \begin{equation*}
      \prod_{j=1}^d\partial_{x_j}S_j(\upto{\bsx}{j}) = \frac{f(\bsx)}{\int_{\uid}f(\bsy)\dd\bsy}
    \end{equation*}
  for all $\bsx$ in the complement of the null set $N_{d-1}$.
\end{lemma}

\begin{proof}
  Well-definedness and triangularity is clear. To show monotonicity and bijectivity consider the mapping
  $x_j \mapsto S_j(\upto{\bsx}{j-1},x_j)$ for any fixed $\upto{\bsx}{j-1}$ and $j\in\{1,\dots,d\}$. We
  distinguish between two cases:
  \begin{enumerate}
    \item $\upto{\bsx}{j-1}$ is such that $(\upto{\bsx}{j-1}, \from{\bsx^*}{j}) \in N_j$ for any
    $\from{\bsx^*}{j}$: Then $S_j(\upto{\bsx}{j}) = x_j$ is strictly monotonically increasing in $x_j\in
    [0,1]$.
    \item $\upto{\bsx}{j-1}$ is such that $(\upto{\bsx}{j-1}, \from{\bsx^*}{j}) \in N_j^c$ for any
    $\from{\bsx^*}{j}$: Then $s_{j-1}(\upto{\bsx}{j-1})=\int_0^1 s_j(\upto{\bsx}{j})\dd x_j >0$
    and thus $s_j(\upto{\bsx}{j})>0$ a.e.\ in $x_j\in [0,1]$ by \Cref{ass:A}. Hence, $x_j\mapsto
    \int_{0}^{x_j}s_j(\upto{\bsx}{j-1},t)\dd t$ and by consequence $x_j\mapsto S_j(\upto{\bsx}{j})$ is
    strictly monotonically increasing in $x_j\in [0,1]$.
  \end{enumerate}
  This shows monotonicity of $x_j\mapsto S_j(\upto{\bsx}{j})$. Bijectivity follows from the fact that in both
  cases $S_j(\upto{\bsx}{j-1},0) = 0$ and $S_j(\upto{\bsx}{j-1},1) = 1$. Since this holds for any
  $\upto{\bsx}{j-1}$ and $j\in\{1,\dots,d\}$ it follows that $S:\uid\to\uid$ is monotone and bijective.

  Finally, let $\bsx\in N_{d-1}^c$. Then, by~\eqref{eq:Ninested} we have $s_1(\upto{\bsx}{1}),\dots,
  s_{d-1}(\upto{\bsx}{d-1})\neq 0$, and by continuity of these functions they are all nonzero in a
  neighborhood of $\bsx$. Thus, by their definition in~\eqref{eq:Sproof}
    \begin{equation*}
      \prod_{j=1}^d\partial_{x_j}S_j(\upto{\bsx}{j})
      = \prod_{j=1}^d \frac{s_j(\upto{\bsx}{j})}{s_{j-1}(\upto{\bsx}{j-1})}
      = \frac{s_d(\bsx)}{s_0(\bsx)} = \frac{f(\bsx)}{\int_{\uid}f(\bsy)\dd\bsy}.
    \end{equation*}
  The set $N_{d-1}$ is a null set by \Cref{ass:A}.
\end{proof}

\begin{lemma}\label{lma:TBorel}
  The map $T \dfn S^{-1}:\uid\to\uid$ is a triangular, monotone, Borel measurable bijection.
\end{lemma}

\begin{proof}
  By \Cref{lma:polynomial_transport_well_def}, $S:\uid\to\uid$ is a triangular and monotone bijection. It is
  easy to check that the same then holds for $T:\uid\to\uid$. It remains to verify Borel measurability.

  Set $P_j \dfn N_j \setminus N_{j-1}$ for $j \in \{1, \dots, d-1\}$ and $P_d \dfn \uid \setminus N_{d-1}$.
  Then, due to the nestedness of the $N_j$ (cf.~\eqref{eq:Ninested}), it holds that $\uplus_{j=1}^{d-1} P_j =
  N_{d-1}$, i.e.~$(P_j)_{j=1}^d$ forms a partition of $\uid$. Moreover, $N_j = s_{j}^{-1}(\{0\}) \in \B{\uid}$
  since $f$ and therefore $s_j$ are continuous for each $j\in\{1,\dots,d\}$.
  Thus, $P_j \in \B{\uid}$ for each $j\in\{1,\dots,d\}$.

  Next, we claim that $S|_{P_j}$ maps Borel sets to Borel sets, i.e.~if $B \in \B{P_j}$ then $S(B) \in
  \B{S(P_j)}$. First, note that by definition
    \begin{equation*}
      P_j = Q_j \times \uij{d-j} \quad \text{ where } \quad
      Q_j \dfn \set{\upto{\bsx}{j} \in \uij{j}}{s_{j-1}(\upto{\bsx}{j-1}) \neq 0,~s_{j}(\upto{\bsx}{j}) = 0}
      \subseteq \uij{i}
    \end{equation*}
  and hence for $j<d$ due to~\eqref{eq:Ninested}
    \begin{equation}\label{eq:Qiprop}
      \bsx \in P_j \qquad \Rightarrow \qquad
      s_1(x_1),\dots,s_{j-1}(\upto{\bsx}{j-1}) \neq 0 \quad \text{and} \quad
      s_{j}(\upto{\bsx}{j}),\dots,s_{d}(\upto{\bsx}{d}) = 0.
    \end{equation}
  Moreover, it is a property of separable spaces that it holds
    \begin{equation*}
      \B{P_j} = \B{Q_j} \otimes \underbrace{\B{\ui} \otimes \dots \otimes \B{\ui}}_{d-j \text{ times}}
    \end{equation*}
  i.e.~the Borel sigma algebra on $P_j$ coincides with the product sigma algebra~\cite[Theorem~14.8.]{klenke}.
  Thus, if $\mathcal{G}_{Q_j}$ is a generating set of $\B{Q_j}$ and $\mathcal{G}_{\ui}$ is a generating set of
  $\B{\ui}$, then $\mathcal{G}_{P_j} \dfn \mathcal{G}_{Q_j} \times \mathcal{G}_{\ui} \times \cdots \times
  \mathcal{G}_{\ui}$ generates $\B{P_j}$. Now fix $G \in \mathcal{G}_{P_j}$, i.e.~let $G$ be of the form $G =
  G_j \times G_{j+1} \times \cdots \times G_{d}$ with $G_j \in \mathcal{G}_{Q_j}$ and $G_k \in
  \mathcal{G}_{\ui}$ for $k \in \{j+1, \dots, d\}$. We wish to compute the image of $G$ under $S$. Fix
  $\bsx\in G$ arbitrary. Then by~\eqref{eq:Qiprop} and~\eqref{eq:Siproof} $S_k(\upto{\bsx}{k}) = x_k$ for all
  $k>j$. Therefore,
    \begin{align*}
      S(G) &= \set{S(\bsx)}{\bsx\in G}\\
      &=\set{(S_1(x_1),\dots,S_{j}(\upto{\bsx}{j}),x_{j+1},\dots,x_{d})}{(\upto{\bsx}{j},x_{j+1},\dots,x_d)\in G_j\times\cdots\times G_d}\\
      &= \upto{S}{j}(G_j) \times G_{j+1} \times \cdots \times G_{d}.
    \end{align*}
  By definition, $\upto{S}{j}\dfn(S_k)_{k=1}^j:P_j\to\R^j$ is continuous and injective (cf.~\eqref{eq:Siproof}
  and \Cref{lma:polynomial_transport_well_def}). According to~\cite[Theorem~15.1]{kechris2012classical} such
  functions map Borel sets to Borel sets, and thus $\upto{S}{j}(G_j) \in \B{\R^j}$. Therefore,
  $$
    S(G) \in \B{\upto{S}{j}(Q_j)} \times \underbrace{\B{\ui} \times \cdots \times \B{\ui}}_{d-j \text{ times}} \in \B{\R^d},
  $$
  i.e.~$S(G)$ is Borel measurable. Since we have shown the statement for arbitrary $G$ in a generating set of
  $\B{P_j}$, it applies to all $B \in \B{P_j}$. This shows the claim.

  Finally, consider an arbitrary Borel set $B \in \B{\uid}$.
  Then $S(B) = \cup_{j=1}^d S(B\cap P_j) \in \B{\uid}$.
  In all, $S$ maps Borel sets to Borel sets, i.e.~$T \dfn S^{-1}:\uid\to\uid$ is Borel measurable.
\end{proof}

The measurability of $T$ shown in \Cref{lma:TBorel} ensures that $T_\sharp\mu$ is a well-defined measure on
$\uid$ equipped with the Borel sigma-algebra. We next give a formula for its Lebesgue density:

\begin{lemma}\label{lemma:Tmudensity}
  The pushforward of the uniform measure $\mu$ under $T$ has the Radon-Nikodym derivative
  \begin{equation}\label{eq:lemma:Tmudensity}
    \frac{\dd T_\sharp \mu}{\dd \mu} (\bsx) = \prod_{j=1}^d \partial_j S_j (\upto{\bsx}{j})
    = \frac{f(\bsx)}{\int_{\uid}f(\bsy)\dd\bsy}.
  \end{equation}
\end{lemma}

\begin{proof}
  By \Cref{lma:polynomial_transport_well_def} $S:\uid\to\uid$ is a triangular, monotone bijection. From its
  definition in~\eqref{eq:Sproof} it is also clear that $S$ is measurable, and that $x_j\mapsto
  S_j(\upto{\bsx}{j})$ belongs to $C^1([0,1])$ for every $j\in\{1,\dots,d\}$. Thus, by~\cite[Lemma
  2.3]{bogachev}, for every measurable $M\in\B{\uid}$ holds\footnote{We choose $\varphi$ to be the indicator
  function on $S(M)$ in~\cite[Lemma 2.3]{bogachev}. This function is measurable since $S(M)$ is measurable by
  \Cref{lma:TBorel}.}
    \begin{equation*}
      T_\sharp\mu(M) = \mu(S(M)) = \int_{S(M)}1\dd\bsx
      = \int_{M}\prod_{j=1}^d\partial_{x_j}S_j(\upto{\bsx}{j})\dd\bsx.
    \end{equation*}
  This shows that $\prod_{j=1}^d\partial_{x_j}S_j(\upto{\bsx}{j})$ is a Lebesgue density of $T_\sharp\mu$. By
  \Cref{lma:polynomial_transport_well_def} this function coincides with the right-hand side
  in~\eqref{eq:lemma:Tmudensity} almost everywhere in $\uid$.
\end{proof}

In all this concludes the proof of \Cref{prop:evaluate_S}.

\subsection{Connection to variational inference}\label{sec:proofmain1}

We next come to the proof of \Cref{thm:main1}, which shows that, up to the multiplicative factor $2$,
\Cref{alg:general} generates a quasi-optimal transport w.r.t.~the set of measures $\cAM$ in~\eqref{eq:cM}.
After the proof we discuss in more detail the connection to the variational inference problem~\eqref{eq:opt}.

\begin{proof}[Proof of \Cref{thm:main1}]
  Let $\int_{\uid} \ftaru = c_\pi$, i.e.~$\ftaru=c_\pi \ftar$. Denote
  $g\dfn\argmin_{h\in\cA}\norm[L^2(\uid)]{\ftarus - h}$. Then, since $\cA$ is a linear space so that
  $g\in\cA$ implies $\frac{1}{c_\pi} g\in\cA$,
    \begin{align}\label{eq:main1_p1}
      \norm[L^2(\uid)]{\ftarus - g}
        &=   \min_{h\in\cA}\norm[L^2(\uid)]{\ftarus - h}\nonumber\\
        &=   \sqrt{c_\pi}\min_{h\in\cA}\norm[L^2(\uid)]{\ftars - h}\nonumber\\
        &\le \sqrt{c_\pi}\min_{\set{h\in\cA}{\int_{\uid}h^2=1,~h\ge 0}}\norm[L^2(\uid)]{\ftars - h}\nonumber\\
        &=   \sqrt{c_\pi}\min_{\tilde\pi\in\cAM} \hellinger(\pi,\tilde\pi).
    \end{align}
  Next, since $g^2\in\cA\subseteq C^0(\uid)$, according to \Cref{prop:evaluate_S} it holds that
  $T_\sharp\mu$ has density $\frac{g^2}{c_g}$ where $c_g\dfn\int_{\uid}g^2$, and where $T = S(g^2)^{-1}$
  is obtained via \Cref{alg:general}. Thus,
    \begin{align}\label{eq:dHtTpi}
      \hellinger(\pi, T_\sharp\mu)
      &=   \normc[L^2(\uid)]{\frac{\ftarus}{\sqrt{c_\pi}} - \frac{|g|}{\sqrt{c_g}}}
      \le  \normc[L^2(\uid)]{\frac{\ftarus}{\sqrt{c_\pi}} - \frac{g}{\sqrt{c_g}}}\nonumber\\
      &\le \normc[L^2(\uid)]{\frac{\ftarus}{\sqrt{c_\pi}} - \frac{g}{\sqrt{c_\pi}}}
         + \normc[L^2(\uid)]{\frac{g}{\sqrt{c_\pi}} - \frac{g}{\sqrt{c_g}}}\nonumber\\
      &=   \frac{1}{\sqrt{c_\pi}} \normc[L^2(\uid)]{\ftarus - g}
         + \frac{\norm[L^2(\uid)]{g}}{\sqrt{c_\pi}}\Big|1 - \frac{\sqrt{c_\pi}}{\sqrt{c_g}}\Big|.
    \end{align}
  Upon observing that $\norm[L^2(\uid)]{g}=\sqrt{c_g}$ we get
    \begin{equation*}
      \norm[L^2(\uid)]{g}\Big|1 - \frac{\sqrt{c_\pi}}{\sqrt{c_g}} \Big|
      = |\sqrt{c_\pi}-\sqrt{c_g}|
      = |\norm[L^2(\uid)]{\ftarus} - \norm[L^2(\uid)]{g}|
      \le \norm[L^2(\uid)]{\ftarus - g}
    \end{equation*}
  so that together with~\eqref{eq:dHtTpi}
    \begin{equation*}
      \hellinger(\pi,T_\sharp\mu)\le \frac{2}{\sqrt{c_\pi}}\norm[L^2(\uid)]{\ftarus - g}.
    \end{equation*}
  Combining the last estimate with~\eqref{eq:main1_p1} yields the result.
\end{proof}

Let us now further comment on the connection of \Cref{alg:general} to problem~\eqref{eq:opt}. According to
\Cref{prop:evaluate_S}, \Cref{alg:general}
computes an exact transport map for a surrogate density $g^2$, where $g\in\cA$ for some set $\cA$. Thus, we
allow the transport to belong to the set
  \begin{equation*}
    \cT \dfn \set{S(g^2)^{-1}}{0\neq g\in\cA}.
  \end{equation*}
According to \Cref{rmk:vi}, our algorithm computes a transport map that achieves quasi-optimality in the
\emph{subset} $\cTM$ of $\cT$ (cf.~\eqref{eq:Tvi}). Indeed, due to $g$ in line 1 of \Cref{alg:general}
solving the optimization problem
  \begin{equation}\label{eq:gopti}
    g = \argmin_{h \in \cA}\norm[L^2(\uid)]{\ftarus - h},
  \end{equation}
in general it does not compute a quasi-optimal transport within $\cT$. To explain this subtlety further
consider the following example:

\begin{example}
  Consider the unnormalized probability density $\ftaru(x) = (x-1/2)^2$ on $[0,1]$, and let $\cA$ be the
  set of polynomials of degree at most $1$ on $[0,1]$. Then $g(x) \dfn x-1/2\in\cA$ yields the exact
  representation $g^2=\ftaru$. However, since $\ftarus=|x-1/2|$, determining $g$
  via~\eqref{eq:gopti} will not yield an exact (or even good) density surrogate in this case.
\end{example}

Specifically, assuming as before that $\cA\subseteq C^0(\uid)$ is a finite dimensional linear space such that
for each $0\neq g\in\cA$, the function $g^2$ satisfies \Cref{ass:A}, we have the following:

\begin{proposition}\label{prop:main1}
  Let $\pi$ be a probability measure on $\uid$ with (possibly) unnormalized Lebesgue density $\ftaru \in
  L^2(\uid)$. Let $T$ be obtained by \Cref{alg:general}. Then
    \begin{equation*}
      \inf_{\tilde T \in \cT} \hellinger(\pi, \tilde T_\sharp\mu)
      \le \hellinger(\pi, T_\sharp\mu)
      \le 2 \inf_{\tilde T \in \cTM} \hellinger(\pi, \tilde T_\sharp\mu).
    \end{equation*}
\end{proposition}

\begin{proof}
  The second inequality holds by \Cref{rmk:vi}. The first inequality follows by the fact that $T \in
  \cT$ by definition of $T$.
\end{proof}

\Cref{prop:main1} clarifies the meaning of \Cref{alg:general} for~\eqref{eq:opt}: While~\eqref{eq:opt} is in
general a nonconvex optimization problem, \Cref{alg:general} only requires to solve the \emph{convex}
optimization problem~\eqref{eq:gopti}, yet is able to achieve a solution at least as good as the minimizer in
$\cTM$.

\subsection{Approximating the target density with weighted least-squares} \label{sec:wls}

In this section we tailor results on the error achieved by the  weighted least-squares method
from~\cite[Theorem~1]{MR3105946} and~\cite[Theorem~2.1~(iii)]{MR3716755} to our use case. Specifically, our
goal is to bound the error of \Cref{alg:wls} in terms of the error of best approximation defined as
  \begin{equation*}
    e_\cA^\infty\left(f\right) \dfn \inf_{h \in \cA} \norm[L^\infty(\uid)]{f - h},
  \end{equation*}
where $\cA$ is a finite dimensional linear ansatz space. This result is then combined with specific polynomial
convergence rates in the next section. We additionally provide an upper bound on the cost of \Cref{alg:wls}
measured in number of floating point operations. Subsequently, we comment on efficient implementation.

\begin{theorem}\label{thm:wls}
  Let $f \in C^0(\uid)$ and $(b_i)_{i=1}^m \subset L^{2}(\uid,\rho)$
  an orthonormal system for some $m \in \N$ and a probability measure $\rho$ on $\uid$.
  Then \Cref{alg:wls} produces an approximation $\wlsop{\bschi}{\cA}{f}$ satisfying
    \begin{equation}\label{eq:thm_wls_bd}
      \norm[L^2(\uid, \rho)]{f - \wlsop{\bschi}{\cA}{f}} \leq (1+\sqrt{2}) e_{\cA}^\infty (f).
    \end{equation}
\end{theorem}

\begin{proof}
  For $n \in \N$ consider the event
  \begin{equation*}
    E \dfn
    \left\{
      (\bschi_k)_{k=1}^n \in \uij{d \times n}\ |\ \norm{ \bsG((\bschi_k)_{k=1}^n) - \bsI} \leq \frac{1}{2}
    \right\}.
  \end{equation*}
  Analogously to~\cite[Theorem~1]{MR3105946} (replacing the unweighted Gramian with the \textit{weighted}
  Gramian matrix as has also been done in the proof of~\cite[Theorem~2.1~(i)]{MR3716755}, fixing
  $\delta = 1/2$ and also noting the correction of the constant in~\cite{MR3913878}) it can be shown that
  \begin{equation} \label{eq:P_E}
    \bbP(E^c) \leq 2 m \exp\left(- c \frac{n}{K_\eta((b_i)_{i=1}^m)}\right)
  \end{equation}
  with $c \dfn \frac{3}{2} \log(\frac{3}{2}) - \frac{1}{2}$ and
  \begin{equation}\label{eq:Kw}
    K_\eta ((b_i)_{i=1}^m)
    \dfn \sup_{\bsx \in \uid} \frac{1}{f_\eta(\bsx)} \sum_{i=1}^m b_i(\bsx)^2 \in (0,\infty).
  \end{equation}
  For $f_\eta(\bsx) = \frac{1}{m} \sum_{i=1}^m b_i(\bsx)^2$ as used in \Cref{alg:wls}, \Cref{alg:wls:f_eta},
  it holds that $K_\eta = m$. Then, inequality~\eqref{eq:P_E} together with the observation that $\frac{1}{c}
  < 10$ implies that $n = 10 m \log (4 m)$ (as set in \Cref{alg:wls}, \Cref{alg:wls:n}) guarantees
    \begin{equation}\label{eq:wls_p_e}
      \bbP(E) \ge \frac{1}{2}.
    \end{equation}
  Now, \Cref{alg:wls:resampling_a} to \Cref{alg:wls:resampling_b} in \Cref{alg:wls} ensure by resampling that
  $(\bschi_k)_{k=1}^n \in E$. Then, bound~\eqref{eq:thm_wls_bd} for $(\bschi_k)_{k=1}^n \in E$ has been shown
  as part of the proof of~\cite[Theorem~2.1~(iii)]{MR3716755}. We repeat the essential steps here for the
  convenience of the reader. First, we observe that $\bschi \in E$ entails the norm equivalence
    \begin{equation*}
      \frac{1}{2} \norm[L^2(\uid, \rho)]{h}^2
      \le \wlsseminorm{h}^2 \le \frac{3}{2} \norm[L^2(\uid, \rho)]{h}^2 \quad \forall h \in \cA
    \end{equation*}
  (note that while $\wlsseminorm{\cdot}$ is only a seminorm on $L^2$, it is a norm on $\cA$). Then,
    \begin{align*}
      \norm[L^2(\uid, \rho)]{f - \wlsop{\bschi}{\cA}{f}}
      &\le \norm[L^2(\uid, \rho)]{f - h} + \norm[L^2(\uid, \rho)]{h - \wlsop{\bschi}{\cA}{f}} \\
      &\le \norm[L^2(\uid, \rho)]{f - h} + \sqrt{2}  \wlsseminorm{h - \wlsop{\bschi}{\cA}{f}} \\
      &\le \norm[L^2(\uid, \rho)]{f - h} + \sqrt{2}  \wlsseminorm{h - f} \\
      &\le (1+\sqrt{2}) \norm[L^\infty(\uid)]{f - h}
    \end{align*}
  for all $h \in \cA$. The statement follows by taking the infimum over all $h \in \cA$.
\end{proof}

If the ansatz space $\cA$ is spanned by an orthonormal polynomial basis, then we can bound the cost of
\Cref{alg:wls} by the arguments outlined in the following, which require $\rho$ on $\uid$ to be the $d$-fold
product of a univariate measure $\rho_1$ on $\ui$, i.e.~$\rho = \otimes_{j=1}^d \rho_1$. Denote by
$(p_r)_{r\in\N}$ the unique polynomial basis that is orthonormal w.r.t.~$L^2(\ui,\rho_1)$ and such that $p_r :
\ui \to \R$ is a polynomial of degree $r$ for all $r\in\N$. With a multi-index set $\Lambda \subset \N_0^d$,
we can then construct an orthonormal system as $(p_\bsnu)_{\bsnu \in \Lambda}$, where $p_\bsnu(\bsx)\dfn
\prod_{j=1}^d p_{\nu_j}(x_j)$. For this scenario, the following theorem bounds the cost (number of floating
point operations) of \Cref{alg:wls} in terms of (i)~the cardinality $|\Lambda|$ of the ansatz space and
(ii)~${\rm cost}(f)$~--~the cost of evaluating $f$. If ${\rm cost}(f)$ is large, the complexity of
\Cref{alg:wls} is dominated by the term $|\Lambda|\log(|\Lambda|){\rm cost}(f)$.

\begin{theorem} \label{thm:wls_cost}
  Let $\Lambda\subset\N_0^d$ be downward closed and such that $|\Lambda| \ge \sqrt{d}$ and consider an
  orthonormal system $(p_\bsnu)_{\bsnu \in \Lambda} \subset L^2(\uid,\rho)$ as above for some probability
  measure $\rho$. Then the expected value of the number of floating point operations of \Cref{alg:wls} scales
  with
    \begin{equation} \label{eq:thm_wls_cost}
      \cO(|\Lambda|^3 \log(|\Lambda|)) + \cO(|\Lambda| \log(|\Lambda|)) {\rm cost}(f)
      \qquad \text{as }|\Lambda|\to\infty.
    \end{equation}
  The constant in the Landau notation is independent of $d$.
\end{theorem}

\begin{proof}
  Denote throughout $m\dfn|\Lambda|$.

  First, note that the number of resampling steps until a suitable point set $(\bschi_k)_{k=1}^n$ has been
  found (i.e.~number of executions of \Cref{alg:wls:check_e} to \Cref{alg:wls:resampling_b}) is random, but
  from~\eqref{eq:wls_p_e} it follows that the expected number of iterations is
  \begin{equation} \label{thm:wls:proof:n_iterations}
    \sum_{j\in\N} j 2^{-j} = 2.
  \end{equation}
  Therefore, in expectation, the overall cost is simply determined by the largest cost of the following
  non-trivial steps:

  \begin{itemize}
    \item \textbf{Generating the samples} $(\bschi_k)_{k=1}^n$ (\Cref{alg:wls:sampling}). In~\cite{hajiali2020multilevel} an efficient method for sampling from $\eta$ with $f_\eta$ as defined in
    \Cref{alg:wls}, \Cref{alg:wls:f_eta}, was proposed as follows: First draw $\bsnu\in\Lambda$ uniformly
    distributed. Then generate a sample $\bschi\in \uid$ from the distribution with density $p_\bsnu(\bsx)
    ^2=\prod_{j=1}^d p_{\nu_j}(x_j)^2$ on $\uid$. The latter is achieved by drawing each $\chi_j$ according to
    the density $p_{\nu_j}^2$, for instance by rejection sampling or using the inverse CDF. Recalling
    that $n = \cO(m \log(m))$, the overall cost is
    $$
      \cO(d n) = \cO(d m \log(m)).
    $$
    \item \textbf{Assembling the Gramian Matrix} $\bsG$ (\Cref{alg:wls:gramian}). In principle, evaluating a
    single polynomial of degree $r$ requires $\cO(r)$ floating point operations. However, using the three term
    recurrence for Legendre polynomials, we can evaluate $(p_i(x))_{i=0}^r$ with complexity $O(r)$ for any
    $x\in\R$, $r\in\N$. Thus, the computation of $(p_\bsnu(\bsx))_{\bsnu\in\Lambda}$, for a downward closed
    set $\Lambda\subseteq\N_0^d$ can be achieved in cost $O(d|\Lambda|)$. Consequently, computing
    $((p_\bsnu(\bschi_k))_{k=1}^n)_{\bsnu \in \Lambda}$ costs $\cO(d n m) = \cO(d m^2 \log(m))$ floating point
    operations. Next, we can evaluate $(f_\eta(\bschi_k))_{k=1}^n$ with a cost of $\cO(n m) = \cO(m^2
    \log(m))$. All of these preparatory computations are, however, exceeded in cost by the final task of
    assembling the Gramian matrix $\bsG$ which is of
    $$
      \cO(n m^2) = \cO(m^3 \log(m)).
    $$
    \item \textbf{Evaluating the norm} $\norm{\bsG - \bsI}$ (\Cref{alg:wls:check_e}). This task amounts to
    finding the largest eigenvalue of $\bsG - \bsI \in \R^{m \times m}$, which
    generally scales with
    $$
      \cO(m^3).
    $$
    \item \textbf{Assembling the left-hand side} $\bsd$ (\Cref{alg:wls:b}). While the evaluations of $p$ and
    $f_\eta$ at all $(\bschi_k)_{k=1}^n$ can be reused from previous steps, now also the target function $f$
    has to be evaluated at all $(\bschi_k)_{k=1}^n$. The final assembly of $\bsd$ then requires $m$ summations
    of $n$ terms. Therefore, the overall cost is
    $$
      n {\rm cost}(f) + \cO(m n) = \cO(m \log(m)) {\rm cost}(f) + \cO(m^2 \log(m)).
    $$
    \item \textbf{Solving the linear system} $\bsc = \bsG^{-1} \bsd$ (\Cref{alg:wls:solve}). This task scales
    not worse than
    $$
      \cO(m^3).
    $$
  \end{itemize}

  Statement~\eqref{eq:thm_wls_cost} follows with the observation, that the largest term in $m$ is $\cO(m^3
  \log(m))$ and the only $f$-dependent term is $ \cO(m \log(m)) {\rm cost}(f)$.
\end{proof}

Note that the number of density evaluations $n$ can be chosen smaller or larger than the value set in
\Cref{alg:wls}, \Cref{alg:wls:n} in order to better balance the contribution of each term in the cost bound
\eqref{eq:thm_wls_cost}. By~\eqref{eq:P_E}, a smaller $n$ entails a higher expected number of resampling
steps. Depending on the application-specific cost of evaluating $f$, this increase in resampling steps until
the condition of \Cref{alg:wls}, \Cref{alg:wls:check_e} is met may well be acceptable.

\begin{remark}
  With recently developed subsampling techniques it is possible to reduce the number of function evaluations
  from $\cO(\abs{\Lambda} \log(\abs{\Lambda}))$ to $\cO(\abs{\Lambda})$, see
  e.g.~\cite{bartel2023constructive}, and in particular Theorem~6.3 therein.
\end{remark}

Further, merely sampling from $\eta$ as defined in \Cref{alg:wls}, \Cref{alg:wls:f_eta} may still be
inconveniently expensive. If the underlying polynomial basis consists of Chebychev polynomials, one can also
sample from the Chebychev measure instead of $\eta$. In this case we obtain the same rates both of the error
and the cost at the expense of the cost increasing slightly by a multiplicative constant as outlined in more
detail in the following remark.

\begin{remark}
  The Chebychev polynomials
    $$
      t_r : [-1,1] \to \R \text{ with } t_0(x) = 1 \text{ and }
      t_r(x) \dfn \sqrt{2} \cos(r \arccos(x)) \text{ for all } r \in \N
    $$
  are orthonormal with respect to the Chebychev measure $\rho_{\rm cheb}$ with density
    $$
      f_{\rm cheb} : [-1,1] \to \Rpz, \qquad f_{\rm cheb}(x) = \frac{1}{\pi \sqrt{1 - x^2}}
    $$
  (cf.~\cite[Equations 5.8.1 and 5.8.3]{press2007numerical}). We let $\hat t_r : [0,1] \to \R$ denote the
  scaled version on $[0,1]$, i.e.~with $\varphi : [0,1] \rightarrow [-1,1], \varphi(x) \dfn 2x-1$ we let
  $\hat{t}_r \dfn t_r \circ \varphi$. These are then orthonormal with respect to the measure with density
    $$
      \hat f_{\rm cheb} : [0,1] \to \Rpz, \qquad
      \hat f_{\rm cheb}(x) \dfn (f_{\rm cheb} \circ \varphi) (x) \varphi'(x)
      = \frac{2}{\pi \sqrt{1 - (2x-1)^2}}.
    $$
  Both versions satisfy $\norm[L^\infty]{t_r} = \norm[L^\infty]{\hat t_r} \le \sqrt{2}$. As usual, we can
  construct a multivariate basis from products of the univariate basis functions and a finite multi-index set
  $\Lambda \subset \N_0^d$.

  Now, for Chebychev polynomials and the Chebychev measure as sampling measure, we can bound the quantity
  $K_{\rm cheb}((\hat t_\bsnu)_{\bsnu \in \Lambda})$ (cf.~\eqref{eq:Kw}) by
    \begin{align*}
      K_{\rm cheb}((\hat t_\bsnu)_{\bsnu \in \Lambda})
      &=   \norm[{L^\infty(\uid)}]{\hat f_{\rm cheb}^{-1} \sum_{\bsnu\in\Lambda} \hat t_\bsnu^2}\\
      &\le \norm[{L^\infty(\uid)}]{\hat f_{\rm cheb}^{-1}}
      \sum_{\bsnu\in\Lambda} \prod_{j=1}^d \norm[{L^\infty(\uid)}]{\hat t_{\nu_j}^2}
      = \norm[{L^\infty([0,1]^d)}]{\hat f_{\rm cheb}^{-1}} \abs{\Lambda} \sqrt{2}^d.
    \end{align*}
  Since
    \begin{equation*}
      \norm[{L^\infty(\uid)}]{\hat f_{\rm cheb}^{-1}}
      = \norm[{L^\infty(\uid)}]{\left(\frac{\pi}{2} \right)^d \prod_{j=1}^d \sqrt{1 - (2 x_j - 1)^2}}
      \le \left(\frac{\pi}{2}\right)^d
    \end{equation*}
  it holds that
    \begin{equation*}
      K_{\rm cheb}((\hat t_\bsnu)_{\bsnu \in \Lambda}) \le \left(\frac{\pi}{\sqrt{2}}\right)^d \abs{\Lambda}.
    \end{equation*}
  Consequently, choosing $n = 10 (\pi/\sqrt{2})^d \abs{\Lambda} \lceil \log (4 |\Lambda|) \rceil$ (i.e.~by a
  factor of $(\pi/\sqrt{2})^d$ larger than in \Cref{thm:wls} and \Cref{alg:wls}) allows us to obtain the same
  error rate~\eqref{eq:thm_wls_bd} despite sampling from the Chebychev measure $\rho_{\rm cheb}$ instead of
  $\eta$. The bound on the cost~\eqref{eq:thm_wls_cost}, however, will increase by the factor
  $(\pi/\sqrt{2})^d$.

  Finally, note that when using the error bound~\eqref{eq:thm_wls_bd} to obtain a bound on the Hellinger
  distance as in \Cref{thm:main1}, the bound in the $L^2(\uid,\rho_{\rm cheb})$ norm has to be transformed
  into a bound in the $L^2(\uid,\mu)$ norm. Since $\norm[L^2(\uid, \mu)]{\cdot} \le
  \norm[L^\infty(\uid)]{\frac{\dd \mu}{\dd \rho_{\rm cheb}}}\norm[L^2(\uid, \rho_{\rm cheb})]{\cdot}$ this
  yields an additional constant factor $\norm[L^\infty(\uid)]{\frac{\dd \mu}{\dd \rho_{\rm cheb}}} =
  \norm[L^\infty(\uid)]{f^{-1}_{\rm cheb}} = (\pi/2)^d$.
\end{remark}

\subsection{Error bounds in Wasserstein and Hellinger distance} \label{sec:error_and_cost}

In this section we derive convergence rates for target densities with square roots in \ckalpha (see
\Cref{thm:error_c_k_alpha}) and with analytic square roots (see \Cref{thm:error_analytic}). Together, these
results prove \Cref{thm:main2}. For this purpose, we first derive a bound on the Hellinger distance in terms
of the $L^2$-norm of the distance between the unnormalized densities (see \Cref{lma:bound_approx})~--~which
corresponds to the error minimized (approximately) by the weighted least-squares method.

Going beyond the scope of \Cref{thm:main2}, we further show that~--~up to the exponent $1/p$~--~the same
convergence is obtained for the error measured in the $p$-Wasserstein distance. This follows trivially after
bounding the $p$-Wasserstein in terms of the Hellinger distance as shown in \Cref{lma:wasserstein_hellinger}.

For this purpose, we consider $\uid$ equipped with the Euclidean distance $\norm[2]{\cdot}$ and two
probability measures $\measi$ and $\measii$ on $\uid$ that are absolutely continuous with respect to the
Lebesgue measure. Recall that for $p \in [1, \infty)$, the $p$-Wasserstein distance between $\measi$ and
$\measii$ is given as
  \begin{equation} \label{eq:def_wasserstein}
    \wasserstein{p}(\measi, \measii) \dfn \inf \limits_{\gamma \in \Gamma(\measi, \measii)}
    \left( \int \limits_{\uid \times \uid} \norm[2]{\bsx-\bsy}^p \dd \gamma(\bsx,\bsy) \right)^{1/p},
  \end{equation}
where $\Gamma(\measi, \measii)$ denotes the set of all measures on $\uid \times \uid$ with marginals $\measi$
and $\measii$ (see e.g.~\cite[Definition 6.1]{villani2008optimal}).

\begin{lemma}\label{lma:wasserstein_hellinger}
  For $\measi$ and $\measii$ as above it holds
    \begin{equation}\label{eq:wasserstein_hellinger}
      \wasserstein{p}(\measi,\measii)
      \leq \left(\frac{d}{2^{1/p}}\right)^{1/2} \hellinger(\measi,\measii)^{1/p}.
    \end{equation}
\end{lemma}

\begin{proof}
  The statement can be shown using the total variation distance, which is defined as
    \begin{equation*}
      \totalvar(\measi,\measii) \dfn \norm[L^1(\uid)]{f_\measi - f_\measii}
    \end{equation*}
  (see e.g.~\cite[Definition 3.1.8]{ruschendorf2014mathematische}).
  For the unit hypercube and $\bsx^* \dfn (\frac{1}{2})_{i=1}^d$ it holds that
  \begin{equation*}
    \sup_{\bsx \in \uid} \norm[2]{\bsx^* - \bsx}^p = \left(\frac{d}{2^{2}}\right)^{p/2}.
  \end{equation*}
  This allows us to apply~\cite[Theorem~6.15]{villani2008optimal} and obtain
  \begin{equation*}
    \wasserstein{p}(\measi,\measii)
    \leq 2^{1-1/p} \left(\left(\frac{d}{2^{2}}\right)^{p/2} \totalvar(\measi, \measii)\right)^{1/p}
    = \frac{d^{1/2}}{2^{1/p}} \totalvar(\measi, \measii)^{1/p}.
  \end{equation*}
  The statement then follows with $\totalvar(\measi,\measii) \le \sqrt{2} \hellinger(\measi,\measii)$, see
  e.g.~\cite[Proposition 3.1.10]{ruschendorf2014mathematische}.
\end{proof}

\begin{lemma}\label{lma:bound_approx}
  For $\measi$ and $\measii$ as above it holds that
    \begin{equation*}
      \hellinger(\measi,\measii)
      \le \frac{2}{\sqrt{c_{\measii}}} \norm[L^2(\uid)]{\sqrt{\hat{f}_\measi} - \sqrt{\hat{f}_\measii}}.
    \end{equation*}
\end{lemma}

\begin{proof}
  For ease of notation, set $\norm{\cdot} \dfn \norm[L^2(\uid)]{\cdot}$.
  We can bound the Hellinger distance by
  \begin{align*}
    \norm{\sqrt{f_\measi} - \sqrt{f_\measii}}
    &
    = \norm{\sqrt{\frac{\hat{f}_\measi}{c_\measi}} - \sqrt{\frac{\hat{f}_\measii}{c_\measii}}}
    = \frac{\norm{\sqrt{\hat{f}_\measi c_\measii} - \sqrt{\hat{f}_\measii c_\measi}}}{\sqrt{c_\measi c_\measii}}\\
    &
    \le \frac{\norm{\sqrt{\hat{f}_\measi c_\measii} - \sqrt{\hat{f}_\measi c_\measi}} + \norm{\sqrt{\hat{f}_\measi c_\measi} - \sqrt{\hat{f}_\measii c_\measi}}}{\sqrt{c_\measi c_\measii}}\\
    &
    =\frac{\norm{\sqrt{\hat{f}_\measi}}}{\sqrt{c_\measi c_\measii}} \abs{\sqrt{c_\measii} - \sqrt{c_\measi}} + \frac{1}{\sqrt{c_\measii}} \norm{\sqrt{\hat{f}_\measi} - \sqrt{\hat{f}_\measii}}\\
    &
    = \frac{1}{\sqrt{c_\measii}} \left( | \norm{\sqrt{\hat{f}_\measii}} - \norm{\sqrt{\hat{f}_\measi}}| + \norm{\sqrt{\hat{f}_\measi } - \sqrt{\hat{f}_\measii}}\right)\\
    &
    \le \frac{2}{\sqrt{c_\measii}} \norm{\sqrt{\hat{f}_\measi} - \sqrt{\hat{f}_\measii}},
  \end{align*}
  where the last inequality follows by an application of the reverse triangle inequality.
\end{proof}

\subsubsection{Target densities with square roots in \ckalpha}{}
\label{sec:error_analysis_cka_densities}

Let $k \in \N_0$ and $\alpha \in [0,1]$. For $f : \uid \rightarrow \R$, define
  \begin{equation*}
    |f|_{\ckalpha}
    \dfn \sup \limits_{\set{\bsgamma\in \N_0^d}{\norm[1]{\bsgamma} = k}}
    \left( \sup_{\substack{\bsx,\bsy \in \uid\\ \bsx \neq \bsy}}
    \frac{|\partial^\bsgamma f(\bsx) - \partial^\bsgamma f(\bsy)|}{\norm[2]{\bsx-\bsy}^\alpha} \right)
     \in [0,\infty]
  \end{equation*}
and
  \begin{equation*}
    \ckalpha \dfn \setc{ f \in C^k(\uid)}{|f|_{\ckalpha} < \infty }.
  \end{equation*}
For $\ell \in \Rp$ we set
  \begin{equation}\label{eq:def_lambda_ell}
    \Lambda_{\ell} \dfn \set{\bsnu \in \N_0^d}{\norm[1]{\bsnu} < \ell}.
  \end{equation}

\begin{theorem}\label{thm:error_c_k_alpha}
  Let $\ftaru : \uid \to \Rpz$ be such that $\ftarus \in \ckalpha$ for some $k \in \N_0$
  and $\alpha \in [0,1]$. Then it holds for all $\ell \in \N$ and $p \in [1,\infty)$ that $T^\ell \dfn
  \alg{\ref{alg:main}}(\ftaru,\Lambda_\ell)$ satisfies
    \begin{align} \label{eq:d_h_c_k_alpha}
      \hellinger(\pi, T_\sharp^\ell \mu)
      &\le C_{\rm H} |\ftars|_{\ckalpha} |\Lambda_\ell|^{-\frac{k+\alpha}{d}}\\
      \wasserstein{p}(\pi, T_\sharp^\ell \mu)
      &\le C_{\rm W} \left(|\ftars|_{\ckalpha} |\Lambda_\ell|^{-\frac{k+\alpha}{d}} \right)^{1/p} \label{eq:d_w_c_k_alpha}
    \end{align}
  with constants $C_{\rm H} = C_{\rm H}(d,k)$ and $C_{\rm W} = C_{\rm W}(d,k,p)$.
\end{theorem}

\begin{proof}
  By construction, $T_\sharp^\ell \mu$ is a measure with unnormalized density $g_\ell^2$ where $g_\ell =
  \alg{\ref{alg:wls}}(\ftarus,\Lambda_\ell)$. Combining the WLS error~\eqref{eq:thm_wls_bd} with the
  error of best polynomial approximation~\eqref{eq:eba_c_k_alpha} for functions in \ckalpha yields
    \begin{equation*}
      \norm[L^2(\uid)]{\ftarus - g_\ell}
      \le (1+\sqrt{2}) C(d, k) |\ftarus|_{\ckalpha} |\Lambda_\ell|^{-\frac{k+\alpha}{d}}.
    \end{equation*}
  An application of \Cref{lma:bound_approx} then implies~\eqref{eq:d_h_c_k_alpha} via
    \begin{equation*}
      \hellinger(\pi, T_\sharp^\ell \mu)
      \le 2 (1+\sqrt{2}) C(d, k) \frac{|\ftarus|_{\ckalpha}}{\sqrt{c_\pi}} |\Lambda_\ell|^{-\frac{k+\alpha}{d}}
      = C_{\rm H} |\ftars|_{\ckalpha} |\Lambda_\ell|^{-\frac{k+\alpha}{d}}
    \end{equation*}
  with $C_{\rm H} = C_{\rm H}(d,k) \dfn 2 (1+\sqrt{2}) C(d, k).$ Together
  with~\eqref{eq:wasserstein_hellinger} we obtain
    \begin{equation*}
      \wasserstein{p}(\pi, T_\sharp^\ell \mu)
      \leq C_{\rm W} \left( |\ftars|_{\ckalpha} |\Lambda_\ell|^{-\frac{k+\alpha}{d}} \right)^{1/p}
    \end{equation*}
  with $C_{\rm W} = C_{\rm W}(d,k,p) \dfn (\frac{d}{2^{1/p}})^{1/2} C_{\rm H}(d,k)^{1/p}$, which
  shows~\eqref{eq:d_w_c_k_alpha}.
\end{proof}

\subsubsection{Target densities with analytic square roots} \label{sec:error_analysis_analytic_densities}

We now come to the case where $\ftarus$ is an analytic function. Specifically, we'll work under the
following assumption:

\begin{assumption}\label{asm:analytic}
  The function $f : \uid \to \R$ admits a uniformly bounded holomorphic extension to the polyellipse
  $\cE_\bsrho \dfn \times_{j=1}^d \cE_{\rho_j} \subseteq \C^d$ where $\bsrho \dfn (\rho_j)_{j=1}^d \in (1,
  \infty)^d$ and
  \begin{equation*}
    \cE_{\rho_j} \dfn \left\{\frac{\frac{z+1}{2}+\frac{2}{z+1}}{2} : z \in \C, 1 \leq |z| < \rho_j \right\}
    \subset \C.
  \end{equation*}
\end{assumption}

The so-called Bernstein ellipse $\cE_\rho$ forms a complex open set around the interval $\ui$ in the complex
plane. As $\rho \searrow 1$, the ellipse collapses to the interval $\ui$. In particular, for any function that
is analytic on an open complex set containing $\uid$, there always exists $\bsrho$ such that
\Cref{asm:analytic} holds.
The specific values of $\bsrho$ can be used to construct customized ansatz spaces defined as
  \begin{equation}\label{eq:def_lambda_bsrho_ell}
    \Lambda_{\bsrho,\ell} \dfn \set{\bsnu \in \N_0^d}{\sum_{j=1}^d \nu_j \log(\rho_j) < \ell},
  \end{equation}
for $\ell \in \Rp$. Using classic approximation theory, they yield exponential convergence rates as stated in
the following theorem.

\begin{theorem} \label{thm:error_analytic}
  Let $\ftaru : \uid \to \Rpz$ be such that $\ftarus$ satisfies \Cref{asm:analytic} for some $\bsrho \in (1, \infty)^d$. Let $\beta \in \Rp$ be
  such that $\beta < ( d! \prod \limits_{j=1}^d \log(\rho_j))^{1/d}$. Then it holds for all $\ell \in \N$ that
  $T^\ell \dfn \alg{\ref{alg:main}}(\ftaru,\Lambda_{\bsrho,\ell})$ satisfies
    \begin{align*}
      \hellinger(\pi, T_\sharp^\ell \mu)
      &\le C_{\rm H} \norm[L^\infty(\cE_\bsrho )]{\ftars} \exp\left(-\beta |\Lambda_{\bsrho,\ell}|^{1/d}\right)\\
      \wasserstein{p}(\pi, T_\sharp^\ell \mu)
      &\le C_{\rm W} \norm[L^\infty(\cE_\bsrho )]{\ftars}^{1/p} \exp\left(-\frac{\beta}{p} |\Lambda_{\bsrho,\ell}|^{1/d}\right)
    \end{align*}
  with constants $C_{\rm H} = C_{\rm H}(d, \bsrho, \beta)$ and $C_{\rm W} = C_{\rm W}(d, \bsrho, \beta, p)$.
\end{theorem}

\begin{proof}
  If $\ftarus$ satisfies \Cref{asm:analytic}, then~\cite[Theorem~3.5]{opschoor2021exponential}
  implies that there exists a constant $C = C(d, \bsrho, \beta)$ such that
    \begin{equation} \label{eq:error_holomorphic_density}
      e_{\Lambda_{\bsrho,\ell}} \left( \ftarus\right)_\infty
      \le C \norm[L^\infty(\cE_\bsrho )]{\ftarus} e^{-\beta |\Lambda_{\bsrho,\ell}|^{1/d}}.
    \end{equation}
  We can now proceed analogous to the proof of \Cref{thm:error_c_k_alpha}: By
  combining~\eqref{eq:error_holomorphic_density} with~\eqref{eq:thm_wls_bd} and \Cref{lma:bound_approx} we
  obtain
    \begin{equation*}
      \hellinger(\pi, \tilde{T}_\sharp^\ell \mu)
      \le 2 (1+\sqrt{2}) C \frac{\norm[L^\infty(\cE_\bsrho )]{\ftarus}}{\sqrt{c_{\pi}}} e^{-\beta |\Lambda_{\bsrho,\ell}|^{1/d}}
      = C_{\rm H} \norm[L^\infty(\cE_\bsrho)]{\ftars} e^{-\beta |\Lambda_{\bsrho,\ell}|^{1/d}}.
    \end{equation*}
  Thus, the statement holds with $C_{\rm H} \dfn 2 (1+\sqrt{2}) C(d, \bsrho, \beta)$ and $C_{\rm W} \dfn
  (\frac{d}{2^{1/p}})^{1/2} C_{\rm H}^{1/p}$.
\end{proof}

\section{Implementation and cost of evaluating $S$ and $T$} \label{sec:implementation_and_sampling_cost}

Recall that our algorithm for approximating transport maps involves two main steps. Initially, we calculate a
density surrogate through a least-squares method, as per \Cref{alg:wls}. This requires the evaluation of the
(unnormalized) target density $\ftaru$ at the evaluation points $(\bschi_k)_{k=1}^n\subset\uid$. Assuming
each such evaluation to be costly, this will represent the dominant portion of the overall cost, tantamount to
$n$ times the cost of a single density evaluation, cf.~\Cref{thm:wls_cost}. The second step, described in
\Cref{alg:S}, is centered on the construction of the exact transport corresponding to the normalized surrogate
density. In this section, we discuss the efficient implementation of this stage and examining the
computational cost of each such transport evaluation (presuming that the surrogate density has been
precomputed in an offline phase).

\subsection{Leveraging orthonormality to evaluate $S$}

By our choice of surrogate construction method (outlined in \Cref{sec:constrdens}) and our choice of
ansatz space (cf.~\eqref{eq:PLambda}) the density surrogate will be a multivariate polynomial of the form
  \begin{equation*}
    g(\bsx)^2 = \sum_{\bsnu, \bsmu \in \Lambda} c_\bsnu c_\bsmu L_\bsnu (\bsx) L_\bsmu (\bsx).
  \end{equation*}
Since an evaluation of $S$ requires integrating the density surrogate, a naive implementation would scale as
$\cO(\abs{\Lambda}^2)$. With $I_{\nu_i \mu_j}(x) \dfn \int_{0}^{x} L_{\nu_i}(t) L_{\mu_j}(t) \dd t$ and
$I_{\from{\bsnu}{j} \from{\bsmu}{j}} \dfn \prod_{i=j}^d I_{\nu_i \mu_i}$, we have $S_j(\upto{\bsx}{j}) =
r_j(\upto{\bsx}{j}) / s_{j-1}(\upto{\bsx}{j-1})$ where
\begin{equation}
  \begin{alignedat}{3} \label{eq:r_j_s_j}
    r_j(\upto{\bsx}{j}) &\dfn \sum \limits_{\bsnu, \bsmu \in \Lambda} c_\bsnu c_\bsmu
    L_{\upto{\bsnu}{j-1}}(\upto{\bsx}{j-1}) L_{\upto{\bsmu}{j-1}}(\upto{\bsx}{j-1})
    I_{\nu_j \mu_j}(x_j) I_{\from{\bsnu}{j+1} \from{\bsmu}{j+1}}({\boldsymbol 1}),\\
    s_{j-1}(\upto{\bsx}{j-1}) &\dfn \sum \limits_{\bsnu, \bsmu \in \Lambda} c_\bsnu c_\bsmu
    L_{\upto{\bsnu}{j-1}}(\upto{\bsx}{j-1}) L_{\upto{\bsmu}{j-1}}(\upto{\bsx}{j-1})
    I_{\from{\bsnu}{j} \from{\bsmu}{j}}({\boldsymbol 1}).
  \end{alignedat}
\end{equation}
Examining the above expression, we see that (i)~$I_{\from{\bsnu}{j} \from{\bsmu}{j}}({\boldsymbol 1}) = 0
\Leftrightarrow \nu_k \neq \mu_i$ for any $i,k \ge j$ due to orthonormality and of the $L_j$ and (ii)~the
individual terms used for the computation $s_j$ can be reused for the computation of $s_{j+1}$.

To formalize this idea and derive an efficient algorithm to evaluate $(S_j(\upto{\bsx}{j}))_{j=1}^d$, we
introduce some notation. Let
  $$
    \from{\Lambda}{j} \dfn
    \{\from{\bsnu}{j} \ | \ \exists \upto{\bsnu}{j-1} : (\upto{\bsnu}{j-1}, \from{\bsnu}{j}) \in \Lambda \}
  $$
be the set of all unique \textit{tails} of length $d-j+1$ of multi-indices from the set $\Lambda$ and let
  $$
    \from{\Lambda}{j}(\from{\bsnu}{j+1})
    \dfn \{ \from{\bsmu}{j} \in \from{\Lambda}{j} \ | \ \from{\bsmu}{j+1} = \from{\bsnu}{j+1}\}
  $$
be the subset of tails $\from{\bsmu}{j} \in \from{\Lambda}{j}$ that differ only in the $j$-th coordinate and
equal $\from{\bsnu}{j+1}$ otherwise. Note that we can construct $\from{\Lambda}{j}$ as $\from{\Lambda}{j} =
\cup_{\from{\bsnu}{j+1} \in \from{\Lambda}{j+1}} \from{\Lambda}{j}(\from{\bsnu}{j+1}).$ For consistency, we
define $\from{\bsnu}{d+1} \dfn \emptyset$ and $\from{\Lambda}{d}(\from{\bsnu}{d+1}) =
\from{\Lambda}{d}(\emptyset) \dfn \from{\Lambda}{d}$. Let further
  $$
    \upto{\Lambda}{j}(\from{\bsnu}{j+1})
    \dfn \{ \upto{\bsnu}{j} \ | \ (\upto{\bsnu}{j}, \from{\bsnu}{j+1}) \in \Lambda \}
  $$
be the set of all \textit{heads} of multi-indices from $\Lambda$ that have $\from{\bsnu}{j+1}$ as tail.
Now, for all $j \in \{1, \dots, d\}$ and $\from{\bsnu}{j}$ we define the quantity
  $$
    \gamma_{\from{\bsnu}{j}} (\upto{\bsx}{j-1})
    \dfn \sum \limits_{\upto{\bsnu}{j-1} \in \upto{\Lambda}{j-1}(\from{\bsnu}{j})} c_{(\upto{\bsnu}{j-1}, \from{\bsnu}{j})} L_{\upto{\bsnu}{j-1}}(\upto{\bsx}{j-1}).
  $$
Note that we can construct the sequence $(\gamma_\bsnu)_{\bsnu\in\Lambda}$ recursively as
  $$
    \gamma_{\from{\bsnu}{j+1}}(\upto{\bsx}{j})
    = \sum \limits_{\from{\bsnu}{j} \in \from{\Lambda}{j}(\from{\bsnu}{j+1})} \gamma_{\from{\bsnu}{j}} (\upto{\bsx}{j-1}) L_{\nu_j}(x_j)
  $$
with $\gamma_{\from{\bsnu}{1}} = \gamma_{\bsnu} = c_\bsnu$. With these concepts at hand we can
express~\eqref{eq:r_j_s_j} as
  \begin{align*}
    r_j(\upto{\bsx}{j})
    &= \sum \limits_{\from{\bsnu}{j+1}\in\from{\Lambda}{j+1}}
    \sum \limits_{\substack{\from{\bsmu}{j},\from{\bskappa}{j}\\ \in \from{\Lambda}{j}(\from{\bsnu}{j+1})}}
    \gamma_{\from{\bsmu}{j}} (\upto{\bsx}{j-1}) \gamma_{\from{\bskappa}{j}} (\upto{\bsx}{j-1}) I_{\mu_j,\kappa_j}(x_j) \\
    s_{j-1}(\upto{\bsx}{j-1})
    &= \sum \limits_{\from{\bsnu}{j}\in \from{\Lambda}{j}} \gamma_{\from{\bsnu}{j}}(\upto{\bsx}{j-1})^2.
  \end{align*}
The overall procedure of evaluating $S(\bsx)$ is summarized below in \Cref{alg:S_impl}.

\begin{algorithm}
  \caption{Evaluate $S(\bsx)$ \\
    \emph{Input:} $\Lambda \subset \N_0^d$,
                  $(c_\bsnu)_{\bsnu \in \Lambda} \subset \R^{\abs{\Lambda}}$,
                  $\bsx \in \uid$\\
    \emph{Output:} $S(\bsx) \in \uid$}\label{alg:S_impl}
  \begin{algorithmic}[1]
    \State $(\gamma_\bsnu)_{\bsnu\in\Lambda}=(c_\bsnu)_{\bsnu\in\Lambda}$
    \State $s_0 = \sum_{\bsnu \in \Lambda} (\gamma_\bsnu)^2$
    \For{$j \in \{1,\dots,d\}$}
      \State $m_j = \sup \set{\nu_j}{\bsnu \in \Lambda}$
      \State $I = (\int_{-1}^{x_j} L_k (t) L_\ell (t) \dd t)_{k,\ell \in \{1, \dots, m_j \}}$ \label{alg:S_impl:I}
      \State $r_j, s_j = 0 $
      \For{$\from{\bsnu}{j+1}\in \from{\Lambda}{j+1}$} \label{alg:S_impl:inner_loop_start}
        \State $r_j \pluseq \sum \limits_{\from{\bsmu}{j},\from{\bskappa}{j} \in \from{\Lambda}{j}(\from{\bsnu}{j+1})} \gamma_{\from{\bsmu}{j}} \gamma_{\from{\bskappa}{j}} I_{\mu_j,\kappa_j}$
        \If{$j < d$}
          \State $\gamma_{\from{\bsnu}{j+1}} = \sum \limits_{\from{\bsmu}{j} \in \from{\Lambda}{j}(\from{\bsnu}{j+1})} \gamma_{\from{\bsmu}{j}} L_{\nu_j}(x_j) $
          \State $s_j \pluseq (\gamma_{\from{\bsnu}{j+1}})^2$
        \EndIf
      \EndFor \label{alg:S_impl:inner_loop_end}
      \State $S_j = r_j / s_{j-1}$
    \EndFor
    \State
    \Return $(S_j)_{j=1}^d$
  \end{algorithmic}
\end{algorithm}

For the index sets used in \Cref{sec:error_and_cost}, \Cref{alg:S_impl} reduces the cost of one evaluation
of $S$ to from $\cO(\abs{\Lambda}^2)$ in a naive implementation to $\cO(\abs{\Lambda}^{1+\frac{1}{d}})$, as we
show in the following \Cref{thm:S_impl}. To this end, first note that both the set $\Lambda_\ell$ as defined
in~\eqref{eq:def_lambda_ell} and the set $\Lambda_{\bsrho, \ell}$ as defined
in~\eqref{eq:def_lambda_bsrho_ell} can be expressed as
  \begin{equation}\label{eq:def_lambda_bsk_ell}
    \Lambda_{\bsk, \ell} \dfn \set{\bsnu \in \N_0^d}{\bsnu \cdot \bsk < \ell},
  \end{equation}
with $\bsk = (1,1,\dots,1) \in \R^d$ in the case of $\Lambda_\ell$ and $\bsk \dfn (\log(\rho_1),\log(\rho_2),
\dots, \log(\rho_d)) \in \R^d$ in case of $\Lambda_{\bsrho, \ell}$.

\begin{theorem} \label{thm:S_impl}
  For $\ell \in \Rp$ and $\bsk \in \Rp^d$ let $\Lambda_{\bsk,\ell}$ be as in~\eqref{eq:def_lambda_bsk_ell}.
  Denote $k_{\rm min} \dfn \min \set{k_j}{j \in \{1, \dots, d\}}$ and $k_{\rm max} \dfn \max \set{k_j}{j \in
  \{1, \dots, d\}}$. Then,
    \begin{equation*}
      \cost{\ref{alg:S_impl}} = \cO\left(
      \left(\frac{k_{\rm max}}{k_{\rm min}}\right)^{d+1} |\Lambda|_{\bsk,\ell}^{1+\frac{1}{d}} \right)
    \end{equation*}
  in terms of floating point operations as $\ell \to \infty$. The constant in the Landau notation is
  independent of $\bsk$, $\ell$ and $d$.
\end{theorem}

\begin{proof}
  Denote in the following $\Lambda \dfn \Lambda_{\bsk,\ell}$. The cost of the inner for-loop in
  \Cref{alg:S_impl:inner_loop_start} to \Cref{alg:S_impl:inner_loop_end} scales with
  $\cO(\sum_{\from{\bsnu}{j+1}\in \from{\Lambda}{j+1}} |\from{\Lambda}{j}(\from{\bsnu}{j+1}) |^2)$. This is
  greater than or equal to the preceding term on \Cref{alg:S_impl:I}, which is of cost $\cO(m_j^2)$,
  where $m_j=\sup\set{\nu_j}{\bsnu\in\Lambda}$, since for at least one $\from{\bsnu}{j+1}\in
  \from{\Lambda}{j+1}$ it holds that $|\from{\Lambda}{j}(\from{\bsnu}{j+1}) | = m_j$. Therefore, the overall
  cost scales like
  \begin{equation*}
    \cost{\ref{alg:S_impl}} =
    \cO\left(\sum \limits_{j=1}^d \sum \limits_{\from{\bsnu}{j+1}\in \from{\Lambda}{j+1}}
    |\from{\Lambda}{j}(\from{\bsnu}{j+1})|^2\right).
  \end{equation*}
  For $\Lambda$ as in~\eqref{eq:def_lambda_bsk_ell} it holds for all $j \in \{1, \dots, d\}$ that
  \begin{itemize}
    \item $|\from{\Lambda}{j}(\from{\bsnu}{j+1}) | \le \ell / k_{\rm min}$ for all $\from{\bsnu}{j+1}\in
      \from{\Lambda}{j+1}$ and
    \item $|\from{\Lambda}{j+1}| \le (\ell / k_{\rm min})^{d-j}$.
  \end{itemize}
  Thus,
  \begin{align*}
    \cost{\ref{alg:S_impl}}
    &= \cO\left(\sum \limits_{j=1}^d\sum \limits_{\from{\bsnu}{j+1} \in \from{\Lambda}{j+1}} |\from{\Lambda}{j}(\from{\bsnu}{j+1}) |^2\right)
    = \cO\left(\sum \limits_{j=1}^d (\ell / k_{\rm min})^{d-j} (\ell / k_{\rm min})^2\right) \\
    &= \cO\left( \frac{(\ell / k_{\rm min})^{2} ((\ell / k_{\rm min})^{d} - 1)}{\ell / k_{\rm min} - 1}\right)
    = \cO\left((\ell / k_{\rm min})^{d+1}\right).
  \end{align*}
  By observing that $\abs{\Lambda} \ge (\ell / k_{\rm max} )^d$ we find that
  \begin{align*}
    \cost{\ref{alg:S_impl}}
    = \cO \left(\left(\frac{k_{\rm max}}{k_{\rm min}}\right)^{d+1} \abs{\Lambda}^{1+\frac{1}{d}}\right).
  \end{align*}
\end{proof}

\begin{remark}
  Multi-index sets of the form~\eqref{eq:def_lambda_bsk_ell} can be constructed efficiently via recursion over
  the dimension $d$. The following algorithm constructs an index set $\Lambda_{\bsk, \ell}$ as $\Lambda_{\bsk,
  \ell} = \alg{\ref{alg:construct_lambda}}(\bsk, \ell, ())$ with complexity of $\cO(\abs{\Lambda_{\bsk,
  \ell}})$. Note that the recursion is started with $j=0$ and a \enquote{0-dimensional} multi-index $\bsnu =
  ()$.
  \begin{algorithm}[H]
    \caption{Construct $\Lambda_{\bsk, \ell}$ \\
      \emph{Input:} $\bsk \in \R^{d-j}$,
                    $\ell \in \Rp$,
                    $\bsnu \in \N_0^{j}$ for some $j \in \{0,\dots,d\}$\\
      \emph{Output:} $\Lambda_{\bsk, \ell} \subset \N_0^{d-j}$}\label{alg:construct_lambda}
    \begin{algorithmic}[1]
      \If{$j == d$}
        \State
        \Return $\{\bsnu\}$
      \EndIf
      \State $\tilde{\Lambda} = \{\}$
      \State $m = \lfloor\frac{\ell}{k_1}\rfloor - \delta_{\frac{\ell}{k_1},\lfloor\frac{\ell}{k_1}\rfloor}$
      \Comment{Here, $\delta$ is the Kronecker delta.}
      \For{$i \in \{0, \dots, m\}$}
        \State{$\tilde{\Lambda} = \tilde{\Lambda} \cup \alg{\ref{alg:construct_lambda}}(\from{\bsk}{2}, \ell-i k_1, (\bsnu,i))$}
      \EndFor
      \State
      \Return $\tilde{\Lambda} $
    \end{algorithmic}
  \end{algorithm}
\end{remark}

\subsection{Inverting $S$ via bisection}

It remains to discuss the computation of $T=S^{-1}$ from $S$. To this end, we use a standard univariate
bisection algorithm, see \Cref{alg:bisection}.

\begin{algorithm}[H]
  \caption{Approximate $S^{-1}(y)$ for $S:\ui\to\ui$\\
    \emph{Input:} $S : \ui \to \ui$ strictly monotonically increasing and bijective, $y \in \ui$, $r\in \N$\\
    \emph{Output:} $x \in \ui$ such that $\abs{x - S^{-1}(y)} \le 2^{-r}$ }
  \label{alg:bisection}
  \begin{algorithmic}[1]
    \State $x = \frac{1}{2}$, $a=0$, $b=1$
      \For{$\_ \in \{1, \dots, r\}$}
        \If{$S(x) > y$} \label{alg:T_impl:condition}
          \State $b = x$
        \Else
          \State $a = x$
        \EndIf
        \State $x = \frac{a+b}{2}$
      \EndFor
    \State
    \Return $x$
  \end{algorithmic}
\end{algorithm}

In order to invert a multivariate KR map $S:\uid\to\uid$, we apply this method iteratively to each component
$S_j$, considered as a univariate function in the last variable, from $j=1$ to $d$:

\begin{algorithm}[H]
  \caption{Approximate $S^{-1}(\bsy)$ for $S:\uid\to\uid$\\
  \emph{Input:} $S : \uid \to \uid$, $\bsy \in \uid$, $r \in \N$\\
  \emph{Output:} $\tilde \bsx \in \uid$ such that $\norm[1]{\tilde \bsx - S^{-1}(\bsy)} = \cO(2^{-r})$}
  \label{alg:T_impl}
  \begin{algorithmic}[1]
    \For{$j \in \{1,\dots,d\}$}
      \State $S_{\upto{\tilde \bsx}{j-1}}(x_j) \dfn S_j(\upto{\tilde \bsx}{j-1}, x_j)$
      \State $\tilde x_j = \alg{\ref{alg:bisection}}(S_{\upto{\tilde \bsx}{j-1}},y_j,r)$
    \EndFor
    \State
    \Return $\tilde \bsx$
  \end{algorithmic}
\end{algorithm}

The error analysis of \Cref{alg:T_impl} has to take into account that when inverting $S_j(\upto{\bsx}{j-1},
x_j)$ as a function of $x_j$, the errors on $x_i$ for $i \in \{1,\dots, j-1\}$ will further deteriorate the
bisection result. In order to bound the resulting error, we have to assume that the density $f$, from which
the KR transport $S$ is constructed, belongs to $C^1(\uid;\Rp)$. If $f \in C^1(\uid;\Rp)$, then it follows
immediately that
  \begin{align} \label{eq:ass:bisection_density:a}
    M(f) &\dfn \min_{\bsx\in\uid} f(\bsx),\\ \label{eq:ass:bisection_density:b}
    L(f) &\dfn \frac{2}{M(f)} \ \max_{\bsx \in \uid} \max_{j \in \{1, \dots, d\}} |\partial_{x_j} f(\bsx)|
  \end{align}
are such that $0 < M(f), L(f) < \infty$. Under this assumption, error and cost of the multivariate bisection
can be bound as follows.

\begin{theorem} \label{thm:T_impl}
  Consider a (potentially unnormalized) density $f \in C^1(\uid;\Rp)$. Let $S=\alg{\ref{alg:S}}(f)$,
  $\bsy \in \uid$ and $r \in \N$. Then, $\tilde{\bsx} = \alg{\ref{alg:T_impl}}(S, \bsy, r)$ satisfies
    \begin{equation} \label{thm:T_impl:error}
      \norm[1]{\tilde \bsx - S^{-1}(\bsy)} \le C 2^{-r},
    \end{equation}
  with $C = C(d, f) = (L(f)+1)^{d}$ with $L(f)$ as in~\eqref{eq:ass:bisection_density:b}.
  Further, the number of floating point operations required by \Cref{alg:T_impl} scales with
    \begin{equation*}
      \cost{\ref{alg:T_impl}} = \cO \left(r \cost{\ref{alg:S}} \right).
    \end{equation*}
  The constant in the Landau notation is independent of $d$ and $r$.
\end{theorem}

\begin{proof}
  In the following, denote $\bsx^* \dfn S^{-1}(\bsy)$ and let $\tilde \bsx$ be the approximation to $\bsx^*$
  produced by \Cref{alg:T_impl}. In the following we show by induction over $j=1,\dots,d$ that
    \begin{equation*}
      \norm[1]{\upto{\tilde \bsx}{j} - \upto{\bsx^*}{j}}
      \le 2^{-r} \sum_{i=1}^j \binom{j}{i} L^{i-1} \dfnn \varepsilon_j.
    \end{equation*}
  For $j = 1$,
    \begin{equation*}
      \abs{x_1 - x_1^*} = \abs{x_1 - S^{-1}_1(y_1)} \le 2^{-r}.
    \end{equation*}
  For $j > 1$ and any $\bsx\in\uid$ denote $S_{\upto{\bsx}{j-1}}(x_j) \dfn S_j(\upto{\bsx}{j})$. Then
    \begin{align} \label{thm:T_impl:error_split}
      \abs{\tilde x_j - x_j^*} = \abs{\tilde x_j - S_{\upto{\bsx^*}{j-1}}^{-1}(y_j)}
      &\le \abs{\tilde x_j - S_{\upto{\tilde \bsx}{j-1}}^{-1}(y_j)}
      + \abs{S_{\upto{\tilde \bsx}{j-1}}^{-1}(y_j) - S_{\upto{\bsx^*}{j-1}}^{-1}(y_j)}.
    \end{align}
  Now, the function $\upto{\bsx}{j-1} \mapsto S_{\upto{\bsx}{j-1}}^{-1}(y_j)$ corresponds to the implicit
  function $\varphi_j(\upto{\bsx}{j-1})$ defined via
    \begin{equation*}
      S_j(\upto{\bsx}{j-1}, \varphi_j(\upto{\bsx}{j-1})) = y_j.
    \end{equation*}
  If $\varphi_j$ is Lipschitz, then we can bound the second term on the right-hand side
  of~\eqref{thm:T_impl:error_split} by
    \begin{equation*}
      \abs{S_{\upto{\tilde \bsx}{j-1}}^{-1}(y_j) - S_{\upto{\bsx^*}{j-1}}^{-1}(y_j)}
      = \abs{\varphi_j(\upto{\tilde \bsx}{j-1}) - \varphi_j(\upto{\bsx^*}{j-1})}
      \le \lip(\varphi_j) \norm[1]{\upto{\tilde \bsx}{j-1} - \upto{\bsx^*}{j-1}}.
    \end{equation*}
  We obtain a bound on the Lipschitz constant
    \begin{equation*}
      \lip(\varphi_j)
      \dfn \sup_{\bsx \in \uid, i\in\{1,\dots,d\}} \abs{\partial_{x_i} \varphi_j(\upto{\bsx}{j-1})}
    \end{equation*}
  by considering
    \begin{equation*}
      \partial_{x_i} S_j(\upto{\bsx}{j-1}, \varphi_j(\upto{\bsx}{j-1})) = 0
    \end{equation*}
  for $i \in \{1, \dots, j-1\}$. Via the chain rule it follows that
    \begin{equation*}
      \partial_{x_i} \varphi_j(\upto{\bsx}{j-1})
      = \frac{\partial_{x_i} S_j(\upto{\bsx}{j})}{\partial_{x_j} S_j(\upto{\bsx}{j})}
      \Bigg\rvert_{x_j = \varphi_j(\upto{\bsx}{j-1})}.
    \end{equation*}
  Recall that by definition, $S_j = \frac{r_j(\upto{\bsx}{j})}{s_{j-1}(\upto{\bsx}{j-1})}$ with
    \begin{equation*}
      r_j(\upto{\bsx}{j})
      \dfn \int_0^{x_j} \int_{\uij{d-j}} f(\upto{\bsx}{j-1}, t, \from{\bsx}{j+1}) \dd \from{\bsx}{j+1} \dd t
    \end{equation*}
  and
    \begin{equation*}
      s_{j-1}(\upto{\bsx}{j-1})
      \dfn \int_{\uij{d-j+1}} f(\upto{\bsx}{j-1}, \from{\bsx}{j}) \dd \from{\bsx}{j}.
    \end{equation*}
  Since $\partial_{x_j}r_j=s_j$ we have $\partial_{x_j}S_j=\frac{s_j}{s_{j-1}}$, and thus
    \begin{align*}
      \abs{\frac{\partial_{x_i} S_j(\upto{\bsx}{j})}{\partial_{x_j} S_j(\upto{\bsx}{j})}}
      &= \abs{\left(\frac{\partial_{x_i}r_j(\upto{\bsx}{j})}{s_{j-1}(\upto{\bsx}{j-1})} - \partial_{x_i}s_{j-1}(\upto{\bsx}{j-1}) \frac{r_j(\upto{\bsx}{j})}{s_{j-1}^2(\upto{\bsx}{j-1})}\right) \frac{s_{j-1}(\upto{\bsx}{j-1})}{s_{j}(\upto{\bsx}{j})}}\\
      &\le \left(\abs{\partial_{x_i}r_j(\upto{\bsx}{j})} + \abs{\partial_{x_i}s_{j-1}(\upto{\bsx}{j-1})} \underbrace{\frac{r_j(\upto{\bsx}{j})}{s_{j-1}(\upto{\bsx}{j-1})}}_{\le 1}\right) \frac{1}{s_{j}(\upto{\bsx}{j})}\\
      &\le \frac{2}{M} \ \max_{\bsx \in \uid} \abs{\partial_{x_i} f(\bsx)},
    \end{align*}
  with $M = M(f)$ as in~\eqref{eq:ass:bisection_density:a}. Therefore,
    \begin{equation*}
      \lip(\varphi_j)
      \le \frac{2}{M}\ \max_{\bsx\in\uid} \max_{i \in \{1,\dots d\}}
      \abs{\partial_{x_i} f(\bsx)} = L(f) \dfnn L
    \end{equation*}
  for all $j \in \{1, \dots, d\}$ with $L(f)$ as in~\eqref{eq:ass:bisection_density:b}. With the above we find
    \begin{equation*}
      \abs{\tilde x_j - S_{\upto{\bsx^*}{j-1}}^{-1}(y_j)}
      \le 2^{-r} + L \ \norm[1]{\upto{\tilde \bsx}{j-1} - \upto{\bsx^*}{j-1}}
    \end{equation*}
  The induction step then follows under the induction hypothesis:
  \begin{align*}
    \norm[1]{\upto{\tilde \bsx}{j} - \upto{\bsx^*}{j}}
    &= \norm[1]{\upto{\tilde \bsx}{j-1} - \upto{\bsx^*}{j-1}} + \abs{\tilde x_j - S_{\upto{\bsx^*}{j-1}}^{-1}(y_j)}\\
    &\le (1 + L) \norm[1]{\upto{\tilde \bsx}{j-1} - \upto{\bsx^*}{j-1}} + 2^{-r}\\
    &\le 2^{-r}\left( (1 + L)\sum_{i=1}^{j-1} \binom{j-1}{i} L^{i-1} \right)\\
    &= 2^{-r}\left(\sum_{i=1}^{j-1} \binom{j-1}{i} L^{i-1} + \sum_{i=1}^{j-1} \binom{j-1}{i} L^{i} \right)\\
    &= 2^{-r}\left(\sum_{i=1}^{j-1} \binom{j-1}{i} L^{i-1} + \sum_{i=2}^{j-1} \binom{j-1}{i-1} L^{i-1} + L^{j-1}\right)\\
    &= 2^{-r}\left(1 + \sum_{i=2}^{j-1} \left(\binom{j-1}{i} + \binom{j-1}{i-1}\right) L^{i-1} + L^{j-1} \right)\\
    &= 2^{-r}\left(1 + \sum_{i=2}^{j-1} \binom{j}{i} L^{i-1} + L^{j-1}\right)\\
    &= 2^{-r} \sum_{i=1}^{j} \binom{j}{i} L^{i-1}.
  \end{align*}
  Statement~\eqref{thm:T_impl:error} then follows with $j=d$ and the observation that $\sum_{i=1}^{d}
  \binom{d}{i} L^{i-1} = (L+1)^d - 1$.

  The bound on the cost follows by observing that each component of $S$ has to be evaluated $r$ times.
\end{proof}

\begin{corollary}\label{cor:error_rate_with_bisection}
  Consider a measure $\pi$ on $\uid$, and a sequence of transports $(S_\ell)_{\ell\in\N}$ mapping from
  $\uid\to\uid$. Denote $\pi_\ell \dfn S_\ell^\sharp \mu$, where $\mu$ is the uniform measure on $\uid$. If
  \begin{enumerate}
    \item the densities $f_{\pi_\ell}$ of $\pi_\ell$ are in $C^1(\uid;\Rp)$ and there exist
      $M,L \in \Rp$ such that $M \le M(f_{\pi_\ell})$ and $L \ge L(f_{\pi_\ell})$ for all $\ell \in \N$ and
    \item $\dist(\pi, \pi_\ell) = \cO(q(\ell))$ for some $q : \N \to \Rpz$ and
      $\dist \in \{\hellinger, \wasserstein{p}\}$
  \end{enumerate}
  then it holds for $\tilde T^\ell : \uid \to \uid$ where $\tilde T^\ell(\bsy) =
  \alg{\ref{alg:T_impl}}(S_\ell, \bsy, r)$ with $r = \log(1/q(\ell))$ that
    \begin{equation*}
      \dist(\pi, \tilde T^\ell_\sharp \mu) = \cO(q(\ell)).
    \end{equation*}
\end{corollary}

\begin{proof}
  By the triangle inequality we have that
    \begin{equation*}
      \dist(\pi, \tilde{T}_{\ell, \sharp} \mu)
      \le \dist(\pi, T^\ell_\sharp \mu) + \dist(T^\ell_\sharp \mu, \tilde{T}_{\ell, \sharp} \mu).
    \end{equation*}
  The first term is $\cO(q(\ell))$ by assumption. Fix $\dist = \wasserstein{p}$. By
  \Cref{thm:wasserstein_transport_1},
  \begin{align*}
    \wasserstein{p}(T^\ell_\sharp \mu, \tilde T^\ell_\sharp \mu)
    &\le \left(\int_{\uid} \norm[1]{T^\ell(\bsx) - \tilde T^\ell (\bsx)}^p \dd \bsx \right)^{1/p} \\
    &\le \sup_{\bsx \in \uid} \norm[1]{T^\ell(\bsx) - \tilde T^\ell (\bsx)}.
  \end{align*}
  Combining the above with~\eqref{thm:T_impl:error} and choosing $r = \log(1/q(\ell))$, we find
    \begin{align*}
      \wasserstein{p}(T^\ell_\sharp \mu, \tilde T^\ell_\sharp \mu) = \cO(q(\ell)).
    \end{align*}
  This bound carries over to the case $\dist = \hellinger$ since $\hellinger \le \sqrt{2} \wasserstein{2}$.
\end{proof}

\begin{algorithm}[H]
  \caption{Approximate transport computation from unnormalized density\\
    \emph{Input:} $\ftaru : \uid \to \Rpz$, $\Lambda \subset \N_0^d$ downward closed, $r \in \N$\\
    \emph{Output:} $\tilde T : \uid \to \uid$ such that $\tilde T_\sharp \mu \approx \pi$ } \label{alg:general_final}
  \begin{algorithmic}[1]
    \State $g \dfn \alg{\ref{alg:wls}}(\ftarus, (L_\bsnu)_{\bsnu \in \Lambda})$ \label{alg:general_final:2}
    \State $S(\bsy) \dfn \alg{\ref{alg:S_impl}}(g^2, \bsy)$
    \State $\tilde T(\bsx) \dfn \alg{\ref{alg:T_impl}}(S, \bsx, r)$
    \State
    \Return $\tilde T$
  \end{algorithmic}
\end{algorithm}

The following corollary makes the convergence rate and cost per sample explicit for densities in \ckalpha.
It follows immediately from combining \Cref{cor:error_rate_with_bisection} and \Cref{thm:T_impl} with
\Cref{thm:error_c_k_alpha}.

\begin{corollary}
  Let $\ftaru$ be such that $\ftarus \in \ckalpha$ for some $k \in \N_0$ and $\alpha \in
  [0,1]$. For $\ell \in \N$, let $\Lambda_\ell$ as in~\eqref{eq:def_lambda_ell}, and suppose there exist
  $M_0$, $L_0>0$ such that for all $\ell \in \N$ it holds for $g_\ell \dfn
  \alg{\ref{alg:wls}}(\ftarus, (L_\bsnu)_{\bsnu \in \Lambda})$ that $M_0 \le M(g_\ell^2)$ and $L_0
  \ge L(g_\ell^2)$.

  Then for any $p \in [1,\infty)$, it holds with $r_\ell \dfn \frac{k+\alpha}{d} \log(\abs{\Lambda_\ell})$ and
  $\tilde T^\ell \dfn \alg{\ref{alg:general_final}}(\ftaru,\Lambda_\ell,r_\ell)$ that
    \begin{align*}
      \hellinger(\pi, \tilde T_\sharp^\ell \mu) &= \cO\left(|\Lambda_\ell|^{-\frac{k+\alpha}{d}}\right)\\
      \wasserstein{p}(\pi, \tilde T_\sharp^\ell \mu)
      &= \cO\left(|\Lambda_\ell|^{-\frac{k+\alpha}{d} \frac{1}{p}} \right)
    \end{align*}
  as $\ell\to\infty$. The number of floating point operations required to evaluate $\tilde{T}^\ell$ scales
  like
    \begin{equation*}
      \cO\left(\frac{k+\alpha}{d} \log(\abs{\Lambda_\ell}) \abs{\Lambda_{\ell}}^{1+\frac{1}{d}}\right),
    \end{equation*}
  where the constant in the Landau notation is independent of $d$ and $\ell$.
\end{corollary}

The corresponding results for analytic functions~--~displaying exponential convergence rates~--~are obtained
immediately via \Cref{thm:error_analytic}, $\Lambda = \Lambda_{\bsrho, \ell}$ as
in~\eqref{eq:def_lambda_bsrho_ell} and $r = \beta \abs{\Lambda_{\bsrho, \ell}}^{1/d}$.

\section{Interpolation based method}\label{sec:interpolation}

Our presentation so far focused on approximating $\ftarus$ via a weighted least-squares method in
Line 1 of \Cref{alg:general}. In principle any other type of approximation could be used instead. Our main
motivation for using least-squares is that, due to the number of samples $n$ exceeding the dimension $m$ of
the polynomial ansatz space, this method remains reasonably stable even if the ansatz space is not expressive
enough to obtain good approximations of the target density, as occurs for example for moderate polynomial
degree strongly concentrated distributions.

In this subsection we discuss the use of interpolation to obtain density surrogates. While interpolation will
typically perform poorly in a scenario as described above, for sufficiently smooth (in a suitable sense)
target densities it yields a deterministic and fast converging method, with guaranteed error bounds.
We proceed by shortly recalling polynomial interpolation on $\uid$ in \Cref{sec:mulint}. Afterwards
we provide convergence rates of an interpolation based variant of \Cref{alg:main} in \Cref{sec:finite}.
Finally, we address the case of high- (or infinite-) dimensional distributions in \Cref{sec:infinite}. Most
proofs are deferred to \Cref{app:interpolation}.

\subsection{Multivariate interpolation}\label{sec:mulint}

\newcommand{\iptarget}{{f}}
For $m\in\N_0$ and $\iptarget \in C^0(\ui)$ we denote the one dimensional Lagrange interpolant by
  \begin{equation*}
    I_m[\iptarget](x) = \sum_{j=0}^{m} \iptarget(\chi_{m,j}) \ell_{m,j}(x),\qquad
    \ell_{m,j}(x) = \prod_{\substack{i=0\\ i\neq j}}^m\frac{x-\chi_{m,i}}{\chi_{m,j}-\chi_{m,i}},
  \end{equation*}
where for each $m\in\N$, $(\chi_{m,j})_{j=0}^m$ is a distinct sequence of points in $\ui$. Additionally, for
$m=0$ we let $\ell_{0,0}\equiv 1$ be the constant $1$ function. Throughout what follows we will assume that
the interpolation nodes are such that their Lebesgue constant ${\rm Leb}((\chi_{m,j})_{j=0}^m)$, which is
defined as the norm of the linear operator $I_m:C^0(\ui)\to C^0(\ui)$, grows at most logarithmically in
$m$, i.e.
  \begin{equation}\label{eq:leb}
    {\rm Leb}((\chi_{m,j})_{j=0}^m)=\norm[]{I_m}\le C\log(m)\qquad\forall m\in\N.
  \end{equation}
This is for instance satisfied for the Chebychev points, see e.g.~\cite{gunttner1980evaluation}. We have that
$I_m:C^0(\ui)\to\bbP_m$, where here and in the following for $m\in\N_0$
  \begin{equation*}
    \bbP_m\dfn {\rm span}\{1,x,\dots,x^m\}.
  \end{equation*}

Next, to define a multivariate interpolant, given a multi-index $\bsnu\in\N_0^d$ we introduce the tensorized
operator $I_\bsnu= \otimes_{j=1}^dI_{\nu_j}$, where $I_{\nu_j}$ acts on the $j$-th variable. Then
$I_\bsnu:C^0(\ui^d)\to \otimes_{j=1}^d \bbP_{\nu_j}$ is a bounded linear interpolation operator.

To treat infinite parametric functions we define the set
  \begin{equation*}
    \cF\dfn\set{\bsnu\in\N_0^{\N}}{\norm[\infty]{\bsnu}<\infty}
  \end{equation*}
of finitely supported multi-indices and set for $\bsnu\in\cF$, $I_\bsnu= \otimes_{j\in\N}I_{\nu_j}$. This
operator is understood in the sense
\begin{equation*}
  I_\bsnu[\iptarget](\bsx) =
  \sum_{i_1=0}^{\nu_1}\sum_{i_2=0}^{\nu_2}\dots
  \iptarget(\chi_{\nu_1,i_1},\chi_{\nu_2,i_2},\dots)\prod_{j\in\N}\ell_{\nu_j,i_j}(x_j)\qquad
  \forall \bsx\in \ui^\N,
\end{equation*}
where $\iptarget:\ui^\N\to\R$. In this case $I_\bsnu:C^0(\ui^\N)\to C^0(\ui^\N)$, where $C^0(\ui^\N)$ is
equipped with the product topology. Note that due to $\bsnu\in\cF$, the function
$I_\bsnu[\iptarget]:\ui^\N\to\R$ depends only on the finitely many variables for which $\nu_j\neq 0$.

Finally, if $d\in\N$ (or $d=\infty$) such that $\Lambda\subseteq\N_0^d$ (or $\Lambda\subseteq\cF$) is a finite
downward closed set of multi-indices, then
  \begin{equation}\label{eq:ILambda}
    I_\Lambda\dfn \sum_{\bsnu\in\Lambda}c_{\bsnu,\Lambda}I_\bsnu,\qquad
    c_{\bsnu,\Lambda}\dfn \sum_{\set{\bse\in \{0,1\}^d}{\bsnu+\bse \in \Lambda}} (-1)^{\abs{\bse}}
  \end{equation}
where the coefficients $c_{\bsnu,\Lambda}$ are referred to as combination coefficients. For basic properties
and further details of these polynomial interpolation operators, we refer to~\cite{MR1768951,MR3230010} or,
e.g.,~\cite[Section~1.3]{zech2018sparse}.

\subsection{Convergence rates in finite dimensions}\label{sec:finite}

The following theorem, proven in \Cref{app:interpolation}, shows that under assumption~\eqref{eq:leb} on
the interpolation nodes, a simple tensor-product interpolation operator is able to achieve the same
convergence rates for the $L^\infty$-error as the WLS method in case the target function belongs to one of the
classes already considered in \Cref{sec:error_and_cost}, namely to $\ckalpha$ and to the class of analytic
functions. This holds true up to a logarithmic factor originating from the Lebesgue constant of the
interpolation operator.

\begin{theorem}\label{thm:ip_error_c_k_alpha_analytic}
  Let $d\in\N$, $k\in\N$, $\alpha\in[0,1]$ and $f : \uid \to \R$. For $m\in\N$ set $\Lambda_m \dfn
  \set{\bsnu\in\N_0^d}{0\le\nu_j\le \lceil m^{1/d}\rceil}$. Then there exist $C_1 = C_1(d,\iptarget) > 0, C_2
  = C_2(d,k,\iptarget) > 0$, and $\beta = \beta(d,\iptarget) > 0$ such that for all $m\in\N$
    \begin{equation*}
      \norm[L^\infty(\uid)]{\iptarget - I_{\Lambda_m}[\iptarget]} \le
      \begin{cases}
        C_1 \exp(- \beta m^{1/d}) & \text{if } \iptarget \text{ satisfies \Cref{asm:analytic}},\\
        C_2 m^{-\frac{k+\alpha}{d}} (1+\log(m))^d & \text{if } \iptarget \in \ckalpha.
      \end{cases}
    \end{equation*}
  The computation of $I_{\Lambda_m}[\iptarget]$ requires $O(m)$ evaluations of $\iptarget$.
\end{theorem}

Next let
  \begin{equation*}
    \ckmix\dfn \set{\iptarget\in C^0(\uid)}{\partial^\Balpha \iptarget(\bsx)\in C^0(\uid),~\forall \Balpha\in \{0,\dots,k\}^d},
  \end{equation*}
where
  \begin{equation*}
    \norm[\ckmix]{\iptarget}
    \dfn \sum_{\Balpha\in\{0,\dots,k\}^d} \sup_{\bsx\in\uid}|\partial^\Balpha \iptarget(\bsx)|.
  \end{equation*}
As is well-known, for the case of mixed regularity, sparse-grids can
lessen the curse of dimensionality in the following sense:

\begin{theorem}\label{thm:ip_error_c_k_mix}
  Let $d\in\N$ and $k\in\N$. There exists $C=C(d,k)>0$ and for every $m\in\N$ exists
  $\Lambda_m\subseteq\N_0^d$ such that for all $\iptarget\in \ckmix$ holds
    \begin{equation*}
      \norm[L^\infty(\uid)]{\iptarget - I_{\Lambda_m}[\iptarget]}
      \le C \norm[\ckmix]{\iptarget}m^{-k} (1+\log(m))^d.
    \end{equation*}
  The computation of $I_{\Lambda_m}[\iptarget]$ requires the evaluation of $\iptarget$ at
  $O(m (1+\log(m))^{d-1})$ points in $\uid$.
\end{theorem}

Observe that the term of the type $m^{-k/d}$ in \Cref{thm:ip_error_c_k_alpha_analytic} has been replaced by
$m^{-k}$, which is $d$-independent, however we have an additional $(1+\log(m))^d$ factor that still depends
exponentially on $d$. Results of this type are in principle well-known, see for example~\cite{MR2249147},
where sparse-grid interpolations are discussed for piecewise polynomial approximants (i.e.~the polynomial
degree is fixed, and the spatial grid is refined as $m$ increases). We did not find the result for sparse-grid
polynomial interpolation as stated in \Cref{thm:ip_error_c_k_mix} in the literature however, which is why we
provide a proof and further references in \Cref{app:interpolation}.

With these approximation results at hand, we immediately get an analogue of \Cref{thm:main2} for the case of
interpolation:

\begin{theorem}\label{thm:ip_hellinger}
  For $d\in \N$, let $\pi$ be a probability measure on $\uid$ with density $\ftar : \uid \to \Rpz$. Let $T^m$
  be computed by \Cref{alg:main} but using the surrogates $g\dfn I_{\Lambda_m}[\ftarus]$ from
  Theorems \ref{thm:ip_error_c_k_alpha_analytic}, \ref{thm:ip_error_c_k_mix} instead of $g$ in line 1 of
  \Cref{alg:main}. Then there exist $C_1=C_1(d,\ftar) > 0$, $C_2=C_(d,k,\ftar) > 0$ and
  $\beta=\beta(d,\ftar)>0$ such that for all $m\in\N$
  \begin{equation*}
    \hellinger(\pi, T^m_\sharp \mu) \le
    \begin{cases}
      C_1 \exp(- \beta m^{1/d}) & \text{if } \ftarus \text{ satisfies \Cref{asm:analytic}}\\
      C_2 m^{-\frac{k+\alpha}{d}}(1+\log(m))^d & \text{if } \ftarus \in \ckalpha\\
      C_2 m^{-k} (1+\log(m))^d & \text{if } \ftarus \in \ckmix.
    \end{cases}
  \end{equation*}
  The number of function evaluations required to compute $I_{\Lambda_m}[\ftarus]$ is of size $O(m
  (1+\log(m))^{d-1})$ in case $\ftarus\in \ckmix$, and $O(m)$ otherwise.
\end{theorem}

\begin{proof}
  We only show the case of $\ftarus \in \ckmix$. The other two cases follow by similar arguments,
  also see \Cref{thm:error_c_k_alpha} and \Cref{thm:error_analytic}.

  Let again $\ftaru = c_\pi \ftar$, where $\ftar$ is the normalized version of the density. Then by
  \Cref{lma:bound_approx} and \Cref{thm:ip_error_c_k_mix}
    \begin{align*}
      \hellinger(\pi, T^{m}_\sharp\rho)
      &\le \frac{2}{\sqrt{c_\pi}} \norm[L^2(\uid)]{\ftarus - I_{\Lambda_m}[\ftarus] } \\
      &\le \frac{C}{\sqrt{c_\pi}}\norm[\ckmix]{\ftarus}m^{-k}(1 + \log(m))^d \\
      &= C \norm[\ckmix]{\ftars}m^{-k}(1 + \log(m))^d.
    \end{align*}
  The claim about the number of required function evaluations holds by \Cref{thm:ip_error_c_k_mix}.
\end{proof}

\begin{remark}
  A similar result can be obtained for the Wasserstein distance $\wasserstein{p}$. In this case, the error
  bounds specified on the right-hand side of \Cref{thm:ip_hellinger} need to be raised to the power of $1/p$.
\end{remark}

\subsection{High-dimensional measures}\label{sec:infinite}

For certain inverse problems, the posterior admits a Radon-Nikodym derivative w.r.t.~to the infinite product
probability measure $\otimes_{j\in\N}\lambda|_{\ui}$ on $\ui^\N$, which is of the form
  \begin{equation}\label{eq:infdens}
    \ftaru(\bsx) \dfn \mathfrak{f}\Big(\sum_{j\in\N}x_j\psi_j\Big)\qquad x_j\in \ui.
  \end{equation}
Here $\bsx=(x_j)_{j\in\N}\in \ui^\N$ is the parameter that is to be inferred, $\mathfrak{f}:X\to \R$ for some
real Banach space $X$, and $\psi_j\in X$ for each $j\in\N$.  Without going into further details on the
derivation of such formulas, we mention that this occurs in particular when the prior on $\bsx$ is expressed
in a Karhunen-Lo\'eve type expansion, also see \Cref{rmk:analytic} ahead. In the next theorem, we denote by
$X_\C$ a complexification of the Banach space $X$, see~\cite{munoz99} for details. In order to do
computations, the dimension must be truncated, so we introduce the unnormalized density
  \begin{equation*}
    \hat f_{\pi,d}(x_1,\dots,x_d)\dfn \mathfrak{f}\Big(\sum_{j=1}^d x_j\psi_j\Big)\qquad x_j\in \ui,
  \end{equation*}
and denote the corresponding measure on $\uid$ by $\pi_d$.

As is well known~\cite{MR3230010,MR3298364}, for such functions sparse-grid interpolation yields dimension
independent approximation rates as stated in the following theorem. We refer to~\cite[Theorem
3.1.2]{zech2018sparse} for a proof of the result in the present formulation:

\begin{theorem}\label{thm:infdens}
  Let $\ftaru$ be as in~\eqref{eq:infdens} and let $p \in (0,1)$ be such that
  $\sum_{j\in\N}\norm[X]{\psi_j}^p<\infty$. Assume that $\sqrt{\mathfrak{f}}$ allows a complex differentiable
  extension to some open set $O\subseteq X_\C$ containing (the compact set)
  $\set{\sum_{j\in\N}x_j\psi_j}{x_j\in \ui}\subseteq X$.

  Then there exists $C$ such that for each $d$, $m\in\N$, there exists $\Lambda_{m,d}\subseteq\N_0^d$ such
  that
    \begin{equation*}
      \norm[{L^\infty(\uid)}]{\sqrt{\hat f_{\pi,d}}-I_{\Lambda_{m,d}}[\sqrt{\hat f_{\pi,d}}]}
      \le C m^{-\frac{1}{p}+1}.
    \end{equation*}
  Moreover, the computation of $I_{\Lambda_{m,d}}[\sqrt{\hat f_{\pi,d}}]$ requires evaluating $\hat f_{\pi,d}$
  at $O(m)$ interpolation points.
\end{theorem}

The crucial point is that all constants in \Cref{thm:infdens} remain independent of the dimension $d$. Hence,
for measures allowing densities of the type~\eqref{eq:infdens}, the curse of dimensionality can be overcome
with our method, as we state in the next result. We emphasize however, that this is due to the intrinsic
properties of these measures, and other methodologies, for example based on QMC quadrature or sparse-grids,
are also known to be dimension robust in this scenario, see for instance~\cite{MR2903278,MR3805859}.
Additionally, we refer to~\cite{zech2021sparse} for a discussion of the infinite-parametric transport in this
situation. Denoting by $\mu_d$ the uniform measure on $\uid$, we then have:

\begin{theorem}
  Consider the setting of \Cref{thm:infdens}. Let $T^{m,d}$ be computed by \Cref{alg:main} but using the
  surrogates $g\dfn I_{\Lambda_m}[\ftarus]$ from \Cref{thm:infdens}, instead of $g$ in line 1 of
  \Cref{alg:main}. Then there exists $C>0$ such that for all $m$, $d\in\N$ holds
    \begin{equation*}
      \hellinger(\pi_d,T^{m,d}_\sharp\rho_d)\le C m^{-\frac{1}{p}+1}.
    \end{equation*}
  Moreover, the computation of $T^{m,d}$ requires $O(m)$ evaluations of $\hat f_{\pi,d}$.
\end{theorem}

The proof is analogous to \Cref{thm:ip_hellinger}, but using \Cref{thm:infdens}.

\section{Numerical experiments} \label{sec:num_exp}

We apply our transport based sampling method to two exemplary applications. In \Cref{sec:num_exp:rosenbrock}
we generate two-dimensional samples distributed according to a multimodal distribution constructed from
Rosenbrock functions. In \Cref{sec:num_exp:deconvolution} we use the method to solve a deconvolution problem
formulated as a Bayesian inverse problem.

\subsection{Rosenbrock function} \label{sec:num_exp:rosenbrock}

The Rosenbrock function, often used as a test function to benchmark optimization algorithms, is defined as
  \begin{equation*}
    r_{a,b}(x_1, x_2) = (a - x_1)^2 - b (x_2- x_1^2)^2
  \end{equation*}
with parameters $a,b \in \Rp$. In order to construct a multimodal density on to the unit cube we consider the
following parametrized combination of an affine transformation and a rotation
  \begin{equation*}
    t_{s, \theta, \bsc}(x_1, x_2) = s
    \begin{pmatrix} \cos(\theta) & - \sin(\theta) \\ \sin(\theta) & \cos(\theta) \end{pmatrix}
    \begin{pmatrix} x_1 - c_1 \\ x_2 - c_2 \end{pmatrix},
  \end{equation*}
with parameters $s \in \Rp$, $\theta \in [0, 2\pi)$ and $\bsc = (c_1,c_2) \in \uij{2}$.  We construct an
unnormalized multimodal test density (shown in \Cref{fig:rosenbrock_density}) as
  \begin{equation*}
    \ftaru(x_1, x_2) \dfn \sum \limits_{j=1}^3 \exp(- r_{a,b} \circ t_{s, \theta_j, \bsc_j} (x_1, x_2)),
  \end{equation*}
with parameters $a= 0.4$, $b=4$, $s=7$, $\theta_1 = 6.147$, $\theta_2 = 4.052$, $\theta_3 = 1.96$, $\bsc_1 =
(0.437, 0.606)$, $\bsc_2 = (0.414, 0.347)$ and $\bsc_3 = (0.649, 0.457)$.

\newcommand{\w}{0.328\textwidth}
\begin{figure}[H]
  \centering \includegraphics[width=\w]{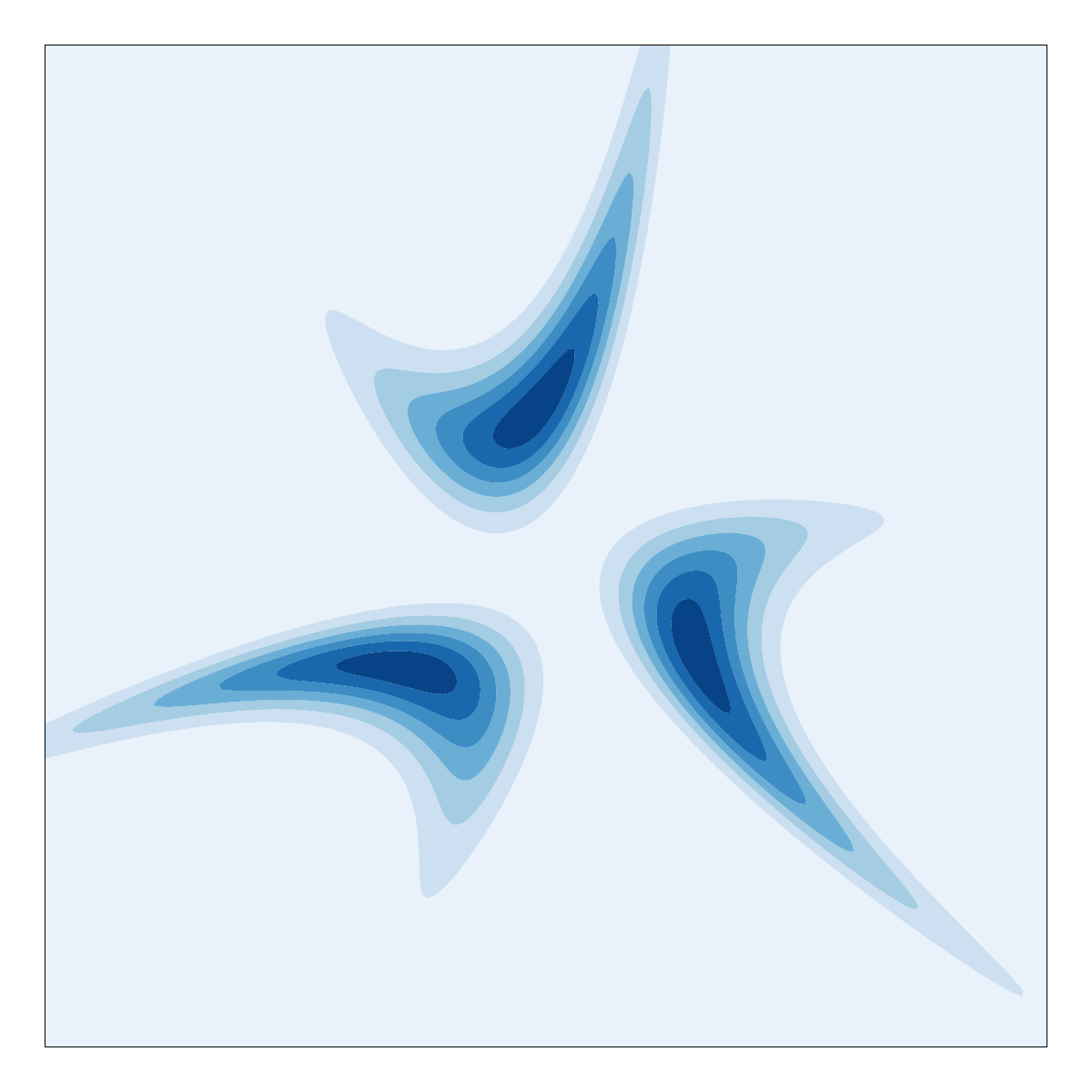}
  \caption{Test density $\ftar$.} \label{fig:rosenbrock_density}
\end{figure}

\Cref{fig:rosenbrock_samples} shows uniformly distributed samples and how a transport $T$ transforms the
domain $\Omega$, such that the sample points are transported towards regions of high probability. We observe
that artifacts present in density surrogates of low polynomial degree (and consequently also presents in the
resulting transport, see \Cref{fig:rosenbrock_samples_small}) disappear when the ansatz space is increased,
see \Cref{fig:rosenbrock_samples_large}. In order to quantify the convergence of $T_\sharp \mu$ towards $\pi$,
we plot in \Cref{fig:rosenbrock_error} the density approximation error and the Hellinger distance between the
two measures in dependence of the size of the ansatz space $\mathbb{P}_\Lambda$. It features the
exponential decay as established by our theoretical analysis in \Cref{sec:error_analysis_analytic_densities}
for analytic functions. Note, that the number of density evaluations required in the process grows as
$\abs{\Lambda} \log (\abs{\Lambda})$.

\begin{figure}[H]
  \centering
  \begin{subfigure}[b]{\w}
    \centering
    \includegraphics[width=\textwidth]{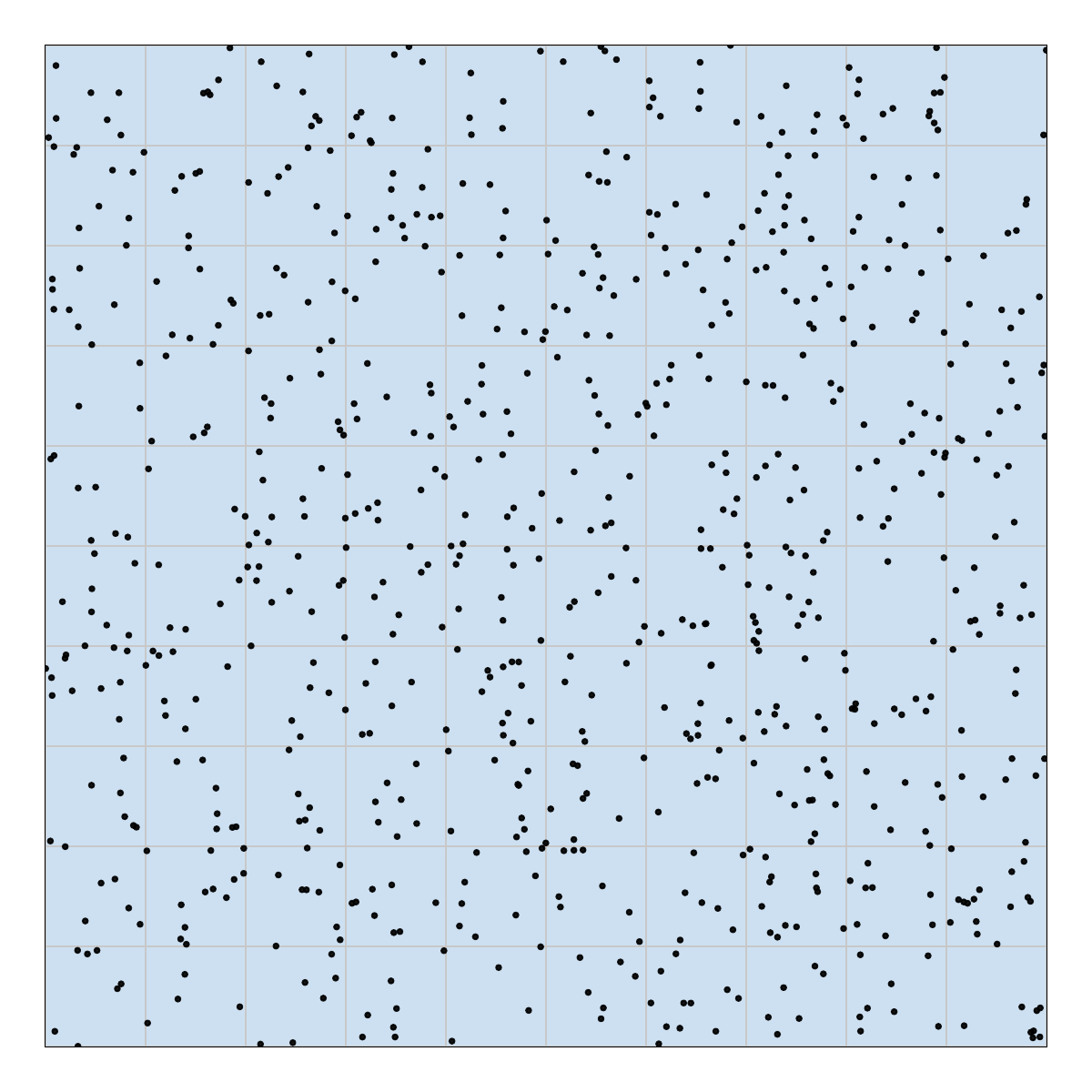}
    \caption{ $(\bsx_k)_{k=0}^N$ and $f_\mu$}
    \label{fig:rosenbrock_samples_uni}
  \end{subfigure}
  \begin{subfigure}[b]{\w}
    \centering
    \includegraphics[width=\textwidth]{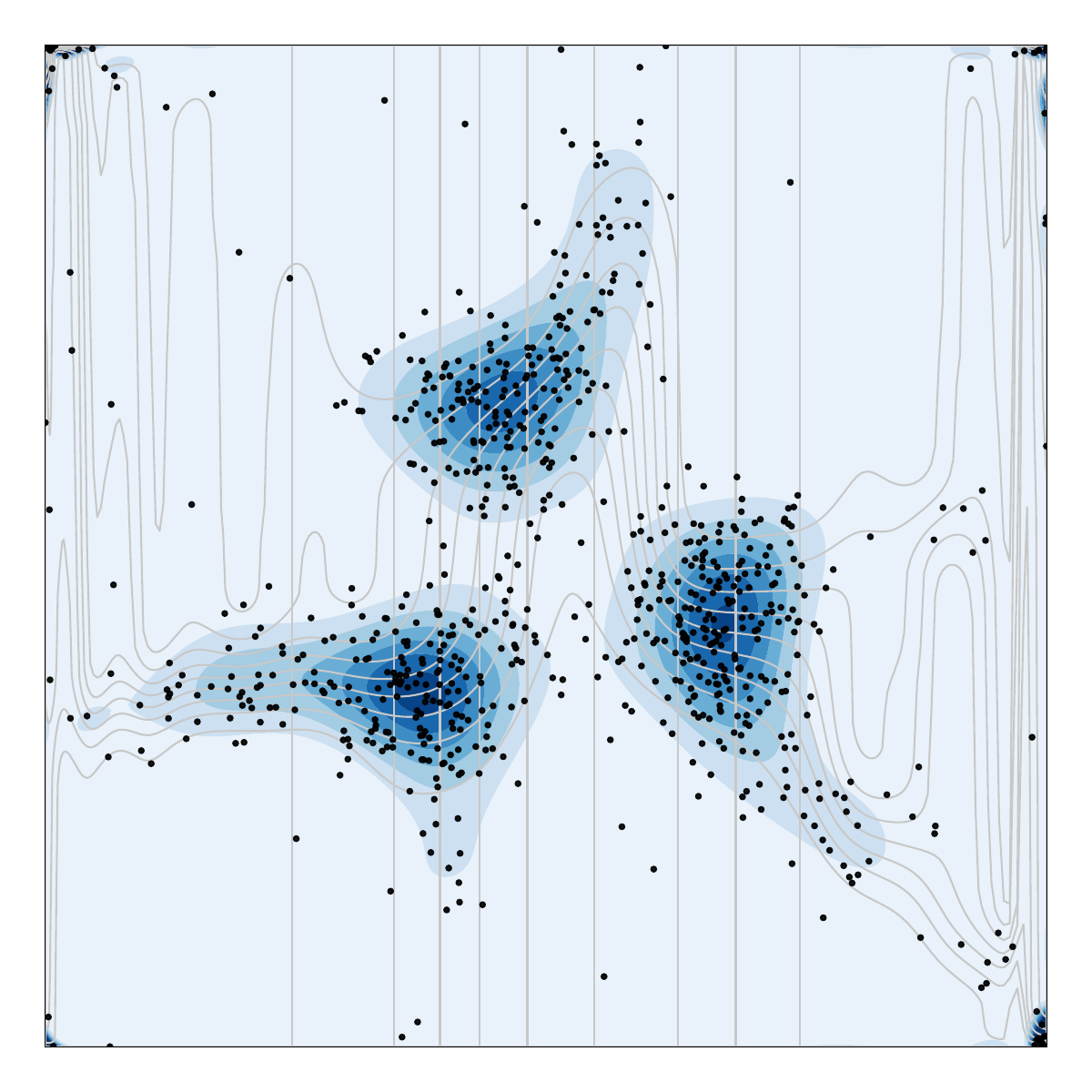}
    \caption{$(T_1(\bsx_k))_{k=0}^N$ and $f_{T_{1,\sharp}\mu}$}
    \label{fig:rosenbrock_samples_small}
  \end{subfigure}
  \begin{subfigure}[b]{\w}
    \centering
    \includegraphics[width=\textwidth]{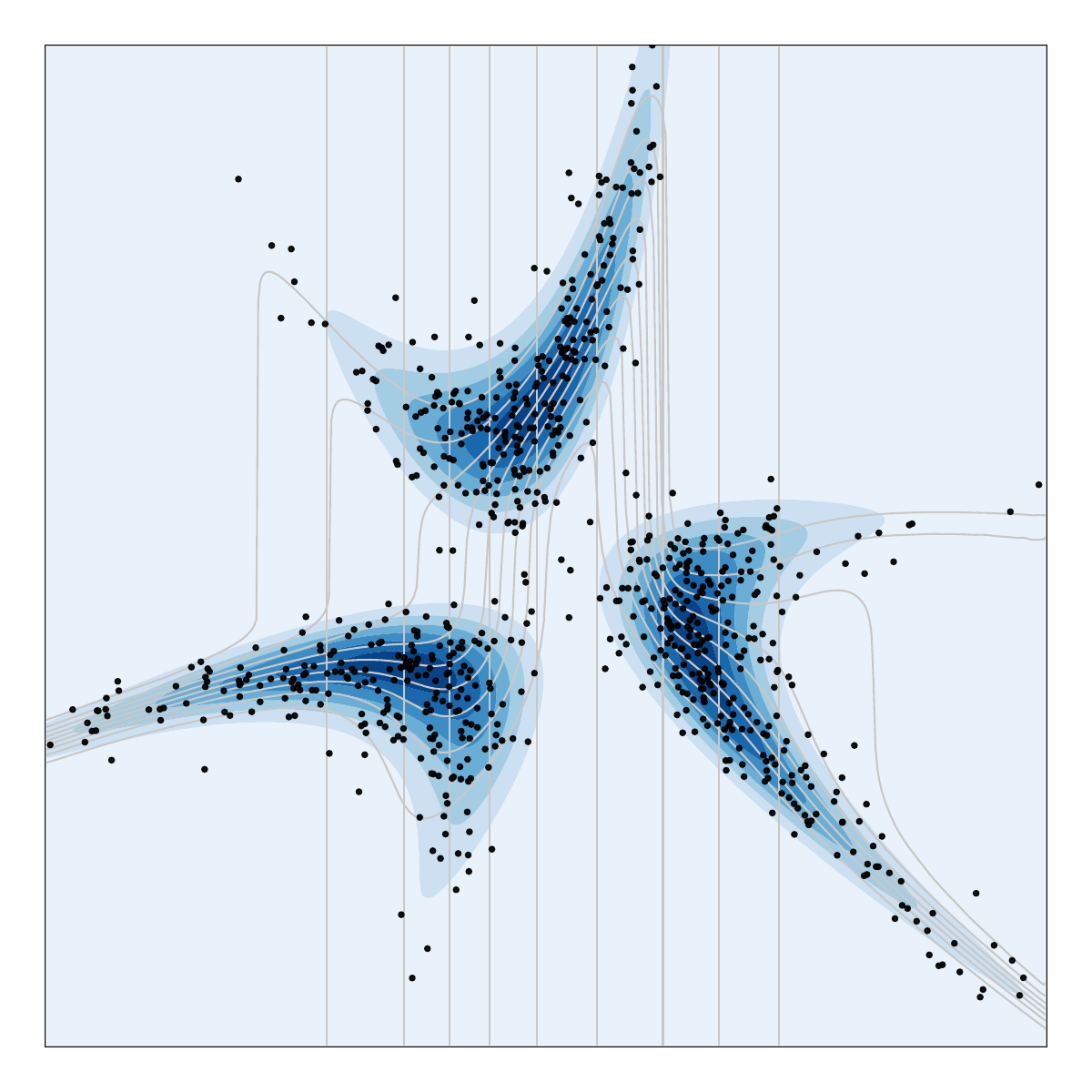}
    \caption{$(T_2(\bsx_k))_{k=0}^N$ and $f_{T_{2,\sharp}\mu}$}
    \label{fig:rosenbrock_samples_large}
  \end{subfigure}
  \caption{(\subref{fig:rosenbrock_samples_uni}) shows $N=800$ realizations of the random variable $X \sim
    \mu$, $(\bsx_k)_{k=0}^N$ over the background of the uniform density $f_\mu$ and a uniform grid.
    (\subref{fig:rosenbrock_samples_small}) and (\subref{fig:rosenbrock_samples_large}) show the
    transformation of these samples, the density and the grid lines under two transports $T_1$ and $T_2$,
    constructed from density surrogates from a small ansatz space $\mathbb{P}_{\Lambda_1}$ and from a large
    ansatz space $\mathbb{P}_{\Lambda_2}$, respectively. Specifically, we chose $\Lambda_1 = \set{\bsnu \in
    \N^d}{\norm[1]{\bsnu} \le 15}$ (thus $\abs{\Lambda_1} = 136$) and $\Lambda_2 = \set{\bsnu \in
    \N^d}{\norm[1]{\bsnu} \le 65}$ (thus $\abs{\Lambda_2} = 2211$).}
  \label{fig:rosenbrock_samples}
\end{figure}

\begin{figure}[H]
  \centering \input{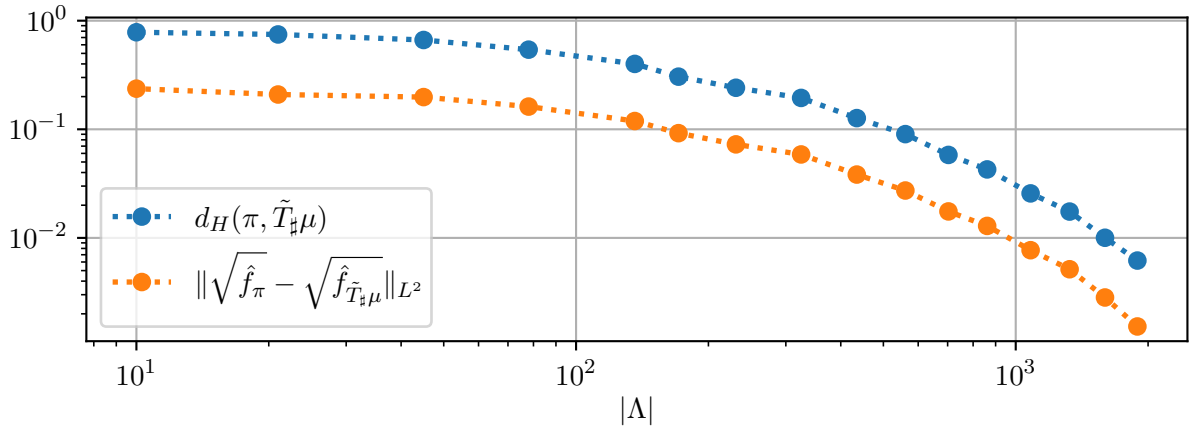}
  \caption{Hellinger distance between the target distribution $\pi$ and $T_\sharp \mu$ in dependence of the
    size of the ansatz space $\abs{\Lambda}$.}
  \label{fig:rosenbrock_error}
\end{figure}

\subsection{Deconvolution as a Bayesian inverse problem} \label{sec:num_exp:deconvolution}

The convolution of a signal function $s \in L^\infty ([-1,1])$ with a kernel function $k \in C^0(\R)$ is
defined by the bounded linear operator
  \begin{equation*}
    \cC(s)(y) \dfn s * k (y) \dfn \int_{-1}^1 s(t)k(y-t)\dd t\qquad y\in \R.
  \end{equation*}
We parametrize the signal $s$ in terms of a set of functions $(\psi_j)_{j=1}^d\subset L^\infty ([-1,1])$ and
$\bsx\in \R^d$ via
  \begin{equation*}
    s(\bsx)\dfn s(\bsx,y) \dfn \sum \limits_{j=1}^d x_j \psi_j(y),\qquad y\in\R.
  \end{equation*}
As a forward map $\cF:[-1,1]^d\to\R^r$, we consider the convoluted signal at $r\in\N$ distinct points
$(\chi_j)_{j=1}^r \dfnn \bschi$, $\chi_j \in \R$, such that
  \begin{equation*}
    \cF(\bsx)=(\cC(s(\bsx))(\chi_j))_{j=1}^r.
  \end{equation*}
We assume an additive Gaussian noise model, i.e.~a measurement of the form
  $$
    \bsm = \cF(\bsx) + \bseta \in \R^r,
  $$
with $\bseta \sim \cN(0, \Sigma)$ independent of $\bsx$ for some positive definite noise covariance matrix
$\Sigma \in \R^{r\times r}$.
Under these assumptions the posterior density, corresponding to the conditional distribution of the unknown
$\bsx$ given the measurement $\bsm$, then reads
  \begin{equation*}
    \ftar (\bsx) = \frac{\ell(\bsm|\bsx) p(\bsx)}{Z(\bsm)} \quad \bsx \in \R^d,
  \end{equation*}
where $\ell(\bsm|\bsx)$ is the data likelihood, $p(\bsx)$ the prior and $Z(\bsm) \dfn \int \ell(\bsm|\bsx)
p(\bsx) \dd \bsx$ the evidence. Since the measurement noise is additive, the data likelihood corresponds to
the density of the shifted normal distribution:
  $$
    \ell(\bsm|\bsx)
    = \frac{1}{\sqrt{2 \pi \det \Gamma}} e^{-\frac{1}{2}(\bsm - \cF(\bsx))\Gamma^{-1}(\bsm - \cF(\bsx))}.
  $$
We assume a uniform prior on $[-1,1]^d$ which corresponds to $p(\bsx)=2^{-d}$ on $[-1,1]^d$, so that
  \begin{equation}\label{eq:posterior}
    \ftaru (\bsx) =  e^{-\frac{1}{2}(\bsm - \cF(\bsx))\Gamma^{-1}(\bsm - \cF(\bsx))}, \quad \bsx \in [-1,1]^d,
  \end{equation}
is an unnormalized version of the posterior density, which we consider to be the target distribution.

\begin{remark}\label{rmk:analytic}
  We point out that $\ftaru(\bsx)$ in~\eqref{eq:posterior}, is of the type $\mathfrak{f}(\sum_{j=1}^d
  x_j\psi_j)$ for an analytic function $\mathfrak{f}$. In particular, $\ftaru$ is analytic in each
  variable, so that \Cref{thm:main2} suggests exponential convergence rates for fixed dimension $d\in\N$.
\end{remark}

For our experiment, we consider the ansatz functions
  $$
  \psi_{\ell,i}(y) = 2^{-\alpha \ell}
  \begin{cases}
    2^\ell (y + 1 - \frac{i}{2^{\ell-1}}) \quad &\text{if } y+1 \in [\frac{i}{2^{\ell-1}}, \frac{i+\frac{1}{2}}{2^{\ell-1}}) \\
    2^\ell (\frac{i+1}{2^{\ell-1}} - (y+1)) \quad &\text{if } y+1 \in [\frac{i+\frac{1}{2}}{2^{\ell-1}}, \frac{i+1}{2^{\ell-1}}] \\
    0\quad &\text{otherwise}
  \end{cases}
  $$
for $\ell \in \{0, 1, \ldots, k\}$ and $i \in \{0,1, \ldots, 2^\ell-1\}$, see \Cref{fig:convolution_basis}.
The parameter $\alpha$ controls how fast the ansatz functions decay and is set to $2$ in our evaluation. An
exemplary signal $s(\bsx)$ constructed from a randomly chosen parameter $\bsx \in [-1,1]^d$
is shown in \Cref{fig:convolution_signal}. In \Cref{fig:convolution_measurement}, its convolution with a
Gaussian kernel function $k(y) = e^{-10 y^2}$ is shown together with synthetic measurement data $\bsm$, taken
at points $(\chi_j)_{j=1}^n$ with $n=10$ and $\chi_j = \frac{2j+1}{n}-1$ for $j \in \{1, \dots, n\}$. The
noise $\bseta$ is assumed to be distributed according to $\cN(0,\sigma \mathds{1})$ for some $\sigma \in \Rp$.

\renewcommand{\w}{0.32\textwidth}

\begin{figure}[H]
  \centering
  \begin{minipage}[t]{\w}
    \centering
    \adjustbox{width=\textwidth}{\input{convolution_basis.pgf}}
    \caption{Basis hat functions $(\psi_{\ell,i})_{i=0}^{2^\ell-1}$ for $\ell \in \{1, \dots, 4\}$.}
    \label{fig:convolution_basis}
  \end{minipage} \hfill
  \begin{minipage}[t]{\w}
    \centering
    \adjustbox{width=\textwidth}{\input{convolution_signal.pgf}}
    \caption{A random signal function $s$.}
    \label{fig:convolution_signal}
  \end{minipage} \hfill
  \begin{minipage}[t]{\w}
    \centering
    \adjustbox{width=\textwidth}{\input{convolution_measurement.pgf}}
    \caption{Convoluted signal $c(y)$ and noisy measurements $(m_j)_{j=1}^{10}$.}
    \label{fig:convolution_measurement}
  \end{minipage}
\end{figure}

In order to construct transport surrogates $T$ that push the uniform measure $\mu$ forward to an approximate
posterior distribution, we choose anisotropic ansatz spaces $\mathbb{P}_\Lambda$ that reflect the decreasing
contribution of the basis functions with increasing dimension index. Specifically, we construct
$\Lambda_{\bsk, \ell}$ as introduced in~\eqref{eq:def_lambda_bsk_ell} with $\bsk = (\lceil \log_2(i+1) \rceil
)_{i=1}^d$. The Hellinger distances between $T_\sharp \mu$ and the true posterior distribution $\pi$ are shown
in \Cref{fig:convolution_error} in dependence of the number of density evaluations $n$ invested during the
construction of the surrogate. Note that the target distribution in this example (the
posterior) supported on $[-1,1]^d$ rather than $\ui^d$ as in the rest of the manuscript. This case is easily
accommodated by the introduction of an additional linear transformation $\uid\to [-1,1]^d$.

\begin{figure}[H]
  \centering
  \input{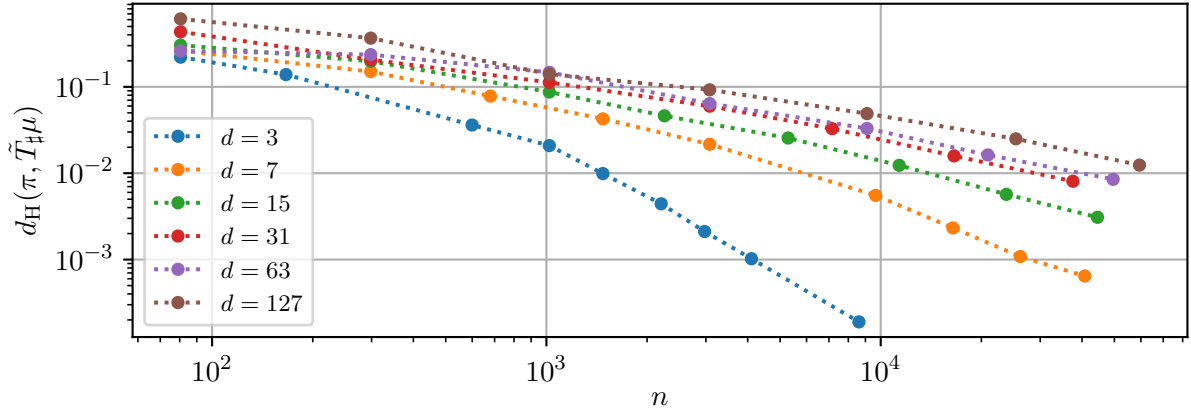}
  \caption{Convergence of the Hellinger distance between the deconvolution posterior $\pi$ and surrogate
    pushforward $T_\sharp \mu$ with increasing ansatz space cardinality.}
  \label{fig:convolution_error}
\end{figure}

\section{Conclusions}

In this work, we proposed and analyzed a sampling algorithm that is based on transport maps constructed from
polynomial surrogates of the target probability measure. These surrogates can be obtained using for example
weighted least-squares or interpolation, based on pointwise evaluations of the possibly unnormalized target
density. The resulting method is broadly applicable for measures which are absolutely continuous on the
$d$-dimensional hypersquare.

Central to this work are convergence statements for our proposed method, providing upper bounds on the
Hellinger distance between the target measure and the approximating measure. In \Cref{sec:error_and_cost} we
apriorily constructed polynomial ansatz spaces tailored to smoothness properties of the target density. This
enabled us to achieve algebraic convergence rates in terms of the size of the ansatz space for $k$-times
differentiable target densities and exponential convergence rates for analytic target densities.

The established error bounds are, however, in general subject to the curse of dimensionality. To address the
high-dimensional case, we used known sparsity and approximation results as for example provided
in~\cite{CDS10_404,MR3298364,zech2018sparse}. Leveraging such statements, we showed in \Cref{sec:infinite}
that our algorithm can overcome the curse of dimensionality for measures of the type~\eqref{eq:infdens}, which
occur for example as posterior distributions in Bayesian inverse problems.

The cost of our sampling algorithm is driven by the size of the ansatz space on the one hand and by the cost
of evaluating the target density on the other. The latter, however, only contributes to the cost of the
offline phase (the construction of the density surrogate) and not to the cost of the online phase (the
generation of the samples). We discussed the application of weighted least-squares in the offline phase and
the ensuing cost in detail in \Cref{sec:wls}.
An efficient implementation of the online phase, along with bounds on the resulting cost were presented in
\Cref{sec:implementation_and_sampling_cost}. The sampling process can further be expedited by generating
samples concurrently, a task that is straightforward with our sampling method.

\appendix
\section{Proof of \Cref{lma:decomposition}} \label{app:decomposition}

\begin{proof}[Proof of \Cref{lma:decomposition}]
  Let $f_j$ be as in the formula given in the lemma. Then
  \begin{equation*}
    \prod_{j=1}^d f_j(\upto{\bsx}{j}) =
    \frac{\int_{\uij{d-1}} f(\bsx) \dd \from{\bsx}{2}}{\int_{\uid} f(\bsx) \dd \bsx}
    \cdots \frac{f(\bsx)}{\int_{\ui} f(\bsx) \dd \from{\bsx}{d}} = f(\bsx),
  \end{equation*}
  where we used $\int_{\uid} f(\bsx) \dd \bsx = 1$. This shows that~\ref{item:prodfj} holds and
  clearly~\ref{item:intfj1} is also satisfied.

  It remains to show uniqueness of the decomposition. Suppose $f(\bsx)=\prod_{j=1}^d g_j(\upto{\bsx}{j})$ such
  that the $g_j$ satisfy~\ref{item:intfj1}. Then for all $x_1\in \ui$
  \begin{equation*}
    f_1(x_1) = \int_{\uij{d-1}} f(\bsx) \dd \from{\bsx}{2}
    = g_1(x_1) \int_0^1 g_2(\upto{\bsx}{2}) \dots \int_0^1 g_d(\bsx) \dd \mu(x_d) \dots \dd \mu(x_2)= g_1(x_1)
  \end{equation*}
  by~\ref{item:intfj1}. This shows $f_1\equiv g_1$. By induction, we obtain $g_j \equiv f_j$ for all
  $j\in\{1,\dots,d\}$.
\end{proof}

\section{Integrating products of Legendre polynomials}\label{app:legendre}

Denote by $P_n : [-1,1] \to \R$ the Legendre polynomial of degree $n \in \N_0$ with the normalization $P_n(1)
= 1$ (cf.~\cite[22.2.10]{abramowitz1948handbook}). Let $\varphi : [0,1] \to [-1,1]$ be defined as
$\varphi(x) \dfn 2x-1$. We can then construct Legendre polynomials $L_n : [0,1] \to \R$ that are orthonormal
in $L^2([0,1], \mu)$ as
  \begin{equation}\label{eq:def_legendre_ui}
    L_n(x) = \sqrt{2n+1} P_n \circ \varphi (x)
  \end{equation}
In this section we establish efficient formulas for the computation of the integrals $\int_{0}^x L_l(t)
L_k(t)\dd t$ for arbitrary $x\in [0,1]$ and $l$, $k \in \N_0$.

First recall that $P_0(x)=1$, $P_1(x)=x$ and for $n\ge 2$ and $m=0,\dots,n-1$ there holds the three term
recursion
  \begin{equation}\label{eq:Leg3}
    P_n^{(m)}(x) = \frac{(2n-1)xP_{n-1}^{(m)}(x)-(n-1+m)P_{n-2}^{(m)}(x)}{n-m},
  \end{equation}
which (for $m\in\{0,1\}$) allows computing $(P_j(x))_{j=0}^n$ and $(P_j'(x))_{j=0}^n$ at some fixed
$x\in[-1,1]$ with complexity $O(n)$. Furthermore, $P_n$ is a solution to the ODE
  \begin{equation}\label{eq:LegODE}
    n(n+1)P_n = -((1-x^2)P_n')',
  \end{equation}
with the boundary conditions $P_n(-1)=(-1)^n$ and $P_n(1)=1$.

\begin{lemma}\label{lemma:intPnPk}
  Let $k \neq n$. Then for $x \in [0,1]$
    \begin{equation*}
      \int_{0}^x L_n(t) L_k(t) \dd t = (x-x^2) \frac{L_n(x)L_k'(x)-L_k(x)L_n'(x)}{(n+k+1)(n-k)}.
    \end{equation*}
\end{lemma}

\begin{proof}
  We proceed as in~\cite[Sec.~4.5]{MR0350075}. By~\eqref{eq:LegODE}
    \begin{equation*}
      n(n+1)P_nP_k = -((1-x^2)P_n')'P_k
    \end{equation*}
  and the same holds with $n$ and $k$ reversed. Subtracting these equations from each other we get
    \begin{equation*}
      (n(n+1) - k(k+1)) P_n P_k = ((1-x^2)P_k')'P_n - ((1-x^2)P_n')'P_k = ((1-x^2)(P_k'P_n-P_n'P_k))'.
    \end{equation*}
  Integrating yields
    \begin{equation*}
      \int_{-1}^xP_n(t)P_k(t)\dd t = \frac{(1-x^2)(P_k'(x)P_n(x)-P_n'(x)P_k(x))}{(n+k+1)(n-k)}.
    \end{equation*}
  By a change of variables we obtain
    \begin{align*}
      \int_0^x P_n\circ\varphi(t) P_k\circ\varphi(t) \varphi'(t) \dd t
      &= \int_{\varphi(0)}^{\varphi(x)} P_n(t) P_k(t) \dd t \\
      &= \frac{(1-\varphi(x)^2)(P_k'\circ\varphi(x)P_n\circ\varphi(x)
          -P_n'\circ\varphi(x)P_k\circ\varphi(x))}{(n+k+1)(n-k)}.
    \end{align*}
  Note that $\varphi'(x) = 2$. Multiplying both sides with $\sqrt{2n+1}\sqrt{2k+1}$ and using $L_n(x) = \sqrt{2n+1} P_n \circ \varphi (x)$ as well as $L_n'(x) = \sqrt{2n+1} P_n' \circ \varphi(x) \varphi'(x)$ yields
  \begin{align*}
    \int_0^x L_n(t) L_k(t) \dd t
    &= \frac{(1-\varphi(x)^2)}{4} \frac{(L_k'(x) L_n(x) - L_n'(x) L_k(x))}{(n+k+1)(n-k)}.
  \end{align*}
  The proof concludes with the observation that $\frac{(1-\varphi(x)^2)}{4} = x - x^2$.
\end{proof}

For the case $n=k$ we establish a recursion formula. Together with \Cref{lemma:intPnPk}, it allows us to
compute the vector $(\int_{0}^x L_j^2(t)\dd t)_{j=0}^n$ for every fixed $x \in [0,1]$ with cost $O(n)$.

\begin{lemma}
  For $x\in [0,1]$ and $n\ge 2$
    \begin{align}\label{eq:intPn2}
      \int_0^x L_n(t)^2 \dd t
      &= \int_0^x L_{n-1}(t)^2\dd t
        +\int_0^x L_{n+1}(t)L_{n-1}(t)\dd t \frac{(n+1)\sqrt{2n-1}}{n\sqrt{2n+3}}\nonumber\\
      & -\int_0^x L_{n}(t)L_{n-2}(t)\dd t \frac{(n-1)\sqrt{2n+1}}{n\sqrt{2n-3}}.
    \end{align}
\end{lemma}

\begin{proof}
  By~\eqref{eq:Leg3} (with $m=0$)
    \begin{equation}\label{eq:xPn}
      xP_n = \frac{(n+1)P_{n+1}+nP_{n-1}}{2n+1}.
    \end{equation}
  Multiplying~\eqref{eq:Leg3} with $P_n$ and using~\eqref{eq:xPn},
    \begin{align*}
      P_n^2 &= \frac{(2n-1)xP_{n-1}P_n-(n-1)P_{n-2}P_n}{n}\\
            &= \frac{(2n-1)P_{n-1}\frac{(n+1)P_{n+1}+nP_{n-1}}{2n+1}-(n-1)P_{n-2}P_n}{n}\\
            &= P_{n-1}^2\frac{2n-1}{2n+1}+P_{n+1}P_{n-1}\frac{(2n-1)(n+1)}{n(2n+1)} - P_nP_{n-2}\frac{n-1}{n}.
    \end{align*}
  Multiplying both sides with $(2n+1)$ yields
    \begin{align*}
      (\sqrt{2n+1} P_n)^2
            = (\sqrt{2n-1} P_{n-1})^2
            + \sqrt{2n+3} P_{n+1} \sqrt{2n-1} P_{n-1} \frac{\sqrt{2n-1}(n+1)}{n \sqrt{2n+3}} \\
            - \sqrt{2n+1} P_n \sqrt{2n-3} P_{n-2} \frac{n-1 \sqrt{2n+1}}{n \sqrt{2n-3}}.
    \end{align*}
  Concatenating this expression with $\varphi$, using $L_n=\sqrt{2n+1}P_n \circ \varphi$ and integrating from $0$ to $x$ yields~\eqref{eq:intPn2}.
\end{proof}

\section{Error of polynomial approximation}\label{app:ebapp}

In this section let for any $\ell \in \N$
  \begin{equation*}
    \Lambda_\ell \dfn \{ \bsnu \in \N_0^d : \norm[1]{\bsnu} < \ell \}.
  \end{equation*}
We then have the following classical, fundamental result, for example proved
in~\cite[Theorem~1]{bagby2002multivariate}:

\begin{theorem}\label{thm:polynomial}
  Let $d$, $k\in\N$. There exists $C=C(d, k)$ such that for all $f\in C^k(\uid)$, $\ell\in\N$
    \begin{equation*}
      e_{\Lambda_\ell} (f)_\infty
      \dfn \min \limits_{h \in \bbP_{\Lambda_\ell}} \norm[L^\infty(\uid)]{f - h}
      \leq C(d, k) \ell^{-k} \sup \limits_{\set{\bsgamma\in \N_0^d}{\norm[1]{\bsgamma} = k}}
      \left( \sup \limits_{ |\bsx - \bsy| \leq 1/\ell} \bigg| \partial^\bsgamma f(\bsx) - \partial^\bsgamma f(\bsy)\bigg| \right)
    \end{equation*}
\end{theorem}

\begin{theorem} \label{thm:eba_c_k_alpha}
  Let $f \in \ckalpha$ for some $k \in \N_0$ and $\alpha \in [0,1]$ and let. Then there exists $C(d, k)$ such
  that
    \begin{equation} \label{eq:eba_c_k_alpha}
      e_{\Lambda_\ell} \left(f \right)_\infty \leq C(d,k) |f|_{\ckalpha}
      |\Lambda_\ell|^{-\frac{k+\alpha}{d}}.
    \end{equation}
\end{theorem}

\begin{proof}
  Since $f \in \ckalpha$ it holds that
  $$
    \sup \limits_{ \bsgamma\in \N^d : \norm[1]{\bsgamma} = k} \left( \sup \limits_{ |\bsx - \bsy| \leq 1/\ell} \bigg| \partial^\bsgamma f(\bsx) - \partial^\bsgamma f(\bsy)\bigg| \right) \leq |f|_{\ckalpha} \sup \limits_{ |\bsx - \bsy| \leq 1/\ell} |\bsx - \bsy|^\alpha = |f|_{\ckalpha} \ell^{-\alpha}
  $$
  and therefore
  $$
    e_{\Lambda_\ell} \left( f \right)_\infty \le C(d, k) |f|_{\ckalpha} \ell^{-k - \alpha}.
  $$
  With $|\Lambda_\ell| \le \ell^d$ and thus $\ell \geq |\Lambda_\ell|^{1/d}$ we find
  $$
    e_{\Lambda_\ell} \left(f \right)_\infty \le C(d, k) |f|_{\ckalpha}|\Lambda_\ell|^{-\frac{k+\alpha}{d}}.
  $$
\end{proof}

\section{Bounding the Wasserstein distance in terms of transports}

The following bound on the Wasserstein distance (cf.~\eqref{eq:def_wasserstein}) is closely related to
\cite[Theorem~2~(ii)]{sagiv2019wasserstein}, but derived in a more concise way and applying to a more general
class of transports.

\begin{theorem}\label{thm:wasserstein_transport_1}
  Let $(\Omega_1, \cA, \measi)$ be a measure space and let $(\Omega_2, m_2)$ be a polish metric space. If
  $T : \Omega_1 \rightarrow \Omega_2$ and $\tilde T : \Omega_1 \rightarrow \Omega_2$ are measurable, then the
  $p$-Wasserstein distance between the pushforward measures of $\measi$ under $T$ and $\tilde{T}$ can be
  bounded by
    \begin{align*}
      \wasserstein{p}(T_\sharp \measi, \tilde T_\sharp \measi)
      \le \norm[L^p(\Omega_1,\measi)]{m_2(T,\tilde{T})}.
    \end{align*}
\end{theorem}

\begin{proof}
  Let  $\gamma^*  \dfn (T, \tilde{T})_\sharp \measi$ and consider $A \subset \Omega_2$. The marginals of
  $\gamma^*$ are
    \begin{align*}
      \gamma^*(A \times \Omega_2) &= \measi((T, \tilde{T})^{-1}(A \times \Omega_2))
      = \measi(T^{-1}(A) \cap \tilde T^{-1}(\Omega_2)) = \measi(T^{-1}(A) \cap \Omega_1) = \measi(T^{-1}(A))
      = T_\sharp \measi(A)\\
      \gamma^*(\Omega_2 \times A) &= \measi((T, \tilde{T})^{-1}(\Omega_2 \times A))
      = \measi(T^{-1}(\Omega_2) \cap \tilde T^{-1}(A)) = \measi(\Omega_1 \cap \tilde T^{-1}(A) )
      = \measi(\tilde T^{-1}(A))  = \tilde T_\sharp \measi(A).
    \end{align*}
  Therefore, $\gamma^* \in \Gamma(T_\sharp \measi, \tilde T_\sharp \measi)$ and consequently
    \begin{align*}
      \wasserstein{p}(T_\sharp \measi, \tilde T_\sharp \measi)
      &\le \left(\int\limits_{\Omega_2\times\Omega_2} m_2(\bsx, \bsy)^p \dd \gamma^*(\bsx,\bsy) \right)^{1/p}
    \end{align*}
  By the change of variables formula for pushforward measures (see e.g.~\cite[Satz 4.10]{klenke}) it holds
  that
    \begin{align*}
      \int \limits_{\Omega_2 \times \Omega_2} m_2(\bsx, \bsy)^p \dd \gamma^*(\bsx,\bsy)
      = \int \limits_{\Omega_1} m_2(T(\bsx), \tilde{T}(\bsx))^p \dd \measi(\bsx)
    \end{align*}
  and therefore
    \begin{align*}
      \wasserstein{p}(T_\sharp \measi, \tilde T_\sharp \measi)
      \le \left(\int\limits_{\Omega_1} m_2(T(\bsx), \tilde{T}(\bsx))^p \dd \measi(\bsx) \right)^{1/p}.
    \end{align*}
\end{proof}

\section{Error of polynomial interpolation}\label{app:interpolation}

Recall the following basic facts for $\bsnu\in\N_0^d$, and $\Lambda\subseteq\N_0^d$ finite and downward
closed:
  \begin{itemize}
  \item
    $I_\bsnu : C^0(\uid)\to \otimes_{j=1}^d\bbP_{\nu_j}\subset C^0(\uid)$, and
    $I_\Lambda : C^0(\uid)\to \mathbb{P}_\Lambda \subset C^0(\uid)$ are bounded linear operators,
  \item
    $I_\bsnu[\iptarget](\chi_{\nu_1,j_1},\dots,\chi_{\nu_d,j_d})=\iptarget(\chi_{\nu_1,j_1},\dots,\chi_{\nu_d,j_d})$
    for all $0\le j_i\le \nu_i$,
  \item $I_\bsnu[\iptarget]=\iptarget$ whenever $\iptarget\in\otimes_{j=1}^d\bbP_{\nu_j}$, and
    $I_{\Lambda}[\iptarget]=\iptarget$ whenever $\iptarget\in\bbP_\Lambda$ (cf.~\eqref{eq:PLambda}).
  \end{itemize}

Using these properties, we can prove \Cref{thm:ip_error_c_k_alpha_analytic}. The statement for densities with
$\ftarus\in \ckalpha$ is an immediate consequence of the following result, which is
explicit in the dependence of the constants on the interpolated function:

\begin{proposition}\label{prop:intmain}
  Let $d\in\N$, $k\in\N_0$, $\alpha\in [0,1]$ and assume~\eqref{eq:leb} holds. Then there exists $C=C(d,k)$
  such that for all $\iptarget\in \ckalpha$, $m\in\N$ holds with
  $\Lambda_m=\set{\bsnu\in\N_0^d}{0\le\nu_j\le \lceil m^{1/d}\rceil}$
    \begin{align*}
      \norm[L^\infty(\uid)]{\iptarget-I_{\Lambda_m}[\iptarget]}
      &\le C |\iptarget|_\ckalpha (1+\log(m))^d m^{-\frac{k+\alpha}{d}}.
    \end{align*}
  The number of required function evaluations of $\iptarget$ to compute $I_{\Lambda_m}[\iptarget]$ scales like
  $O(m)$.
\end{proposition}

\begin{proof}
  By~\eqref{eq:leb} it holds $\norm[]{I_{\lceil m^{1/d}\rceil }}\le C_d \log(m)$. Moreover, it is easy to see
  that with $\Lambda_m=\set{\bsnu\in\N_0^d}{0\le\nu_j\le\lceil m^{1/d}\rceil}$ it holds
  $I_{\Lambda_m}=\otimes_{j=1}^d I_{\lceil m^{1/d}\rceil}$ for which the operator norm (between the spaces
  $C^0(\uid)\to C^0(\uid)$) is bounded by
    \begin{equation*}
      \norm[]{\otimes_{j=1}^d I_{\lceil m^{1/d}\rceil}}
      \le \norm[]{I_{\lceil m^{1/d}\rceil}}^d
      \le C_d \log(m)^d.
    \end{equation*}
  By \Cref{thm:eba_c_k_alpha}, we can in particular find
  $p\in \otimes_{j=1}^d\bbP_{\lceil m^{1/d}\rceil}$ such that
    \begin{equation*}
      \norm[L^\infty(\uid)]{\iptarget-p}\le C |\iptarget|_\ckalpha m^{-\frac{k+\alpha}{d}}.
    \end{equation*}
  Thus, using $I_{\Lambda_m}[p]=p$,
    \begin{align*}
      \norm[L^\infty(\uid)]{\iptarget-I_{\Lambda_m}[\iptarget]}
      &\le \norm[L^\infty(\uid)]{\iptarget-p+I_{\Lambda_m}[\iptarget-p]}\nonumber\\
      &\le (1+\norm[]{I_{\Lambda_m}})\norm[L^\infty(\uid)]{\iptarget-p}\nonumber\\
      &\le C(1+\log(m))^dm^{-\frac{k+\alpha}{d}}.
    \end{align*}
\end{proof}

The proof for the analytic case in \Cref{thm:ip_error_c_k_alpha_analytic} is analogous, but similar as in the
proof of \Cref{thm:error_analytic}, we use~\cite[Theorem~3.5]{opschoor2021exponential} to obtain a suitable
polynomial approximation $p$ to $\iptarget$.

We next turn to the proof of \Cref{thm:ip_error_c_k_mix}, which is a bit more involved. To this end, we
introduce the interpolation operators
  \begin{equation}\label{eq:tildeDelta}
    \tilde \Delta_{\bsnu} \dfn \otimes_{i=1}^d(I_{2^{\nu_i}}-I_{2^{\nu_{i}-1}}),
  \end{equation}
and for $\Lambda\subseteq\N_0^d$ downward closed
  \begin{equation}\label{eq:SGoptilde}
    \tilde I_\Lambda \dfn \sum_{\bsnu\in\Lambda}\tilde \Delta_{\bsnu}.
  \end{equation}
We use the tilde notation to emphasize, that in contrast to~\eqref{eq:ILambda}, the multi-index entries
$\nu_i$ are scaled according to $2^{\nu_i}$ in~\eqref{eq:tildeDelta}. For the analysis of functions of mixed
regularity, this scaling is preferential.

\begin{lemma}\label{lemma:diffbound}
  Let $d\in\N$, $k\in\N$, $\tilde \Delta_{\bsnu}$ as in \eqref{eq:tildeDelta} and assume \eqref{eq:leb}.
  Then there exists $C=C(d,k)>0$ such that for all $\bsnu\in\N_0^d$ and all $\iptarget \in \ckmix$
  \begin{equation*}
    \norm[C^0(\uid)]{\tilde\Delta_\bsnu[\iptarget]}\le C
    \norm[\ckmix]{\iptarget} \prod_{i=1}^d(1+\nu_i)2^{-k \nu_i}.
  \end{equation*}
\end{lemma}

\begin{proof}
  We proceed by induction over the dimension $d$. For $d=1$, according to \Cref{prop:intmain} (with $d=1$,
  $\alpha=0$) we have for every $\nu\in\N$ and $\iptarget\in C^k([0,1])$
    \begin{align*}
      \norm[{C^0([0,1])}]{\tilde\Delta_\nu[\iptarget]}
      =\norm[{C^0([0,1])}]{I_{2^\nu}[\iptarget]-I_{2^{\nu-1}}[\iptarget]}
      &\le \norm[{C^0([0,1])}]{\iptarget-I_{2^{\nu}}[\iptarget]}
          +\norm[{C^0([0,1])}]{\iptarget-I_{2^{\nu-1}}[\iptarget]}\\
      &\le C_1 (1+\nu) 2^{-\nu k}|\iptarget|_{C^{k,0}(\ui)}\\
      &\le C_1 (1+\nu) 2^{-\nu k}\norm[{C^k([0,1])}]{\iptarget},
    \end{align*}
  for a constant $C_1$ independent of $\iptarget$ and $\nu$.

  Now suppose the statement is true for $d-1\ge 1$. We can write
  $\tilde\Delta_\bsnu=\tilde\Delta_{\bsnu'}^{\bsx'}\tilde\Delta_{\nu_d}^{x_d}$,
  with $\bsnu=(\bsnu',\nu_d)$, $\bsx=(\bsx',x_d)$ and
  \begin{equation*}
    \tilde\Delta_{\bsnu'}^{\bsx'} \dfn \otimes_{j=i}^{d-1}(I_{2^{\nu_i}}^{x_i}-I_{2^{\nu_i-1}}^{x_i}),\qquad
    \tilde\Delta_{\nu_d}^{x_d} \dfn I_{2^{\nu_d}}^{x_d}-I_{2^{\nu_d-1}}^{x_d},
  \end{equation*}
  where the superscript indicates the variable on which the operator acts.

  Next we point out that $\partial_{\bsx'}^{\Balpha}$ commutes with
  $\tilde\Delta_{\nu_d}^{x_d}$ for $\Balpha\in\{0,\dots,k\}^{d-1}$
  since
  \begin{align*}
    \partial_{\bsx'}^\Balpha \tilde \Delta_{\nu_d}^{x_d}[\iptarget](\bsx',x_{d})
    &=\partial_{\bsx'}^\Balpha\left(
    \sum_{j=0}^{2^{\nu_d}}\iptarget(\bsx',\chi_{2^{\nu_{d}},j})\ell_{2^{\nu_{d}},j}(x_{d})
    -\sum_{j=0}^{2^{\nu_d-1}}\iptarget(\bsx',\chi_{2^{\nu_{d}-1},j})\ell_{2^{\nu_d-1},j}(x_d)\right)\\
    &=
    \sum_{j=0}^{2^{\nu_d}}\partial_{\bsx'}^\Balpha \iptarget(\bsx',\chi_{2^{\nu_{d}},j})\ell_{2^{\nu_{d}},j}(x_{d})
    -
    \sum_{j=0}^{2^{\nu_d-1}}\partial_{\bsx'}^\Balpha \iptarget(\bsx',\chi_{2^{\nu_{d}-1},j})\ell_{2^{\nu_d-1},j}(x_d)\nonumber\\
    &=\tilde\Delta_{\nu_d}^{x_d}[\partial_{\bsx'}^\Balpha \iptarget](\bsx',x_d).
  \end{align*}
  Thus,
  \begin{align*}
    \norm[{C^0([0,1])}]{\tilde\Delta_\bsnu[\iptarget]}
    &=\sup_{\bsx'\in[0,1]^{d-1},x_d\in[0,1]}|\tilde\Delta_{\bsnu'}^{\bsx'}\tilde\Delta_{\nu_d}^{x_d}[\iptarget](\bsx',x_d)|\nonumber\\
    &\le C_{d-1}\prod_{i=1}^{d-1}(1+\nu_i)2^{-k\nu_i}
      \sup_{x_d\in[0,1]}\norm[{C^{k,\mix}([0,1]^{d-1})}]{\tilde\Delta_{\nu_d}^{x_d}[\iptarget(\cdot,x_d)]}\nonumber\\
    & = C_{d-1}\prod_{i=1}^{d-1}(1+\nu_i)2^{-k\nu_i}\sup_{x_d\in[0,1]}
      \sup_{\bsx'\in [0,1]^{d-1}}
      \sup_{\Balpha\in \{0,\dots,k\}^{d-1}}
      |\partial^\Balpha_{\bsx'}\tilde\Delta_{\nu_d}^{x_d}[\iptarget(\cdot,x_d)]|\nonumber\\
    &=C_{d-1}\prod_{i=1}^{d-1}(1+\nu_i)2^{-k\nu_i}\sup_{x_d\in[0,1]}
      \sup_{\bsx'\in [0,1]^{d-1}}
      \sup_{\Balpha\in \{0,\dots,k\}^{d-1}}
      |\tilde\Delta_{\nu_d}^{x_d}[\partial^\Balpha_{\bsx'}\iptarget(\cdot,x_d)]|\nonumber\\
    &\le C_{d-1}C_1 \prod_{i=1}^{d}(1+\nu_i)2^{-k\nu_i}
      \sup_{\bsx\in [0,1]^{d}}
      \sup_{\Balpha\in \{0,\dots,k\}^{d}}
      |\partial^\Balpha_{\bsx'}\iptarget(\cdot,x_d)|.
  \end{align*}
  This concludes the proof.
\end{proof}

We can now show the following result, which immediately implies \Cref{thm:ip_error_c_k_mix} as we discuss
subsequently.

\begin{proposition}\label{prop:mix2}
  There exists $C=C(d,k)$ such that for all $\iptarget\in \ckmix$, and all $\ell \in \N$ with $\Lambda_\ell
  \dfn \set{\bsnu\in\N_0^d}{\norm[1]{\bsnu} \le \ell}$ it holds
    \begin{equation}\label{eq:sgriderrro}
      \norm[C^0(\uid)]{\iptarget-\tilde I_{\Lambda_\ell}[\iptarget]}
      \le C \norm[\ckmix]{\iptarget} (1+\ell)^d 2^{-k \ell}.
    \end{equation}
  Moreover, computing $\tilde I_{\Lambda_\ell}[\iptarget]$ requires the evaluation of $\iptarget$ at most
  $O((1+\ell)^{d-1}2^\ell)$ points.
\end{proposition}

\begin{proof}
  \textbf{Step 1: Convergence of $\tilde{I}_{\Lambda_\ell}[f] \to f$ for $\ell \to \infty$.}
  Applying \Cref{lemma:diffbound} to~\eqref{eq:SGoptilde},
    \begin{equation*}
      \sum_{\bsnu\in\N_0^d}\norm[C^0(\uid)]{\tilde\Delta_\bsnu[\iptarget]}
      \le C \norm[\ckmix]{\iptarget}\sum_{\bsnu\in\N_0^d}\prod_{i=1}^d(1+\nu_i)2^{-k\nu_i}
      < \infty.
    \end{equation*}
  Next observe that
    \begin{align*}
      \sum_{\bsnu \in \{0,1,\dots,\ell\}^d} \tilde\Delta_\bsnu
      &=\sum_{\bsnu \in \{0,1,\dots,\ell\}^d} \otimes(I_{2^{\nu_i}}-I_{2^{\nu_i-1}})\\
      &=\sum_{\nu_1=0}^{\ell}(I_{2^{\nu_i}}-I_{2^{\nu_i}-1})\otimes
        \dots \otimes \sum_{\nu_d=0}^{\ell}(I_{2^{\nu_i}}-I_{2^{\nu_i}-1})
      =I_{2^\ell}\otimes\cdots\otimes I_{2^\ell}.
    \end{align*}
  Hence, by \Cref{prop:intmain}
    \begin{equation}\label{eq:ip_limit_I_Lambda_ell}
      \lim_{\ell\to\infty}\sum_{\bsnu \in \{0,1,\dots,\ell\}^d} \tilde\Delta_\bsnu[\iptarget]
      = \lim_{\ell\to\infty} I_{2^\ell}\otimes\cdots\otimes I_{2^\ell}[\iptarget] = \iptarget.
    \end{equation}
  By the absolute convergence of $\sum_{\bsnu \in \{0,1,\dots,\ell\}^d} \tilde\Delta_\bsnu[\iptarget]$ to an
  element of $C^0(\uid)$ we conclude that
    \begin{equation*}
      \sum_{\bsnu\in\N_0^d} \tilde\Delta_\bsnu[\iptarget] = \iptarget
    \end{equation*}
  with absolute convergence in $C^0(\uid)$, and in particular $\lim_{\ell\to\infty}
  \tilde{I}_{\Lambda_\ell}[f]=f$.

  \textbf{Step 2: Error bound.}
  By~\eqref{eq:ip_limit_I_Lambda_ell} and \Cref{lemma:diffbound} it follows
    \begin{align}\label{eq:f-tIcjnerror}
      \norm[C^0(\uid)]{\iptarget-\tilde I_{\Lambda_\ell}[\iptarget]}
      &\le\sum_{\set{\bsnu\in\N_0^d}{\norm[1]{\bsnu}>\ell}}\norm[C^0(\uid)]{\tilde\Delta_\bsnu[\iptarget]} \nonumber \\
      &\le C \norm[\ckmix]{\iptarget} \sum_{\set{\bsnu\in\N_0^d}{\norm[1]{\bsnu}>\ell}}\prod_{i=1}^d(1+\nu_i)2^{-k\nu_i}.
    \end{align}
  We next prove via induction that the last term can be bounded by
    \begin{equation}\label{eq:claimBnlN}
      \sum_{\set{\bsnu\in\N_0^d}{\norm[1]{\bsnu}>\ell}}\prod_{i=1}^d(1+\nu_i)2^{-k\nu_i}
      \le C_d (1+\ell)^d 2^{-k \ell}
    \end{equation}
  for a constant $C_d$ that may depend on $d$ and $k$ but not on $\ell$. Using integration by parts it follows
  for $d=1$ that
    \begin{equation}\label{eq:prop:mix2:induction}
      \sum_{r>\ell}(1+r) 2^{-kr}
      \le \int_{\ell}^\infty (1+r) 2^{-kr}\dd r
      =   \frac{1+\ell}{\log(2) k} 2^{-k \ell} - \int_{\ell}^\infty 2^{-k r}\dd r
      \le C_1 (1+\ell) 2^{-k \ell}.
    \end{equation}
  For $d>1$, the sum over all $\bsnu$ with $\norm[1]{\bsnu}>\ell$ can be split up as follows:
    \begin{align*}
      \sum_{\set{\bsnu\in\N_0^d}{\norm[1]{\bsnu}>\ell}}\prod_{i=1}^d(1+\nu_i)2^{-k\nu_i}
      &= \underbrace{\sum_{\nu_d=0}^\ell(1+\nu_d)2^{-k\nu_d}
         \sum_{\set{\bsnu'\in\N_0^{d-1}}{\norm[1]{\bsnu'}>\ell-\nu_d}}
         \prod_{i=1}^{d-1}(1+\nu_i')2^{-k\nu_i'}}_{\dfnn\ (*)}\\
      &\qquad + \underbrace{\sum_{\nu_d=\ell+1}^\infty(1+\nu_d)2^{-k\nu_d}\sum_{\bsnu'\in\N_0^{d-1}}
                \prod_{i=1}^{d-1}(1+\nu_i')2^{-k\nu_i'}}_{\dfnn\ (**)}.
    \end{align*}
  Using the induction assumption we find that
    \begin{align*}
      (*) &\le \sum_{\nu_d=0}^\ell(1+\nu_d)2^{-k\nu_d} C_{d-1} (1+\ell-\nu_d)^{d-1} 2^{-k(\ell-\nu_d)}\\
      &= C_{d-1} 2^{-k\ell} \sum_{\nu_d=0}^\ell(1+\nu_d) (1+\ell-\nu_d)^{d-1} \le C_{d-1} 2^{-k\ell}(1+\ell)^d.
    \end{align*}
  With the observation that $\sum_{r\in\N}(1+r)2^{-k r} = \frac{4^k}{(2^k-1)^2} \le 4$
  and~\eqref{eq:prop:mix2:induction} we can bound the second term by
    \begin{align*}
      (**) &\le C_d (1+\ell) 2^{-k \ell}
    \end{align*}
  where $C_d$ is independent of $\ell$. This shows~\eqref{eq:claimBnlN} and together
  with~\eqref{eq:f-tIcjnerror}, we find that~\eqref{eq:sgriderrro} holds.

  \textbf{Step 3: Number of target evaluations.}
  Computing $I_{\Lambda_\ell}$ requires evaluating $\iptarget$ at the interpolation points of all tensorized
  interpolants $\otimes_{i=1}^d I_{2^{\nu_i}}$ with $\bsnu \in \N_0^d$ such that $\norm[1]{\bsnu} \le \ell$.
  The number of these points sums up to
    \begin{equation}\label{eq:nrpointsNs}
      \sum_{\set{\bsnu\in\N_0^d}{\norm[1]{\bsnu}\le\ell}}
      \prod_{i=1}^d(1+2^{\nu_i})\le C_d (1+\ell)^{d-1} 2^\ell,
    \end{equation}
  where we show the last inequality again by induction over $d$. Here $C_d$ is again a constant independent
  of $\ell$, but possibly different from the constant above. For $d=1$
    \begin{equation*}
      \sum_{\nu_1=0}^\ell (1+2^{\nu_1}) = (\ell+1)+2^{\ell+1}-1\le C_1 2^\ell
    \end{equation*}
  with $C_1 \dfn 3$. For the induction step assume~\eqref{eq:nrpointsNs} holds for $d-1\ge 1$. Then
    \begin{align*}
      \sum_{\set{\bsnu\in\N_0^d}{\norm[1]{\bsnu}\le\ell}}
      \prod_{i=1}^d(1+2^{\nu_i})
      &= \sum_{\nu_d=0}^\ell(1+2^{\nu_d})
        \sum_{\set{\bsnu'\in\N_0^{d-1}}{\norm[1]{\bsnu'}\le \ell-\nu_d}} \prod_{i=1}^{d-1}(1+2^{\nu_i'})\\
      &\le C_{d-1}\sum_{\nu_d=0}^\ell(1+2^{\nu_d})(\ell-\nu_d)^{d-2}2^{\ell-\nu_d}\\
      &\le C_{d-1} 2^\ell \sum_{\nu_d=0}^\ell (\ell-\nu_d)^{d-2}\\
      &\le C_{d-1} 2^\ell (1+\ell)^{d-1}.
    \end{align*}
\end{proof}

\Cref{thm:ip_error_c_k_mix} now follows by observing that $\tilde I_{\Lambda_\ell}$ in \Cref{prop:mix2} can be
written as an operator $I_{\Lambda_m}$ with $m = 2^\ell$ as in~\eqref{eq:ILambda}, see for example the
discussion in~\cite[Section 2.3]{MR4113052} or~\cite[Section~1.3]{zech2018sparse}.

\bibliographystyle{abbrv} \bibliography{main}
\end{document}